\documentclass[a4paper,fleqn]{cas-sc}
 
 \def\tsc#1{\csdef{#1}{\textsc{\lowercase{#1}}\xspace}}
 \tsc{WGM}
 \tsc{QE}
 
\usepackage{amsmath}
\usepackage{amssymb}
 \usepackage{amsthm}
 
\usepackage{graphicx}
\usepackage{diagbox} 
\usepackage{placeins}
\usepackage{subcaption}
\usepackage{float}

\renewenvironment{proof}[1][\proofname]
{\par\noindent\textbf{#1. }\rmfamily}
{\qed\bigskip\par}

\newtheorem{theorem}{Theorem}
\newtheorem{proposition}{Proposition}
\newtheorem{lemma}{Lemma}

\theoremstyle{remark}
\newtheorem{remark}{Remark}

\numberwithin{equation}{section}

\usepackage{xcolor}

\DeclareMathOperator{\interior}{int}
\DeclareMathOperator{\diag}{diag}

\begin{document}
 \let\WriteBookmarks\relax
 \def\floatpagepagefraction{1}
 \def\textpagefraction{.001}
\shorttitle{A Nonlinear CTH-Based Adaptive Cruise Controller With Safety and String Stability Guarantees}    

\shortauthors{D. Theodosis, A. Samii, N. Bekiaris-Liberis}  

\title [mode = title]{A Nonlinear CTH-Based Adaptive Cruise Controller With Safety and String Stability Guarantees}  



%

\author[1]{Dionysios Theodosis}[orcid=0000-0001-8171-0554]

\cormark[1]


\ead{dtheodp@central.ntua.gr}


\credit{Concptualization, Formal Analysis, Writing-original draft}

\affiliation[1]{organization={Department of Mathematics, National Technical University of Athens},
            addressline={Zografou Campus}, 
            city={Athens},
            postcode={15780}, 
            country={Greece}}

\author[2]{Amirhossein Samii}


\ead{asamii@tuc.gr}


\credit{Investigation, Validation}

\affiliation[2]{organization={Department of Electrical and Computer Engineering, Technical University of Crete},
            addressline={University Campus, Akrotiri}, 
            city={Chania},
            postcode={73100},  
            country={Greece}}

\author[2]{Nikolaos Bekiaris-Liberis}[orcid=0000-0002-4223-2681]


\ead{nlimperis@tuc.gr}


\credit{Investigation, Writing - review \& editing}

\affiliation[2]{organization={Department of Electrical and Computer Engineering, Technical University of Crete},
            addressline={University Campus, Akrotiri}, 
            city={Chania},
            postcode={73100},  
            country={Greece}}

\cortext[1]{Corresponding author}

\fntext[1]{Funded by the European Union (ERC, C-NORA, 101088147). Views and opinions expressed are however those of the authors only and do not necessarily reflect those of the European Union or the European Research Council Executive Agency. Neither the European Union nor the granting authority can be held responsible for them.}


\begin{abstract}
This paper studies the problem of longitudinal control, with safety and string stability guarantees, in vehicle platoons. A nonlinear adaptive cruise controller based on the Constant Time Headway (CTH) policy is designed for platoons of heterogeneous vehicles, that incorporates a speed-dependent nonlinear gain guaranteeing bounded vehicle speeds and a saturation-type nonlinearity to guarantee bounded (or non-aggressive) accelerations, while preserving  collision-avoidance properties of time-headway-based spacing policies. By introducing a suitable spacing-error variable, explicit nonlinear error dynamics are obtained, allowing a direct analysis of the closed-loop system. It is shown that the collision-free region is positively invariant and that the proposed controller guarantees convergence of the spacing errors. Furthermore, nonlinear string-stability properties are established directly from the closed-loop dynamics without resorting to linearization or frequency-domain arguments. In particular, sufficient conditions are derived for $L_p$ string stability, $1\le p\le\infty$, with respect to disturbances, external inputs acting on the leading vehicle, and perturbations of the initial conditions. 
\end{abstract}




\maketitle

\section{Introduction}

Adaptive Cruise Control (ACC) systems constitute one of the most widely deployed driver-assistance technologies in modern vehicles. By regulating speed  through on-board sensing of inter-vehicle distances and speeds, such systems can improve traffic flow performance, fuel efficiency, and road safety. An ACC-equipped vehicle maintains a desired cruising speed in the absence of preceding vehicles, while automatically switches to spacing control when a slower vehicle is detected ahead. The success of ACC systems has motivated extensive research on vehicle platoons, where multiple vehicles travel cooperatively while maintaining desired inter-vehicle spacings (see for instance \cite{Bekiaris1}, \cite{Karafyllis2023}, \cite{Monteil2017}, \cite{Orosz},  \cite{Yang}, \cite{Zheng2016} and references therein).

A fundamental component in the design of ACC systems is the choice of the spacing policy, namely, the rule that specifies the desired spacing of successive vehicles. Different spacing policies have been proposed in the literature, including constant spacing, variable spacing, and constant time-headway policies (see for instance  \cite{Haan}, \cite{Ioannou1993},   \cite{Lunze2019}, \cite{Rajamani2012}, \cite{Santhanakrishnan2003}, \cite{Wijnbergen2020}). One of the most widely studied spacing policies is the CTH policy (see for instance \cite{Ampountolas}, \cite{Bekiaris1}, \cite{Ioannou1993}, \cite{Lunze2019}, \cite{Rajamani2012}, \cite{Shen}), according to which the desired spacing between successive vehicles is proportional to the speed of the following vehicle. The CTH policy has received considerable attention due to its favorable performance characteristics, its ability to accommodate varying traffic conditions, and its suitability for platoon-control applications.

String stability is widely regarded as one of the most important performance requirements for ACC-equipped vehicle platoons. A platoon is string stable if disturbances generated by one vehicle do not amplify as they propagate through the vehicle string. Despite its intuitive interpretation, a variety of string-stability notions have been proposed in the literature, leading to several distinct definitions and characterizations, see for instance \cite{Besselink2017},  \cite{Feng2019}, \cite{Karafyllis2026}, \cite{Montanino}, \cite{Ploeg2014}, \cite{Swaroop1996}, \cite{Zhou}. Among the various notions proposed in the literature, the $L^p$ based framework introduced in \cite{Ploeg2014} provides a unified treatment of both linear and nonlinear platoons. In addition to disturbance attenuation along the string, it captures the effects of external inputs acting on the leading vehicle, as well as perturbations in the initial conditions. Beyond model-specific analyses, the recent work \cite{Karafyllis2026} has focused on deriving general sufficient conditions for string stability. In particular, trajectory-based and Lyapunov-based criteria have been proposed for homogeneous strings of nonlinear interconnected systems, yielding verifiable $L^p$ string-stability conditions through small-gain arguments.

In addition to stability and string stability, collision avoidance is an essential requirement for autonomous vehicle platoons. A variety of safety-critical control methodologies have been proposed in the literature, including nonlinear feedback designs with Follow-the-Leader and Bidirectional architecture (\cite{Karafyllis2023}, \cite{Karafyllis2025}), prescribed-performance control (\cite{Verginis2018}), funnel control (\cite{Berger}), MPC (\cite{Wang}), and Control Barrier Function (CBF)-based designs (\cite{Alan}, \cite{Ames2017}, \cite{Hamdipoor}, \cite{He2018}, \cite{Molnar}, \cite{Zhao}). In particular, CBFs have emerged as a powerful framework for enforcing safety constraints through forward invariance arguments. Nevertheless, most existing works primarily focus on safety guarantees, whereas stability and string-stability properties are often investigated through local linearization, transfer-function analysis, or numerical validation. Consequently, rigorous nonlinear string-stability analyses for collision-free platoons remain relatively limited.

In this paper, we propose a nonlinear controller using the CTH policy for platoons of heterogeneous vehicles, which incorporates a nonlinear speed-dependent gain that naturally enforces prescribed speed constraints and a saturation-type nonlinearity to guarantee bounded (or non-aggressive) accelerations, while preserving collision avoidance properties. A key feature of the proposed controller is that the spacing-policy error satisfies a scalar nonlinear differential equation whose dynamics are independent of the preceding vehicle (as it is the case with the original, linear CTH controller for second-order vehicle dynamics, see \cite{Rajamani2012}). This property provides an explicit characterization of the evolution of the spacing-error dynamics and enables a direct nonlinear analysis of the closed-loop platoon. Safety in terms of collision avoidance and bounded speeds is shown by establishing invariance of appropriate and practically relevant sets. In addition, equilibrium properties are characterized and global stability (with respect to the safe sets) results are established. Furthermore, explicit solutions' estimates are derived quantifying the effect of disturbance propagation along the platoon, leading to rigorous $L^p$ string-stability results for $p\ge1$. In contrast to classical frequency-domain analyses, the obtained results are established directly for the nonlinear closed-loop dynamics and do not rely on local linearization arguments. The proposed methodology is also extended to a third-order vehicle dynamics model through a dynamic realization of the nominal acceleration command. Essentially, the realization is based on a one-step backstepping and, as a result, it leads to a control law of Cooperative Adaptive Cruise Control (CACC)-type, since the acceleration of the preceding vehicle is employed. By introducing an additional acceleration state and an exponentially stable realization error, the resulting augmented third-order controller, constructed on the basis of the third-order vehicle dynamics model, preserves the essential structural properties of the second-order design.

The main contributions of this work can be thus summarized as follows:
\begin{enumerate}
\item A complete analysis of linear CTH-based controllers in terms of set invariance  and string stability is presented.

\item A nonlinear modification of linear CTH-based controllers is proposed for longitudinal platoon control with bounded vehicle speeds and smoother acceleration profiles achieved through a saturation-type nonlinearity.

\item Collision avoidance is established through the positive invariance establishment of the collision-free region, and explicit spacing-error dynamics are obtained.

\item Stability properties of the resulting nonlinear platoon are analyzed without resorting to linearization.

\item Rigorous $L^p$ string-stability estimates for $p\ge1$ are derived directly from the nonlinear closed-loop dynamics.

\item The controller design is extended to a third-order vehicle model through a dynamic realization that preserves the principal properties of the nominal second-order controller.
\end{enumerate}

The remainder of the paper is organized as follows. Section \ref{sec:linear} introduces the platoon model and the properties of widely studied linear CTH-based controllers. Section \ref{sec:nonlinear} presents the nonlinear CTH-based controller and establishes  invariance  of the safety sets for the closed-loop system. Section \ref{sec:ss} is devoted to the analysis of string stability and Section \ref{sec:ls} studies the stability properties of the proposed controller. Section \ref{sec:sims} provides numerical simulation results, including a realistic scenario in which the trajectory of the leading vehicle is obtained from  {the NGSIM (see \cite{Montanino2}) and OpenACC datasets (see \cite{OpenACC})}. The proofs of all results are given in Section \ref{sec:proofs}. Finally, Section \ref{sec:conc} contains concluding remarks and directions for future research.\pagebreak

\noindent \textbf{Notation.} Throughout this paper, we adopt the following notation.

\noindent $*$ ${\mathbb R}_{+} :=[0,+\infty )$ denotes the set of non-negative real numbers.\vspace{0.4em}

\noindent $*$ By  $\diag\{a_1,a_2,\ldots,a_n\}$ we denote the $n\times n$ diagonal matrix with $a_1, a_2,\ldots, a_n \in \mathbb{R}$ on its diagonal.\vspace{0.4em}

\noindent $*$ With $x^{-} $ we denote the negative part of $x\in {\mathbb R}$ , i.e., $x^{-} =\max \left\{-x,0\right\}$, and with $x^{+} $ we denote the positive part of $x\in {\mathbb R}$ , i.e., $x^{+} =\max \left\{x,0\right\}$.\vspace{0.4em}

\noindent $*$ By $|x|$ we denote both the Euclidean norm of a vector $x\in {\mathbb R}^{n} $ and the absolute value of a scalar $x\in {\mathbb R}$.\vspace{0.4em}

\noindent $*$ By $1_n$ we denote the 1-vector of size  $n$, i.e., $1_n=(1,1,\ldots,1)\in\mathbb{R}^n$. \vspace{0.4em}

\noindent $*$ For a set $S\subseteq {\mathbb R}^{n} $, $\bar{S}$ denotes the closure of $S$, $\partial D$ the boundary of $S$, and $\interior(S)$ the interior of $S$.

\noindent $*$ Let $D\subseteq {\mathbb R}^{n} $ be an open set and let $S\subseteq {\mathbb R}^{n} $ be a set that satisfies $D\subseteq S\subseteq \bar{D}$. By $C^{0} (S ; \Omega )$, we denote the class of continuous functions on $S$, which take values in $\Omega \subseteq {\mathbb R}^{m} $. By $C^{k} (S ;  \Omega )$, where $k\ge 1$ is an integer, we denote the class of functions on $S\subseteq {\mathbb R}^{n} $, which take values in $\Omega \subseteq {\mathbb R}^{m} $ and have continuous derivatives of order $k$. In other words, the functions of class $C^{k} (S ; \Omega )$ are the functions which have continuous derivatives of order $k$ in $D=\interior(S)$ that can be continued continuously to all points in $\partial D\cap S$.  When $\Omega ={\mathbb R}$ then we write $C^{0} (S)$ or $C^{k} (S)$.\vspace{0.4em}

\noindent $*$ By $L^{p} $ with $p\ge 1$ we denote the equivalence class of measurable functions $f:{\mathbb R}_{+} \to {\mathbb R}^{n} $ for which $\left\| f\right\| _{[0,t],p} =\left(\int _{0}^{t}\left|f(x)\right|^{p} dx \right)^{1/p} <+\infty $. $L^{\infty } $ denotes the equivalence class of measurable functions $f:{\mathbb R}_{+} \to {\mathbb R}^{n} $ for which $\left\| f\right\| _{[0,t],\infty } ={\mathop{\textrm{ess}\sup }\limits_{x\in [0,t)}} \left(\left|f(x)\right|\right)<+\infty .$ By $L^{\infty } \left(A;\Omega \right)$ we denote the equivalence class of measurable functions $f:A\to \Omega $ for which $\left\| f\right\| _{\infty } ={\mathop{\sup }\limits_{x\in A}} \left(\left|f(x)\right|\right)<+\infty $ where ${\mathop{\sup }\limits_{x\in A}} \left(\left|f(x)\right|\right)$ is the essential supremum. By $L_{\textrm{loc}}^{\infty } \left({\mathbb R}_{+} ;\Omega \right)$ we denote the equivalence class of measurable functions $f:{\mathbb R}_{+} \to \Omega $ with $f\in L^{\infty } \left(\left[0,T\right];\Omega \right)$ for every $T>0$.\vspace{0.4em}

\section{Safe Invariant Sets of the Linear CTH Controller}\label{sec:linear}

A typical model that describes the evolution of a platoon of $n$ identical vehicles consist of the following set of ODEs (see, e.g. \cite{Ampountolas}, \cite{Lunze2019}, \cite{Ploeg2014}) 
\begin{equation} \label{GrindEQ__2_1_} 
\begin{aligned} 
&{\dot{s}_{i} (t)=v_{i-1} (t)-v_{i} (t)} \\
& {\dot{v}_{i} (t)=u_{i} (t)} \end{aligned},\quad i=1,\ldots ,n ,
\end{equation} 
where $s_{i} :=x_{i-1}  -x_{i}  $ is the back-to-back distance of the $i$-th vehicle from the preceding $(i-1)$ vehicle, $v_{i}  $ is the speed of $i$-th vehicle and $u_{i} $ the acceleration (control input) of the $i$-th vehicle. With $v_{0} $ we denote the speed of the leading vehicle which is an external input. 

 A very common and highly studied ACC law is the linear CTH controller (see \cite{Ampountolas}, \cite{Ioannou1993}, \cite{Rajamani2012})  
\begin{equation} \label{GrindEQ__2_2_} 
u_i = k_{p,i} (s_i-r-h_iv_i)+ k_{v,i}(v_{i-1}-v_i),  \;\; i=1,\ldots,n,
\end{equation} 
where $k_{p,i}$, $ k_{v,i}>0$ are controller gains, $h_i$ is the time-headway, and $r_i$ is the standstill spacing. We further denote by $\ell_i $ the physical collision threshold (vehicle length) with $\ell_i<r_i$ and by $v_{\max } $ the speed limit. The closed-loop system \eqref{GrindEQ__2_1_}, \eqref{GrindEQ__2_2_} corresponds to a platoon of heterogeneous vehicles where each vehicle has its own controller gains, time-headway and standstill spacing.
Moreover, when the leader is moving with constant speed $v_{0} (t)\equiv v^{*} >0$ the platoon under \eqref{GrindEQ__2_2_}  admits the equilibrium 
\begin{equation} \label{GrindEQ__2_3_} 
v_i=v^*, \;\;\; s_{i}^{*} =r_i+h_iv^{*} ,\quad i=1,\ldots,n,
\end{equation} 
which is globally asymptotically stable provided that $k_{p,i}$, $ k_{v,i}>0$.

For the  CTH  policy, the controller gains must satisfy appropriate conditions to guarantee string stability (see for instance \cite{Ampountolas}, \cite{Monteil2017}, \cite{Ploeg2014}). In particular, the necessary and sufficient condition for $L^2$ string stability  (with respect to speed errors propagation) is 
\begin{equation}\label{eq:ss_cond}
h^2_i k_{p,i} + 2 h_i  k_{v,i} \ge 2,
\end{equation}
whereas $L^{\infty}$ string stability  (with respect to speed errors propagation) is guaranteed if
\begin{equation}\label{lin:ss}
(h_ik_{p,i}+ k_{v,i})^2\ge4k_{p,i}.
\end{equation}
Both conditions for $L^2$ and $L^{\infty}$ string stability  (with respect to speed errors propagation), relate the time-headway $h$ to the controller gains $ k_{v,i}$ and $k_{p,i}$.

In the context of vehicle platoons, physically meaningful solutions should preserve collision-free movement and non-negative vehicle speeds for all times. Therefore, an important problem is to identify appropriate conditions on the initial states and the closed-loop dynamics under which these properties are guaranteed. For notational convenience, we define
\begin{equation} \label{GrindEQ__2_14_} 
s=\left(s_{1} ,\ldots,s_{n} \right)\in {\mathbb R}^{n} , v=\left(v_{1} ,\ldots,v_{n} \right)\in {\mathbb R}^{n}  .
\end{equation} 
We have the following theorem that shows invariance of solutions in a particular set.

\begin{theorem} \label{thm:1}
Let $v_{0}\in C^0(\mathbb{R}^+;\mathbb{R}^+)$ and define the set
\begin{equation} \label{GrindEQ__2_15_} 
K:=\left\{(s,v)\in {\mathbb R}^{2n} :v_{i} \ge 0,\; s_{i} \ge r_i+c_iv_{i} ,\;i=1,\ldots,n \right\} ,
\end{equation} 
where $c_i>0$, $i=1,\ldots,n$ are constants that satisfy
\begin{equation}\label{GrindEQ__2_16_}
\begin{aligned}
&1-c_i  k_{v,i} \ge 0,\\
&h_i k_{v,i}\ge1,\\
&c_i k_{v,i}-1+c_i k_{p,i} (h_i-c_i) \ge 0.
\end{aligned}
\end{equation}
Then, the solution of \eqref{GrindEQ__2_1_}, \eqref{GrindEQ__2_2_} with $\left(s(0),v(0)\right)\in K$ satisfies $(s(t),v(t))\in K$, for all $t\ge 0$.
\end{theorem}

Theorem \ref{thm:1} provides sufficient conditions under which vehicles do not collide with each other while their speeds remain non-negative. Although many physically meaningful initial conditions outside $K$ (for instance, $s_i(0)\ge r_i$) also lead to collision-free motion, such behavior is no longer guaranteed a priori and may depend on the initial spacing and speed configuration, as well as the leader's motion.  The set $K$ is the positively invariant set for the linear CTH controller and is based on the constant time-headway policy: the condition $s_{i} \ge r_i+c_iv_{i} $ with $c$ satisfying \eqref{GrindEQ__2_16_}, is exactly the quantity needed to prove speed positivity, while the additional assumption $r_i>\ell_i$ ensures that every state in $K$ is strictly collision-free (recall that $r_i$ is the standstill inter-vehicle distance and $\ell_i$ is the vehicle length). Theorem \ref{thm:1} also extends the corresponding result in \cite{Lunze2019} which uses the assumption that $s_{i} (0)=r+hv_{i} (0)$, $v_{i} (0)\ge 0$, for $i=1,\ldots,n$.

Theorem \ref{thm:1} can be proved using Nagumo's Theorem (see \cite{Aubin}). Here we provide an alternative proof via the function $W(t)=\frac{1}{2}\sum_{i=1}^n\left( ((v_i(t))^-)^2+ ((s_i(t)-r_i-c_iv_i(t))^-)^2 \right)$ which measures the violation of the inequalities defining $K$ in \eqref{GrindEQ__2_15_}. A similar approach will also be used for the nonlinear controller presented in the following section. The proof of Theorem \ref{thm:1} can be found in Section \ref{sec:proofs}.

It should be highlighted that the safety condition \(h_i k_{v,i}\ge 1\) in \eqref{GrindEQ__2_16_} implies both the $L^2$ and $L^{\infty}$ string stability conditions \eqref{eq:ss_cond} and \eqref{lin:ss} (provided that $k_{p,i}>0$). Indeed, for $L^{2}$ string stability, since \(k_{p,i}>0\) and \(h_i>0\), the second inequality in \eqref{GrindEQ__2_16_}  gives
\begin{equation}\label{L2:lin}
h_i^2 k_{p,i} + 2 h_i  k_{v,i}
\ge h_i^2 k_{p,i} + 2
> 2.
\end{equation}
Similarly, for $L^{\infty}$ string stability, using the  arithmetic and geometric means inequality $x+y\ge2\sqrt{xy}$, $x,y\ge0$, and the condition  \(h_i k_{v,i}\ge 1\) in \eqref{GrindEQ__2_16_} we have
\begin{equation}\label{Linfty:lin}
h_ik_{p,i}+ k_{v,i}\ge2\sqrt{h_ik_{p,i} k_{v,i}}\ge2\sqrt{k_{p,i}},
\end{equation}
which implies that condition \eqref{lin:ss} holds.
Hence, any controller gains satisfying the safety conditions automatically
satisfy the $L^2$ and $L^{\infty}$ string stability requirements \eqref{eq:ss_cond} and \eqref{lin:ss}, respectively. Notice that the converse is not always true (take for instance,  $h_i=1$, $k_{p,i}=2$, and $ k_{v,i}=0.1$).

The parameters $c_i$ are  auxiliary quantities used to define the invariant set $K$ and its values depend on the gains of the controller via conditions \eqref{GrindEQ__2_16_}. The reason is that higher values of $ k_{v,i}$ and $k_{p,i}$ may enhance the braking capabilities of the vehicle, thus enlarging its safety set. Using  \eqref{GrindEQ__2_16_}, we can deduce that the admissible values of $c_i$ that satisfy these inequalities lie  in the interval
\begin{equation}\label{c:values}
\left[ \frac{ k_{v,i}+k_{p,i} h_i -\sqrt{( k_{v,i}+k_{p,i}h_i)^2-4k_{p,i}}}{2k_{p,i}}, \frac{1}{ k_{v,i}}\right],
\end{equation}
where $ k_{v,i}+k_{p,i}h_i\ge 2\sqrt{ k_{v,i} h_i k_{p,i}}\ge 2\sqrt{k_{p,i}}$, due to the condition $h_i k_{v,i}\ge1$ and the inequality $x+y \ge 2\sqrt{xy}$ that holds for all $x,y\ge0$.
Thus, in practical applications of controller \eqref{GrindEQ__2_2_}, the values of $k_{p,i}$, $ k_{v,i}$, and $c$ should be appropriately selected to ensure both safety and string stability.

A particularly important choice is $c_i=h_i$, $i=1,\ldots,n$. In this case, condition \eqref{GrindEQ__2_16_} reduces to $ k_{v,i}=1/h_i$ and the spacing constraint of the set $K$,   $s_i\ge r_i+h_iv_i$, coincides with the CTH spacing policy. Hence, for $c_i=h_i$, and consequently for $ k_{v,i}=1/h_i$, the positively invariant set $K$ is directly characterized by the desired CTH spacing requirement. Moreover, when $ k_{v,i}=1/h_i$, condition \eqref{GrindEQ__2_16_}, reduces to 
\[
\begin{aligned}
&h_i\ge c_i,\\
& c_ih^{-1}_i-1+c_ik_{p,i}(h_i-c_i)=(h_i-c_i)(k_{p,i}c_i-h^{-1}_i)\ge0,
\end{aligned}
\]
which implies that the admissible values of $c_i$ must satisfy $\frac{1}{k_{p,i}h_i}\leq c_i\leq h_i$.
Therefore, $c_i=h_i$ is the unique admissible value whenever $k_{p,i} h_i^2\leq 1$, whereas for $k_{p,i}h_i^2>1$ one may select values $c_i<h_i$, enlarging the invariant set. One such example is given by
\[
 k_{v,i}=\frac{1}{h_i},\;\;k_{p,i}=\frac{1}{h_i}\left(k_i-\frac{1}{h_i}\right),
\]
where \(k_i>1/h_i\) giving the controller
\begin{equation}\label{CTH:linear}
u_i= \frac{1}{h_i}\left(k_i-\frac{1}{h_i}\right) \left(s_{i} -r_i\right)+\frac{1}{h_i} v_{i-1} -k_i v_{i},\; i=1,\ldots,n.
\end{equation}
The condition $k_ih_i>1$ immediately yields string stability as discussed earlier, as well as safety with $c_i=h_i$. If moreover, $k_ih_i>2$, then $c_i$ can be selected equal to $c_i=\frac{h_i}{k_ih_i-1} =\frac{1}{k_i-h_i^{-1} } $.  The constants $k_i$ then, can be thought as a damping coefficient highlighting the braking capabilities of each vehicle. If the braking capabilities of the vehicle are sufficiently strong, namely if $h_ik_i>2$, the set $K$ is ``larger''.  Note however, that in practice such gains are typically small, see for instance \cite{Ampountolas}, \cite{Orosz}.

The preceding analysis establishes that the linear CTH-based controller guarantees collision-free motion, non-negative speeds, and string stability, provided that the initial condition belongs to the positively invariant set $K$ and that conditions \eqref{eq:ss_cond}, \eqref{L2:lin}, and \eqref{Linfty:lin} hold. For initial conditions outside this set, neither collision avoidance nor non-negative speeds can be ensured. Moreover, the controller does not explicitly incorporate operational constraints such as road speed limits, which is an additional safety requirement, and therefore no upper bounds on the vehicle speeds can be guaranteed, (see for instance \cite{Karafyllis2023} and the simulations in Section \ref{sec:sims}). These limitations motivate the investigation of nonlinear control strategies capable of enforcing additional state constraints while preserving the desirable safety, stability, and string stability properties of the platoon.

\section{Nonlinear CTH-Based ACC Design}\label{sec:nonlinear}

Motivated by the limitations of the linear controller \eqref{GrindEQ__2_2_}  discussed in the previous section, we next introduce a nonlinear controller based on the CTH policy, designed to enforce collision-free motion, speed positivity, and bounded speeds, while retaining the string stability properties of the linear CTH-based controller. The nonlinear CTH-based controller is given by
\begin{equation} \label{GrindEQ__3_1_} 
u_{i} =F_i (s_{i} ,v_{i} ,v_{i-1} )= \frac{1}{h_i} \left(v_{i-1} -v_{i} \right)+\beta _{i} \chi (v_{i} )\sigma(H_{i})  ,\;\; i=1,\ldots,n,
\end{equation} 
where $\beta _{i} >0$ and 
\begin{align}
&H_{i} =s_{i} -r_{i} - h_i v_{i}  , \label{GrindEQ__3_2_} \\
&\chi (v_{i} )=(v_{\max } -v_{i} )\psi (v_{i} ),\label{GrindEQ__3_3_} 
\end{align} 
with $r_{i} >\ell _{i} >0$, $i=1,\ldots,n$, $\psi \in C^{1} \left( {\mathbb R}_{+} \right)$ satisfying $\psi (v)>0$ for $v>0$, and $\sigma$ being a $C^1$ function that satisfies the following properties
\begin{align}
&\sigma(0)=0, \label{s1}\\
&H\sigma(H)>0,\quad H\neq0,\label{s2}\\
&|\sigma(H)|\leq |H|, \;\; H\in\mathbb{R},\label{s3}\\
&\forall R>0, \exists m_R>0: \;\;H\sigma(H)\ge m_R H^2, \;\;|H|\leq R.\label{s4}
\end{align}

Conditions \eqref{s1}-\eqref{s4} impose mild structural assumptions on the nonlinear function $\sigma$ and characterize a broad class of monotone, sign-preserving nonlinearities. In particular, condition \eqref{s1} ensures that the equilibrium spacing policy is preserved and that zero spacing error corresponds to zero control action. Condition \eqref{s2} requires $\sigma$ to have the same sign as its argument, so that it always acts to (globally) stabilize the spacing error. Condition \eqref{s3}  is a linear growth condition, which guarantees (theoretically) well-posedness of the resulting closed-loop system and (practically) that the employed nonlinearity is less aggressive than a linear counterpart. Finally, condition \eqref{s4} is a local sector condition that provides a quadratic lower bound on  $H\sigma(H)$ and is used in the proofs to establish $L^p$ string stability and exponential decay estimates.

Here we consider the case of a string of $n$ heterogeneous vehicles in the sense that each vehicle has its own headway $h_{i} $, standstill distance $r_{i} $,  length $\ell _{i} $, and gain $\beta_i$. Examples of $\sigma$ satisfying conditions \eqref{s1}-\eqref{s4} are 
\begin{align}
&\sigma(H)=H , \quad \textrm{ and }\label{sigma:linear}\\
&\sigma(H)=\tilde{H}\tanh\left(\frac{H}{\tilde{H}}\right),\label{sigma:nonlinear}
\end{align}
for some $\tilde{H}>0$. Observe that for $\chi (v)\equiv 1$, $\sigma(H)=H$, and $\beta_{i}=\frac{1}{h_i}\left(k_i-\frac{1}{h_i}\right) $, controller \eqref{GrindEQ__3_1_} reduces to the linear controller \eqref{CTH:linear}. On the other hand, $\sigma(H)=\tilde{H}\tanh\left(\frac{H}{\tilde{H}}\right)$ provides a saturation on the spacing error $H_i$ allowing for smoother transient behavior and bounded acceleration, which may be beneficial particularly for large initial spacings.

The nonlinear modification via \eqref{GrindEQ__3_3_} allows us to consider additional properties such as bounded speeds, as well as stability and string stability (studied in the following sections). The idea behind this modification of the linear CTH-based controller \eqref{GrindEQ__2_2_} is that, since $\chi(0)=\chi(v_{\max})=0$, the spacing contribution in the feedback vanishes near the speeds $v=0$ and $v=v_{\max}$, thereby preventing the controller from driving speeds outside the admissible interval $[0,v_{\max}]$, while the remaining relaxation term continues to regulate the relative speeds of the preceding and following vehicles.

The leader speed $v_{0} \in C^0(\mathbb{R}_+;\mathbb{R}_+)$ is assumed to satisfy 
\begin{equation} \label{GrindEQ__3_4_} 
0\le v_{0} \left(t\right)\le v_{\max}, \textrm{ for all } t\ge 0. 
\end{equation} 
Define for $i=1,\ldots,n$ 
\begin{equation} \label{GrindEQ__3_5_} 
G_{i} :=s_{i} -\ell _{i} - h_{i}  v_{i}  .
\end{equation} 
We then define the following set
\begin{equation} \label{GrindEQ__3_6_} 
D:=\{ (s,v)\in {\mathbb R}^{2n} :0\le v_{i} \le v_{\max} ,\; \; G_{i} >0,\; \; i=1,\ldots ,n\} . 
\end{equation} 
The safety set $D$ above describes all states in which the speeds are bounded, i.e., $v_i \in [0,v_{\max}]$, $i=1,\ldots,n$, as well as collision-free motion, i.e., $s_i>\ell_i+ h_i v_i \ge \ell_i$, $i=1,\ldots,n$. The following result establishes invariance of the set $D$.

\begin{theorem}\label{thm:2}
For any leader speed $v_{0}  $ satisfying \eqref{GrindEQ__3_4_}, the solution of the closed-loop system \eqref{GrindEQ__2_1_}, \eqref{GrindEQ__3_1_} satisfies the following implication
\begin{equation}\label{GrindEQ__3_7_}
(s(0),v(0))\in D\Rightarrow (s(t),v(t))\in D \textrm{ for all }t\ge 0.
\end{equation}
Moreover, for any $(s(0),v(0))\in D$ the following estimates hold for all $i=1,\ldots,n$ and $t\ge 0$ 
\begin{align} 
&-\left(r_{i} -\ell _{i} \right)<\min \left\{H_{i} (0),0\right\}\le H_{i} (t)\le \max \left\{H_{i} (0),0\right\} ,\label{GrindEQ__3_8_} \\
&\ell_i +\min \left\{G_{i} (0),r_{i} -\ell _{i} \right\}\le s_{i} \left(t\right)\le \ell _{i} +  h_{i} v_{\max} +\max \left\{G_{i} \left(0\right),r_{i} -\ell _{i} \right\} .\label{GrindEQ__3_9_} 
\end{align} 

\end{theorem}

Theorem \ref{thm:2} shows that the set $D$ is positively invariant for \eqref{GrindEQ__2_1_}, \eqref{GrindEQ__3_1_}. For any physically relevant initial condition $v_{i} (0)\in [0,v_{\max } ]$ and $G_{i} (0)>0$, $i=1,\ldots,n$, and any leader speed $v_{0} (t)\in [0,v_{\max } ]$, the solutions of \eqref{GrindEQ__2_1_}, \eqref{GrindEQ__3_1_} are defined for all times and satisfies 
\[\begin{aligned} &{v_{i} (t)\in [0,v_{\max } ]}, \\ &{s_{i} (t)>\ell _{i} }, \end{aligned}\] 
for all $t\ge 0$ and $i=1,\ldots,n$ establishing that there are no collisions among vehicles, the speed of each vehicle is non-negative and respects the speed limit $v_{\max}$. It is worth noting, however, that the speed constraints $0 \leq v_i(t) \leq v_{\max}$, $i=1,\ldots,n$, may still remain valid under weaker assumptions on the initial spacing. In particular, for arbitrary initial conditions satisfying $ 0 \leq v_i(0) \leq v_{\max}$, $i=1,\ldots,n$,
the proof shows that the speeds remain in the interval $[0,v_{\max}]$ for all $t \geq 0$, independently of the values of $s_i(0)> \ell_i$. The condition $G_i(0)>0$ is required  to guarantee collision-free motion and positive invariance of the set~$D$. The proof of Theorem \ref{thm:2} can be found in Section \ref{sec:proofs}.

\begin{remark} \label{rem:1}
\textbf{(a)} In contrast to the linear case, whose safety set $K$ is unbounded with respect to the maximum speed, the nonlinear modification introduced through \eqref{GrindEQ__3_3_} allows the dynamics to evolve inside the constrained set. The key mechanism is the nonlinear factor $\chi(v_i)$ which scales the spacing contribution in the controller and satisfies $\chi(0)=\chi(v_{\max})=0$. Consequently, near the speed boundaries, the spacing term does not contribute to the vehicle's acceleration/deceleration while the relaxation term continues to  regulate the speed without exceeding the admissible bounds. This idea is not restricted to the proposed nonlinear controller and may be incorporated into any other linear controllers based on the constant time-headway policy in order to enforce bounded speeds. 

\textbf{(b)} Compared with the invariant set $K$ defined by means of \eqref{GrindEQ__2_15_} with $c_i=h_i$, the set $D$ imposes less restrictive conditions. More specifically, the set $K$ requires $s_i\ge r_i+h_i v_i$, while the set  $D$ requires only $s_i>\ell_i+ h_iv_i$. Since $\ell_i<r_i$, the nonlinear CTH-based controller allows a larger range of safe spacings. At the same time, $D$ incorporates speed constraints $0\leq v_i \leq v_{\max}$. The reason is the factor $\chi(v_i)$ which eliminates spacing contribution to the feedback law near the speed boundaries, and in particular the condition $\psi(0)=0$. If in \eqref{GrindEQ__3_3_} we assume that $\psi (v)\ge \underline{\psi }>0$ for all $v\ge 0$, and $\sigma(H)=H$, it is easy to show that the set $K$ given by \eqref{GrindEQ__2_15_} with $c_i=h_i$ is positively invariant for the closed-loop system  \eqref{GrindEQ__2_1_}, \eqref{GrindEQ__3_1_}. This can be proved by following analogous arguments to the proof of Theorem \ref{thm:2} and by showing that $H_{i} (t)\ge 0$ for $t\ge0$ and $i=1,\ldots,n$.
\end{remark}

\begin{remark}\label{remark:2}
Nonlinear controllers that establish collision-free movement as well as bounded speeds can be found in \cite{Karafyllis2023}, \cite{Karafyllis2025}. More specifically, the nonlinear controller designed in \cite{Karafyllis2023} for $n$ identical vehicles, establishes such properties in an input-dependent state space of the form
\begin{equation}\label{NL:state_space}
\mathcal{D}(v_0)=\{ (s,v)\in \mathbb{R}^{2n}: v_i \in (0,v_{\max}), s_i > \ell + k^{-1}\max\{0,v_{i}-v_{i-1}\}, i=1,\ldots,n \},
\end{equation} 
when the external input (leading vehicle) satisfies $\dot{v}_0(t)\ge -k v_0(t)$, $t\ge0$.
In contrast, the closed-loop system \eqref{GrindEQ__2_1_}, \eqref{GrindEQ__3_1_} evolves on  the fixed set $D$, defined  by \eqref{GrindEQ__3_6_}, under any leader speed satisfying \eqref{GrindEQ__3_4_}. While the set $\mathcal{D}(v_0)$ may include configurations that are not included in $D$, (the case where $v_{i-1}\ge v_i$ which gives $s_i>\ell$) the simple geometric structure of $D$ is directly tied to the CTH spacing policy for each vehicle and not in the relative speeds. This allows us to consider vehicles with different time-headway $h_i$, different controller gains $k_i$, as well as speeds at the boundary of $[0,v_{\max}]$. The controller in \cite{Karafyllis2023} will be further discussed in the numerical simulations results presented in Section \ref{sec:sims}.

Other controllers with safety guarantees can be found in \cite{Karafyllis2025} which use bidirectional sensing instead of a car-following architecture and all vehicles share the same desired speed, and in \cite{Berger} and \cite{Verginis2018} which use funnel and prescribed control methodologies obtaining practical string stability results without imposing additional speed constraints. Finally, controllers with collision avoidance have been recently designed using CBFs in \cite{Alan}, \cite{Ames2017}, \cite{Hamdipoor}, \cite{He2018}, \cite{Molnar}, \cite{Zhao}, and references therein. These approaches typically establish safety by proving forward invariance of the set $s_i-r_i-h_iv_i\ge0$ which is contained in $D$ since $r_i >\ell_i$.  Stability and string-stability properties are then usually investigated through local linearization, frequency-domain analysis, or numerical simulations. In contrast, the nonlinear CTH-based controller developed here, yields explicit nonlinear error dynamics for the spacing error $H_i$, which allow collision avoidance, stability, and string-stability properties to be studied directly in the nonlinear setting. It should be also noted that the set $D$ defined by \eqref{GrindEQ__3_6_} is characterized by the strict safety constraint $G_i>0$. Consequently, standard CBF formulations, which are typically based on forward invariance of closed sets are not directly applicable.
\end{remark} 

\begin{remark}\label{remark:jerk}
To obtain a jerk-manipulated realization, we augment the acceleration with an additional integrator state leading to the closed-loop system
\begin{equation} \label{GrindEQ__3_10_} 
\begin{aligned} &{\dot{s}_{i} =v_{i-1} -v_{i} } \\ &{\dot{v}_{i} =a_{i} } \\ &{\dot{a}_{i} =-\frac{1}{\tau_i}( a_{i} - F_{i} (s_{i} ,v_{i} ,v_{i-1} ))+\frac{d}{dt} \left(F_{i} (s_{i} ,v_{i} ,v_{i-1} )\right)} \end{aligned}, i=1,\ldots,n 
\end{equation} 
where $\tau_i>0$, $F_{i} $ is given by \eqref{GrindEQ__3_1_} and we assume that $\psi,\sigma \in C^{2} $. The additional acceleration dynamics provide a first-order (filtered) realization of the nominal nonlinear law $F_{i} $. 
We note that this realization results in a CACC-type controller as it employs measurements of acceleration of the preceding vehicle, and thus, it requires V2V communication to be available (see for instance \cite{Bekiaris1}, \cite{Wang}, \cite{Zhao}, \cite{Zheng2016}). In fact, such a realization can be also viewed as a backstepping design for a third-order vehicle dynamics model of the form $\tau_i \dot{a}_i + a_i = \bar{u}_i$, where the jerk input $\bar{u}_i$ is chosen as $\bar{u}_i = F_i + \tau_i \dot{F}_i$. 

To analyze the properties of closed-loop system \eqref{GrindEQ__3_10_}  we proceed as follows. Define the tracking error
\begin{equation} \label{GrindEQ__3_11_} 
w_{i} :=a_{i} -F_{i} (s_{i} ,v_{i} ,v_{i-1} ), \;\; i=1,\ldots,n ,
\end{equation} 
and notice that 
\begin{equation} \label{GrindEQ__3_12_} 
\dot{w}_{i} =\dot{a}_{i} -\frac{d}{dt} \left(F_{i} \right)=-\frac{1}{\tau_i} \left(a_{i} -F_{i} \right)=-\frac{1}{\tau_i} w_{i}  .
\end{equation} 
The latter implies that
\begin{equation} \label{GrindEQ__3_13_} 
w_{i} (t)=\exp \left(-\frac{1}{\tau_i} t  \right)w_{i} (0), t\ge 0 .
\end{equation} 
Define now the manifold
\[M:=\left\{\left(s,v,a\right)\in {\mathbb R}^{3n} :a_{i} =F_{i} (s_{i} ,v_{i} ,v_{i-1} ),\, i=1,\ldots,n\right\},\] 
where $s,v$ are defined by \eqref{GrindEQ__2_14_}, and $a=\left(a_{1} ,\ldots,a_{n} \right)$. Since $w_{i} (0)=0$ and \eqref{GrindEQ__3_13_} imply that $w_{i} (t)\equiv 0$, $i=1,\ldots,n$, the manifold $M$ is positively invariant. Hence, on the manifold $M$, the third-order system \eqref{GrindEQ__3_10_} reduces exactly to the original controller \eqref{GrindEQ__2_1_}, \eqref{GrindEQ__3_1_}. Therefore, the results established previously for the  closed-loop system \eqref{GrindEQ__2_1_}, \eqref{GrindEQ__3_1_}, including collision avoidance, bounded speeds, as well as the results that will be established later on string stability and equilibrium stability, remain valid for \eqref{GrindEQ__3_10_} whenever the initial condition belongs to $M$.  Thus, in this case, model \eqref{GrindEQ__3_10_} evolves in the state-space 
\begin{equation} \label{GrindEQ__3_14_} 
\tilde{D}:=\left\{(s,v,a)\in {\mathbb R}^{3n} :(s,v)\in D,a_{i} =F_{i} (s_{i} ,v_{i} ,v_{i-1} ),i=1,\ldots,n\right\} 
\end{equation} 
and the following implication holds
\begin{equation} \label{GrindEQ__3_15_} 
\left(s(0),v(0),a(0)\right)\in \tilde{D}\Rightarrow \left(s(t),v(t),a(t)\right)\in \tilde{D},\;\; t\ge 0 .
\end{equation} 
The above can also be seen by transforming system \eqref{GrindEQ__3_10_} using \eqref{GrindEQ__3_11_} to obtain
\begin{equation} \label{GrindEQ__3_16_} 
\begin{aligned} & {\dot{s}_{i} =v_{i-1} -v_{i} } \\ &{\dot{v}_{i} =F_{i} (s_{i} ,v_{i} ,v_{i-1} )+w_{i} } \\ &{\dot{w}_{i} =-\frac{1}{\tau_i}w_{i} } \end{aligned} 
\end{equation} 
which shows that the closed-loop system \eqref{GrindEQ__3_10_}  can be viewed as the nominal nonlinear closed-loop system \eqref{GrindEQ__2_1_}, \eqref{GrindEQ__3_1_} perturbed by an exponentially decaying input $w_{i} $.  {The condition $w_i(0)=0$ corresponds to the case where the platoon is already operating under the nonlinear CTH-based controller and the actuator state is initialized consistently with the controller output}. For $w_{i} (0)\ne 0$, the trajectory does not start on the invariant manifold $M$. Nevertheless, the acceleration realization error $w_{i} (t)$ decays exponentially, and therefore, the solutions of the third-order closed-loop system \eqref{GrindEQ__3_1_} converge exponentially towards the ones of the nominal dynamics \eqref{GrindEQ__2_1_}-\eqref{GrindEQ__3_1_}. Hence, the properties of \eqref{GrindEQ__3_1_} are recovered asymptotically.  Derivation of exact safety guarantees requires additional robustness conditions, for example, imposing a sufficiently small mismatch $w_{i} (0)$, or equivalently, sufficiently small $\tau_i\left|w_{i} (0)\right|$. 

We further note that, in principle, the properties of the nonlinear CTH-based controller \eqref{GrindEQ__3_1_} should be robust to acceleration dynamics of the form $\tau_i \dot{a}_i + a_i = u_i$, for sufficiently small $\tau_i$ and $w_i(0)$. This claim relies on a singular perturbation argument in combination with linear results showing that when $\tau\leq h/2$ string stability is preserved; while it is confirmed in the numerical simulation results presented in Section \ref{sec:sims}. (A detailed proof of this claim is however beyond the scope of this paper.)
 
\end{remark}

\section{String Stability Under the Nonlinear CTH-Based Controller}\label{sec:ss}

An important property of ACC-equipped vehicles is that of string stability, which is a performance requirement that guarantees improved transient behavior of the vehicles and disturbance attenuation along the vehicles in the platoon. For the linear CTH-based controller \eqref{GrindEQ__2_1_} with $ k_{v,i}=1/h_i$, the only requirement needed for string stability is that  $k_{p,i}>0$. We shall show that an analogous condition, namely, that $\beta_i>0$ is sufficient for string stability for the nonlinear CTH-based controller \eqref{GrindEQ__3_1_}. Here we use the notion of $L^{p} $ string stability initially presented in \cite{Ploeg2014} as well as the trajectory based approach of \cite{Karafyllis2026} in the proofs.

\begin{theorem}\label{thm:3}
Let    $v^{*} \in \left[0,v_{{\max}} \right]$ be given and define the set $D\subset {\mathbb R}^{2n} $ by means of \eqref{GrindEQ__3_6_}. Suppose also that the leading vehicle's speed $v_0(t)$ satisfies \eqref{GrindEQ__3_4_}. Then, {there exists a non-increasing function $\kappa:\mathbb{R}_+\to(0,+\infty)$} such that for every $p\in[1,+\infty)$, the solution of the initial value problem \eqref{GrindEQ__2_1_}, \eqref{GrindEQ__3_1_} with initial condition $\left(s(0),v(0)\right)\in D$, satisfies the following estimates, for all $t\ge 0$ and all $i=1,\ldots ,n$,  
\begin{equation} \label{GrindEQ__4_1_} 
\|v_{i}-v^{*} \|_{[0,t],p}  \le \|v_{i-1}-v^{*} \|_{[0,t],p}   + \left(\frac{h_i}{p}\right)^{1/p} \left(|v_{i}(0)-v^{*} |+ { \frac{|H_{i} (0)|}{h_i \kappa(|H_i(0)| )} }  \right) ,
\end{equation} 
\begin{equation} \label{GrindEQ__4_2_} 
\left\| v_{i} -v^{*} \right\| _{\left[0,t\right],\infty } \le \left\| v_{i-1} -v^{*} \right\| _{\left[0,t\right],\infty } +\left|v_{i} \left(0\right)-v^{*} \right|+{\frac{|H_{i} (0)|}{h_i \kappa(|H_i(0)| )}},
\end{equation} 
\begin{equation}\label{spacing:i}
\begin{aligned}
\left\|\delta_i\right\|_{\left[0,t\right],\infty }& \leq \frac{h_i }{ h_{i-1}}\left\|\delta_{i-1}\right\|_{\left[0,t\right],\infty }+  h_i  |v_i(0)-v^*|+\frac{h_i }{ h_{i-1}}|H_{i-1}(0)|+ {\left(1+\frac{1}{\kappa(|H_i(0)| )}\right)}|H_i(0)|,\quad i=2,\ldots,n,
\end{aligned}
\end{equation}
\begin{equation}\label{spacing:1}
\left\|\delta_1\right\|_{\left[0,t\right],\infty } \leq  h_1 \left\|v_0-v^*\right\|_{\left[0,t\right],\infty }+ h_1 |v_1(0)-v^*|+{\left(1+\frac{1}{\kappa(|H_1(0)| )}\right)}|H_1(0)|,
\end{equation}
where $\delta_i:=s_i-r_i-h_iv^*$, and
\begin{equation}\label{H}
 \| H_i \|_{[0,t],\infty}\leq |H_i(0)|.
\end{equation}
\end{theorem}

\begin{remark}
Theorem \ref{thm:3} establishes string stability results for the nonlinear CTH-based controller. The string stability condition for the nonlinear CTH-based controller coincides with that of the linear CTH controller \eqref{GrindEQ__2_1_} with $ k_{v,i}=1/h_i$, namely that $k_{p,i}=\beta_i>0$. This is not a consequence of linearization, but rather of the structure of the nonlinear estimates used in the proof of Theorem \ref{thm:3}. In particular, the nonlinear coupling term is bounded and enters the analysis through comparison arguments. As a result, the same   condition arises. The proof of Theorem \ref{thm:3} can be found in Section \ref{sec:proofs}.

Estimates \eqref{GrindEQ__4_1_} and \eqref{GrindEQ__4_2_} establish $L^{p} $ string stability for $p\ge1$. 
The notion of $L^{p} $ string stability encompasses the upstream disturbance attenuation, the external input of the leading vehicle, and perturbations on initial conditions. Notice that when all vehicles are initially at equilibrium position (discussed in the next section), namely when $H_{i} \equiv 0$, $v_{i} \equiv v^{*} $, $i=1,\ldots,n$, the additive perturbation terms $\frac{|H_i(0)|}{h_i\kappa(|H_i(0)| )}  + |v_{i} \left(0\right)-v^{*}| $  in \eqref{GrindEQ__4_1_} and \eqref{GrindEQ__4_2_}, respectively, vanish, resulting in the standard estimates of string stability (see \cite{Ploeg2014}). Estimate \eqref{H} shows that the spacing error $H_i$ is always bounded. This is immediate by noticing that $\dot{H}_i=-h_i\beta_i\chi(v_i)\sigma(H_i)$ which also shows that spacing errors of successive vehicles $H_i$ and $H_{i-1}$ are independent of each other. This is also an inherent property of the linear controller \eqref{CTH:linear}.  Moreover, this independence property is lost   with actuator dynamics (see Remark \ref{remark:jerk}), where the spacing errors of successive vehicles become dynamically coupled through the actuator lag. This is also the case with the linear CTH controller (see the discussion in \cite[Chapter 6]{Rajamani2012}).

Estimate \eqref{spacing:i} shows the amplification or attenuation of spacing disturbances between successive vehicles. The factor $h_i /h_{i-1}$ quantifies this amplification and arises naturally from the heterogeneous time-headway parameters $h_i$, as it is the case with linear, CTH-type controllers (see e.g. \cite{Xiao}). Indeed, the spacing error $\delta_i$ scales the speed error by the headway $h_i $. Consequently, the same speed perturbation corresponds to a larger spacing deviation for vehicles with larger time-headway (note that $\delta_i$ can be written as $\delta_i=H_i+h_i(v_i-v^*)$, with $H_i$ evolving independently of $H_{i-1}$). Thus, non-amplification of spacing disturbances is associated with the condition  $h_i /h_{i-1}\leq1$. This is exactly the case of homogeneous time-headway $h_i=h$,  $i=1,\ldots,n$, for which, in the absence of initial-condition perturbations, \eqref{spacing:i} reduces to $\left\|\delta_{i}\right\|_{\left[0,t\right],\infty }\leq\left\|\delta_{i-1}\right\|_{\left[0,t\right],\infty }$, $i=2,\ldots,n$.
 
\vspace{1em}
\end{remark}

\begin{remark}
The string stability estimates of Theorem \ref{thm:3} can be directly transferred to the third-order case of \eqref{GrindEQ__3_16_}. Indeed, for the case $w_i(0)\neq0$ discussed in Remark \ref{remark:jerk}, we have that $w_i(t)=\exp(-q_i t)w_i(0)$, and thus, $\|w_i\|_{[0,t],p}\leq \left(\frac{1}{pq_i}\right)^{1/p}|w_i(0)|$, $p\in[1,+\infty)$, which by following the same approach as in the proof of Theorem \ref{thm:3}, give the estimates
\begin{equation}  
\|v_{i}-v^{*} \|_{[0,t],p}  \le \|v_{i-1}-v^{*} \|_{[0,t],p}   + \left(\frac{h_i}{p}\right)^{1/p} \left( \frac{|H_{i} (0)|}{h_i\kappa(|H_i(0)| )} +|v_{i}(0)-v^{*} | \right) +h_i \left(\frac{1}{pq_i}\right)^{1/p}|w_i(0)|,
\end{equation} 
\begin{equation} 
\left\| v_{i} -v^{*} \right\| _{\left[0,t\right],\infty } \le \left\| v_{i-1} -v^{*} \right\| _{\left[0,t\right],\infty } +\left|v_{i} \left(0\right)-v^{*} \right|+\frac{|H_{i} (0)|}{h_{i}\kappa(|H_i(0)| ) }+\frac{|w_i(0)|}{q_i}.
\end{equation} 
\end{remark}

The string stability estimates established in Theorem \ref{thm:3} quantify the propagation of disturbances on speed and spacing along the platoon. We next complement these estimates by studying the tracking of a time-varying leader trajectory. More specifically, we investigate whether the platoon can follow the leader motion while maintaining small deviations in spacing and speed. To this end, we derive explicit bounds for the tracking errors in terms of the initial conditions and the leader acceleration. In particular, the obtained estimates show that the tracking errors remain small when the leader acceleration is small, while larger variations of the leader speed lead to proportionally larger, but still bounded, deviations.

First, we need the following lemma whose proof can be found in Section \ref{sec:proofs}:

\begin{lemma} \label{lem:1}
Let $\underline{v},\bar{v}>0$ such that $0<\underline{v}\le v_{0} (t)\le \bar{v}<v_{\max } $ for all $t\ge 0$ and let $(s(0),v(0)\in \interior\left(D\right)$. Then, there exists $\delta :=\delta \left(s(0),v(0)\right)>0$ with $\delta <\min \left\{\underline{v},v_{\max } -\bar{v}\right\}$ such that the solution of \eqref{GrindEQ__2_1_}, \eqref{GrindEQ__3_1_} exists for all times and satisfies $\delta \le v_{i} (t)\le v_{\max } -\delta $, for all $t\ge 0$ and $i=1,\ldots,n$.
\end{lemma}

Lemma \ref{lem:1} gives a stronger characterization of bounded speeds than Theorem \ref{thm:2}. When the leader's speed lies in an arbitrary compact subinterval of $(0,v_{\max})$, the speeds of all following vehicles are bounded away from zero and $v_{\max}$. The latter and definition \eqref{GrindEQ__3_3_} imply that there exists a constant $\underline{ \chi} >0$ such that $\chi(v_i(t))\ge\underline{\chi}$ for all $t\ge0$ and $i=1,\ldots,n$. This property allows us to  prove the following proposition.

\begin{proposition} \label{prop:2}
{Let $\underline{v}  ,\bar{v}  >0$ such that $0<\underline{v}  \le v_{0} (t)\le \bar{v}<v_{\max } $ for all $t\ge 0$ and let $(s(0),v(0))\in int\left(D\right)$. Assume also that $\dot{v}_{0} \in L_{\textrm{loc}}^{+\infty } \left({\mathbb R}_{+} \right)$. Then, for each $i=1,\ldots,n$ there exist constants $\delta ,C_{i}>0$ and a non-increasing function  $\kappa:\mathbb{R}_+\to(0,+\infty)$ such that }
\begin{equation} \label{GrindEQ__4_3_} 
\left|v_{i} (t)-v_{0} (t)\right|+\left|H_{i} (t)\right|\le C_{i} \exp \left(-\mu t\right)\sum _{j=1}^{i}\left(\left|v_{j} (0)-v_{0} (0)\right|+\left|H_{j} (0)\right|\right) +\Gamma _{i} \left\| \dot{v}_{0} \right\| _{[0,t],\infty } , t\ge 0 ,
\end{equation} 
where $0<\mu <\min \left\{\frac{1}{h_{i} } , h_i \beta _{i} \mathop{\min }\limits_{v\in [\delta ,v_{\max } -\delta ]} \left\{\chi (v)\right\}\kappa(|H_i(0)| )\right\}$ and $\Gamma _{i} =\sum_{j=1}^i h_j >0$.
\end{proposition}

The estimate obtained above provides a bound on the tracking errors $v_{i} \left(t\right)-v_{0} \left(t\right)$ and $H_{i} \left(t\right)=s_{i} \left(t\right)-r_{i} - h_i v_{i} \left(t\right)$ in terms of the initial deviations and the variation of the leader speed. In particular, it shows that both the speed error and the spacing policy error decay exponentially from their initial values, up to a residual term proportional to  $\| \dot{v}_{0} \| _{\left[0,t\right], \infty}$. The constants $\Gamma_i$ quantify the sensitivity of the $i$-th vehicle to variations of the leader speed and are equal to the cumulative headway $\sum_{j=1}^i h_j $. This shows that the tracking-error bound increases along the platoon. This reflects the cumulative delay induced by the constant time-headway policy.

Proposition \ref{prop:2} should not be interpreted as string stability since it concerns tracking error bounds with respect to the external input $\dot{v}_0$, rather than amplification of inter-vehicle disturbances. This result should be interpreted as a practical tracking property with respect to time-varying leader speeds. Indeed, in the absence of leader acceleration (i.e., when $v_{0} $ is constant), the estimate reduces to exponential convergence of both errors to zero. When the leader speed varies, the tracking error remains bounded, with a bound that scales with the magnitude of the leader acceleration. Proposition \ref{prop:2} will be used also in the following section to establish certain stability properties.

\section{Lyapunov Stability}\label{sec:ls}

Having established existence of solutions for the nonlinear CTH-based controller  \eqref{GrindEQ__3_1_}, invariance of the set $D$, and string stability, it remains to study the stability of the equilibria of the closed-loop system \eqref{GrindEQ__2_1_}, \eqref{GrindEQ__3_1_}.\\[1em]

\noindent \textbf{Equilibria: }First, we characterize the equilibria of the closed-loop system \eqref{GrindEQ__2_1_}, \eqref{GrindEQ__3_1_}. As in the linear CTH-based controller (and other Follow-the-Leader models), the spacing/speed equilibrium configurations correspond to the leader's constant speed $v_{0} (t)=v^{*} \in [0,v_{\max } ]$ and uniform follower behavior. At equilibrium, the closed-loop system \eqref{GrindEQ__2_1_}, \eqref{GrindEQ__3_1_} implies that all vehicles have the same constant speed
\[v_{i} =v^{*} ,i=1,\ldots,n\] 
and satisfy
\[\chi (v^{*} )\sigma(s_i-r_i-h_iv^*) =0,\, \, \, i=1,\ldots,n.\] 
Due to definition \eqref{GrindEQ__3_3_}, we distinguish the following cases:\\

\noindent \textbf{Case 1:} $v^{*} \in (0,v_{\max } )$. 

\noindent In this case, definition \eqref{GrindEQ__3_3_} gives $\chi (v^{*} )>0$, and consequently, since $\sigma(H)=0$ implies that $H=0$, by \eqref{s1}, it follows that $s_i-r_i-h_iv^* =0$ for $i=1,\ldots,n$. The latter gives that  $s_{i}^* =r_{i} + {h_{i}} v^{*} $ for $i=1,\ldots,n$. Hence, the equilibrium reduces to the singleton
\begin{equation} \label{GrindEQ__5_1_} 
w^{*} =\left(s_{1}^{*} ,\ldots,s_{n}^{*} ,v^{*} 1_{n} \right).
\end{equation} 
\textbf{Case 2: }$v^{*} =0$.

\noindent In this case, definition \eqref{GrindEQ__3_3_} and the fact that $\psi (0)=0$ give that $\chi (v^{*} )=\chi (0)\equiv 0$, and the equilibrium reduces to the condition $v_{i} =0$, with $\sigma(s_i-r_i-h_iv^*)$ arbitrary, for $i=1,\ldots,n$. However, due to Theorem~\ref{thm:2}, the equilibrium belongs to the admissible set $D$ which requires $s_{i} >\ell _{i} + {h_{i}} v_{i} $, $i=1,\ldots,n$. The latter reduces to $s_{i} >\ell _{i} $. Hence, the equilibrium set is given by
\begin{equation} \label{GrindEQ__5_2_} 
E_{0} =\left\{(s,v)\in D:v_{i} =0, i=1,\ldots,n\right\} 
\end{equation} 
which is an unbounded set, since the spacings $s_i$ are only required to  satisfy $s_i>\ell_i$. \\

\noindent \textbf{Case 3:} $v^{*} =v_{\max } $.

\noindent In this case, definition \eqref{GrindEQ__3_3_} gives that $\chi (v^{*} )=\chi (v_{\max } )\equiv 0$, and the equilibrium reduces to the conditions $v_{i} =v_{\max } $ and no restriction on $\sigma(s_i-r_i-h_iv^*)$, for every $i=1,\ldots,n$. Hence, the equilibrium set in this case is given by
\begin{equation} \label{GrindEQ__5_3_} 
E_{v_{\max } } =\left\{(s,v)\in D:v_{i} =v_{\max } ,s_{i} \ge\ell _{i} + {h_{i}} v_{\max } ,i=1,\ldots,n\right\} 
\end{equation} 
which is also an unbounded set (since the spacings satisfy $s_i>\ell_i+h_i v_{\max}$).\\

\begin{remark}
The qualitative difference between the three cases stems from the vanishing of $\chi \left(v\right)$ at $v_{\max} $ and at zero. While for $v^{*} \in (0,v_{\max} )$ the equilibrium is uniquely determined by the spacing condition $s_i^*=r_i+h_iv^*$, at $v^{*} =v_{\max} $ and $v^{*} =0$ this condition is no longer enforced, leading to a continuum of admissible equilibria. This reflects the fact that, at maximal speed, the spacing no longer influences the acceleration dynamics and the vehicles tend to equalize their speed. Completely analogous is the case when $v^{*} =0$. The qualitative difference between the equilibrium cases can be understood from the structure of the dynamics. The relative speeds term $ {\frac{1}{h_{i} }} \left(v_{i-1} -v_{i} \right)$ enforces equalization of speeds along the platoon, so that all vehicles must share a common speed at equilibrium.
\end{remark}

\begin{theorem}\label{thm:4}
Consider a platoon of $n$ vehicles described by \eqref{GrindEQ__2_1_}, \eqref{GrindEQ__3_1_} that evolves on the set $D$ defined by means of \eqref{GrindEQ__3_6_}. Assume that the leader speed is constant, $v_{0} (t)\equiv v^{*} \in [0,v_{\max } ]$. Then, one of the following holds: 

\begin{enumerate}
\item  { If $v^{*} \in (0,v_{\max } )$, then the equilibrium $w^{*} $ defined by \eqref{GrindEQ__5_1_} is globally asymptotically stable in $D$.}

\item  { If $v^{*} =0$, then for all $i=1,\ldots,n$ it holds that}
\begin{equation} \label{GrindEQ__5_4_} 
{\mathop{\lim }\limits_{t\to +\infty }} \left(v_{i} (t)\right)=0 
\end{equation} 
 
 {and there exist constants }$\bar{H}_{i} \in \left(\ell _{i} -r_{i} ,\max \left\{H_{i} (0),0\right\}\right]$ { such that}
\begin{equation} \label{GrindEQ__5_5_} 
{\mathop{\lim }\limits_{t\to +\infty }} \left(s_{i} (t)\right)=r_{i} +\bar{H}_{i}  .
\end{equation}

\item  \textit{If $v^{*} =v_{\max } $, then for all $i=1,\ldots,n$ it holds that}
\begin{equation}\label{GrindEQ__5_6_}
{\mathop{\lim }\limits_{t\to +\infty }} \left(v_{i} (t)\right)=v_{\max } 
\end{equation}
{and there exist constants }$\tilde{H}_{i} \in \left(\ell _{i} -r_{i} ,\max \left\{H_{i} (0),0\right\}\right]$ { such that}
\begin{equation} \label{GrindEQ__5_7_} 
{\mathop{\lim }\limits_{t\to +\infty }} \left(s_{i} (t)\right)=r_{i} +  {h_{i}} v_{\max } +\tilde{H}_{i} . 
\end{equation} 
\end{enumerate}
\end{theorem} 

\begin{remark}
\textbf{(a)} Theorem \ref{thm:4} shows that the asymptotic behavior of the system depends qualitatively on whether the equilibrium speed lies in the interior of the admissible interval or at its boundary. In the case $v^{*} \in (0,v_{\max} )$, the equilibrium set reduces to a singleton and global asymptotic stability is achieved. In contrast, when $v^{*} =v_{\max} $, the equilibrium set $E_{v_{{\max}} } $ is nontrivial, and the trajectories converge to this set without necessarily converging to a unique spacing equilibrium. Completely analogous is the case where $v^{*} =0$. More specifically, when $v^*=0$, the limiting spacing positions satisfy $s_i=r_i+\bar{H}_i$, where $\bar{H}_i$ depends on the initial conditions and remains bounded within the interval $\left(\ell _{i} -r_{i} ,\max \left\{H_{i} (0),0\right\}\right)$. Thus, the platoon converges to a (collision-free) stationary configuration whose spacing remains close to the nominal standstill distance $r_i$. The proof of Theorem \ref{thm:4} can be found in Section \ref{sec:proofs}.

\textbf{(b)} As discussed in Remark \ref{rem:1}(b), if in \eqref{GrindEQ__3_3_} we assume that $\psi (v)\ge \underline{\psi }>0$ for all $v\ge 0$, the set $K$ given by \eqref{GrindEQ__2_15_} with $c=h$ is positively invariant for  \eqref{GrindEQ__2_1_}, \eqref{GrindEQ__3_1_}. Then, the equilibria of the closed-loop system reduce to the following two cases: (i) $v_i=v^* \in [0,v_{\max})$, $s_i=s^* =r+hv^*$; and  (ii) the  set of equilibria described by \eqref{GrindEQ__5_3_}. In this case, it can be proved, similarly to Theorem \ref{thm:4}, that when $v^*=0$, $\lim_{t\to+\infty}(v_i(t))=v^*$ and $\lim_{t\to+\infty}(s_i(t))=r$. This result can directly be extended to the heterogeneous case.
\end{remark}

The use of LaSalle's invariance principle in the proof of Theorem \ref{thm:4} provides convergence to the equilibrium set but does not yield information on the rate of convergence. In particular, exponential convergence cannot be inferred from this argument alone. The following result complements Theorem \ref{thm:4}(1) with exponential rate of convergence under the condition that the initial conditions belong to the interior of the set $D$, i.e., \textit{$\left(s(0),v(0)\right)\in \interior (D)$.}

\begin{proposition}\label{prop:3}
{ Assume that the leader speed is constant, such that $v_{0} (t)\equiv v^{*} \in (0,v_{\max } )$ and let $s_{i}^{*} = r_i+h_iv^*$, $i=1,\ldots,n$, $s^{*} =\left(s_{1}^{*} ,\ldots,s_{n}^{*} \right)$. Then for any $\left(\tilde{s}_{0} ,\tilde{v}_{0} \right)\in \interior\left(D\right)$, the solution of the initial value problem \eqref{GrindEQ__2_1_}, \eqref{GrindEQ__3_1_} with initial condition $\left(s(0),v(0)\right)=\left(\tilde{s}_{0} ,\tilde{v}_{0} \right)$ is defined for all $t\ge 0$ and there exist constants $C=C(s(0),v(0))>0$ and $\mu =\mu (s(0),v(0))>0$ such that the following holds}
\begin{equation} \label{GrindEQ__5_8_} 
\left|\left(s(t)-s^{*} ,v(t)-v^{*} 1_{n} \right)\right|\le C\exp \left(-\mu t\right)\left|\left(s(0)-s^{*} ,v(0)-v^{*} 1_{n} \right)\right|, t\ge 0.  
\end{equation}  
\end{proposition}

The proof of Proposition \ref{prop:3} follows from Proposition \ref{prop:2} and Lemma \ref{lem:1} in the case of constant leader speed and shows that the tracking estimate reduces to exponential convergence toward the equilibrium $w^*$  defined by \eqref{GrindEQ__5_1_}.   Proposition \ref{prop:3} shows exponential convergence along trajectories corresponding to interior equilibrium speeds with the rate of convergence depending on the initial states. We can strengthen this result by considering the family of sets
\begin{equation} \label{GrindEQ__5_9_} 
W_{\bar{\delta},\bar{G}}:=\left\{\left(s,v\right)\in {\mathbb R}^{2n} :\bar{\delta }\le v_{i} \le v_{\max } -\bar{\delta },0< G_{i} \le \bar{G}, \,\,i=1,\ldots,n\right\} \subset D,
\end{equation} 
where $0<\bar{\delta }<\min \left\{v^{*} ,v_{\max } -v^{*} \right\}$,  $v^*\in(0,v_{\max})$, $\bar{G}\ge r_{i} +h_{i} v^{*} $, and $r_i>\ell_i>0$, $i=1,...,n$. Then, we can prove the following theorem which shows uniform exponential stability of the equilibrium for initial conditions in any  set $W\subset D$.
 
\begin{theorem}\label{thm:5}
Let $v^{*} \in (0,v_{\max } )$ and consider the set $W$ defined by \eqref{GrindEQ__5_9_} with arbitrary constants $0<\bar{\delta }<\min \left\{v^{*} ,v_{\max } -v^{*} \right\}$, $\bar{G}\ge r_{i} +h_{i} v^{*} $. Then for any $\left(s(0),v(0)\right)\in W$ there exist constants $\tilde{C},\tilde{\mu }>0$ (depending only on   $\bar{\delta },\bar{G}$) such that 
 \begin{equation} \label{GrindEQ__5_10_} 
 \left|\left(s(t)-s^{*} ,v(t)-v^{*} 1_{n} \right)\right|\le \tilde{C}\exp \left(-\tilde{\mu }t\right)\left|\left(s(0)-s^{*} ,v(0)-v^{*} 1_{n} \right)\right|, t\ge 0 
 \end{equation} 
where $s^{*} =\left(s_{1}^{*} ,\ldots,s_{n}^{*} \right)$.
\end{theorem}

Theorem \ref{thm:5} shows that, whenever the initial speeds are bounded away from $0$ and $v_{\max}$ and the initial spacings $G_i(0)$ are uniformly bounded above, the equilibrium  $w^*$ defined by \eqref{GrindEQ__5_1_}, is uniform exponentially stable. The constants in the estimate depend only on the prescribed bounds $\bar{\delta}$ and $\bar{G}$, which may be chosen arbitrarily.

\section{Simulations}\label{sec:sims}

In this section we provide simulation results for the nonlinear controller  \eqref{GrindEQ__3_1_}. More specifically, we first illustrate string stability results and compare the results with the linear controller  \eqref{CTH:linear}. We then use trajectories and data obtained from the NGSIM (see \cite{Montanino2}) and OpenACC datasets (see \cite{OpenACC}) for: (a) human driven leader, (b) human-driven platoons with automated leader, as well as, (c) platoons of automated vehicles, to compare the respective performances in more realistic scenarios. Finally, we illustrate the robustness of the nonlinear controller when assuming third-order vehicle dynamics and provide throughput comparisons with the controller proposed in \cite{Karafyllis2023}.

\subsection{Numerical Illustration of String Stability}

This section compares the performance of the nonlinear controller \eqref{GrindEQ__3_1_} with the linear CTH design \eqref{CTH:linear}. We consider a heterogeneous platoon of five vehicles. For the nonlinear controller \eqref{GrindEQ__3_1_} we choose 
\begin{equation}\label{psi:example}
\psi(v)=\frac{4}{v_{\max}^2}v
\end{equation}
which normalizes the function $\chi$ in \eqref{GrindEQ__3_3_}, i.e., $\chi(v)\in[0,1]$, $v\in[0,v_{\max}]$. Moreover we select the spacing saturation function \eqref{sigma:nonlinear} and set $\beta_i=\frac{1}{h_i}\left(k_i-\frac{1}{h_i}\right)$ (with $k_i>1/h_i$) which allows as to fairly compare the nonlinear controller with the linear CTH controller \eqref{CTH:linear}. We choose control gain $k_i=\frac{1.2}{h_i}$, $v_{\rm max}=20$, and $\bar{H}= \frac{v_{max}}{2}$. For simplicity we consider the case of homogeneous time-headway $h_i=0.8 \nobreakspace (s)$, standstill spacing $r_i=10 \nobreakspace (m)$, and length $\ell_i=4.5 \nobreakspace (m)$, $i=1,\ldots,4$. The initial conditions are chosen as to belong in the set $K$ given by \eqref{GrindEQ__2_15_} with $c=h$, which is invariant for the linear CTH controller, which in this case also satisfies $K\subset D$ where $D$ is given by \eqref{GrindEQ__3_6_}. More specifically, the initial conditions are given by $v_{i}(0) = 19.5 \left(\frac{m}{s} \right)$, $i=1,\ldots,4$, $v_{0}(0) \equiv 19 \left(\frac{m}{s} \right) $; $s_{1}(0) = 32.5 \nobreakspace (m)$, $s_{2}(0) = 27.1 \nobreakspace (m)$, $s_{3 }(0) = 26.1 \nobreakspace (m)$, and $s_{4 }(0) = 27.6 \nobreakspace (m)$. To illustrate the string stability of the linear and nonlinear controllers, we consider a leading vehicle that performs a braking and accelerating maneuver. The response of the vehicles using the linear controller \eqref{CTH:linear} and the nonlinear controller \eqref{GrindEQ__3_1_} with \eqref{sigma:nonlinear} is shown in Figure~\ref{Fig1}. 

This simulation scenario illustrates the string stability of both the linear and nonlinear CTH-based controllers. There are, however, certain important features that need to be highlighted. First, because of the large initial spacing errors, the linear controller generates relatively large acceleration input, causing some vehicles to temporarily exceed the  speed limit $v_{\max}$ as they attempt to close the gap with their predecessors. In contrast, the nonlinear controller employs the spacing saturation function $\sigma(H)$ defined in \eqref{sigma:nonlinear}, which limits the influence of large spacing errors on the control input. Consequently, it produces smoother acceleration profiles while maintaining the vehicles' speeds within the prescribed bounds.

\begin{figure}[pos=htbp]
\centering
 
\makebox[0.45\linewidth]{\textbf{Linear CTH \eqref{CTH:linear}}}\hfill
\makebox[0.45\linewidth]{\textbf{Nonlinear CTH \eqref{GrindEQ__3_1_}, \eqref{sigma:nonlinear}}}
 
\vspace{0.3em}

\includegraphics[width=0.45\linewidth]{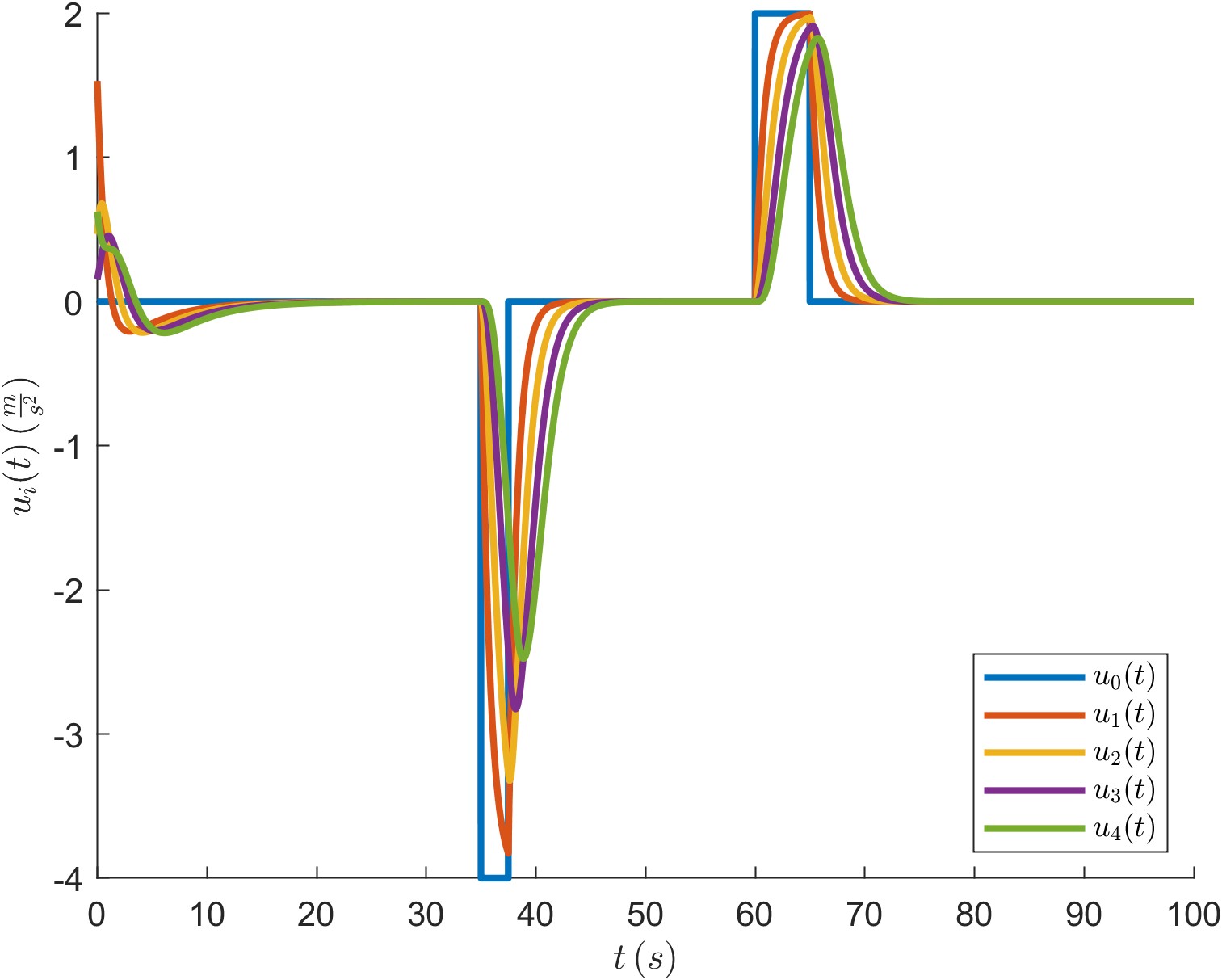}\hfil
\includegraphics[width=0.45\linewidth]{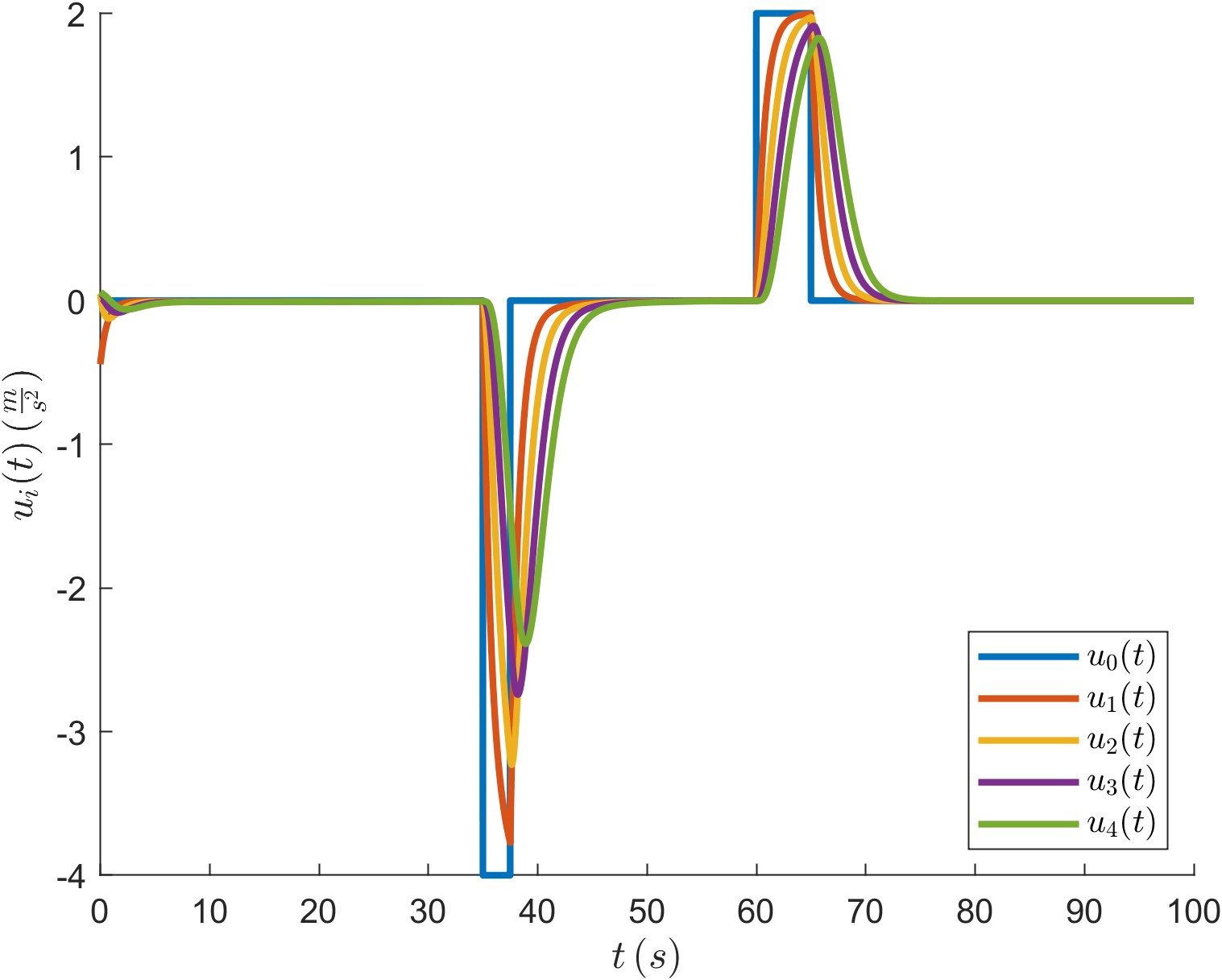}

\includegraphics[width=0.45\linewidth]{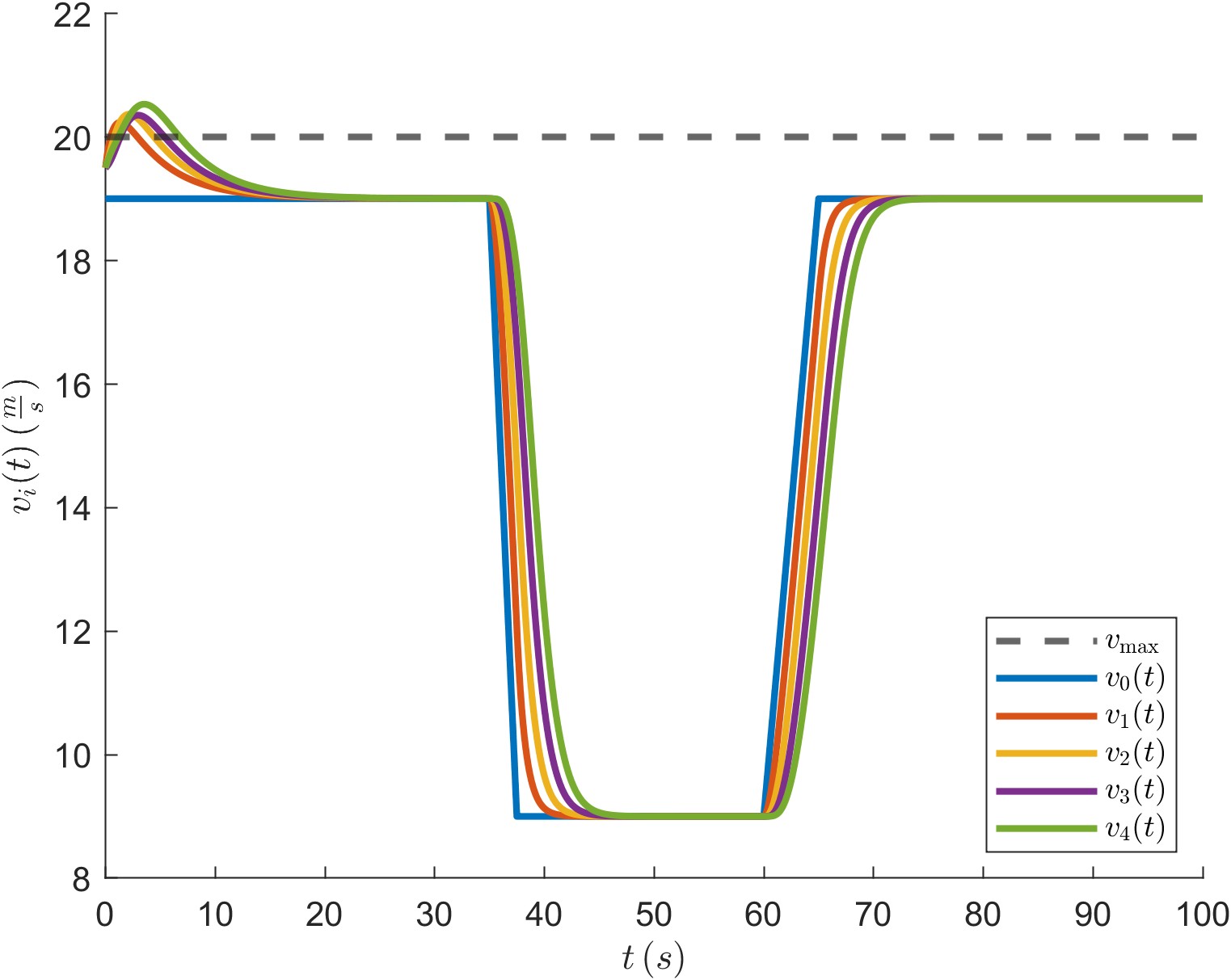}\hfil
\includegraphics[width=0.45\linewidth]{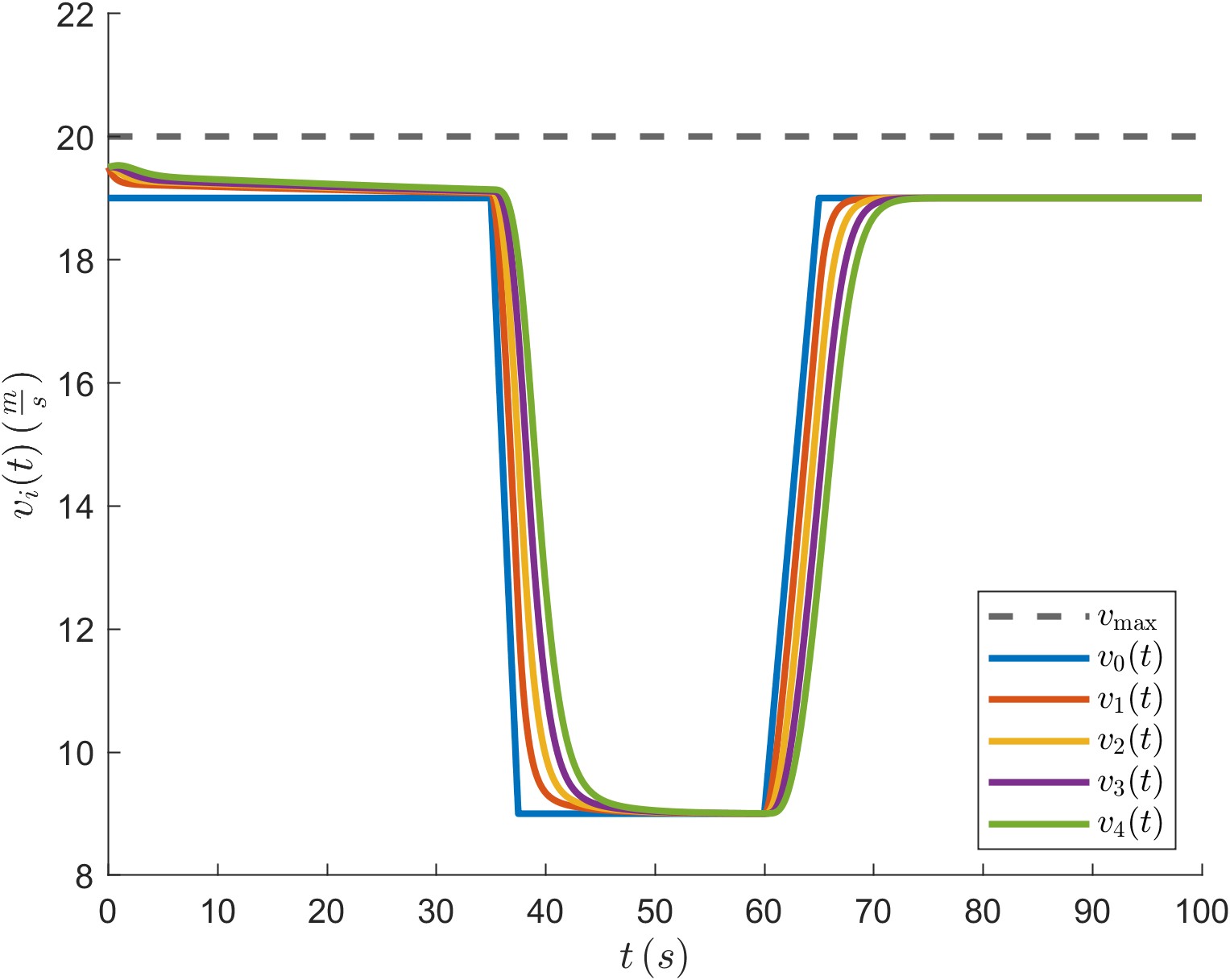}

\includegraphics[width=0.45\linewidth]{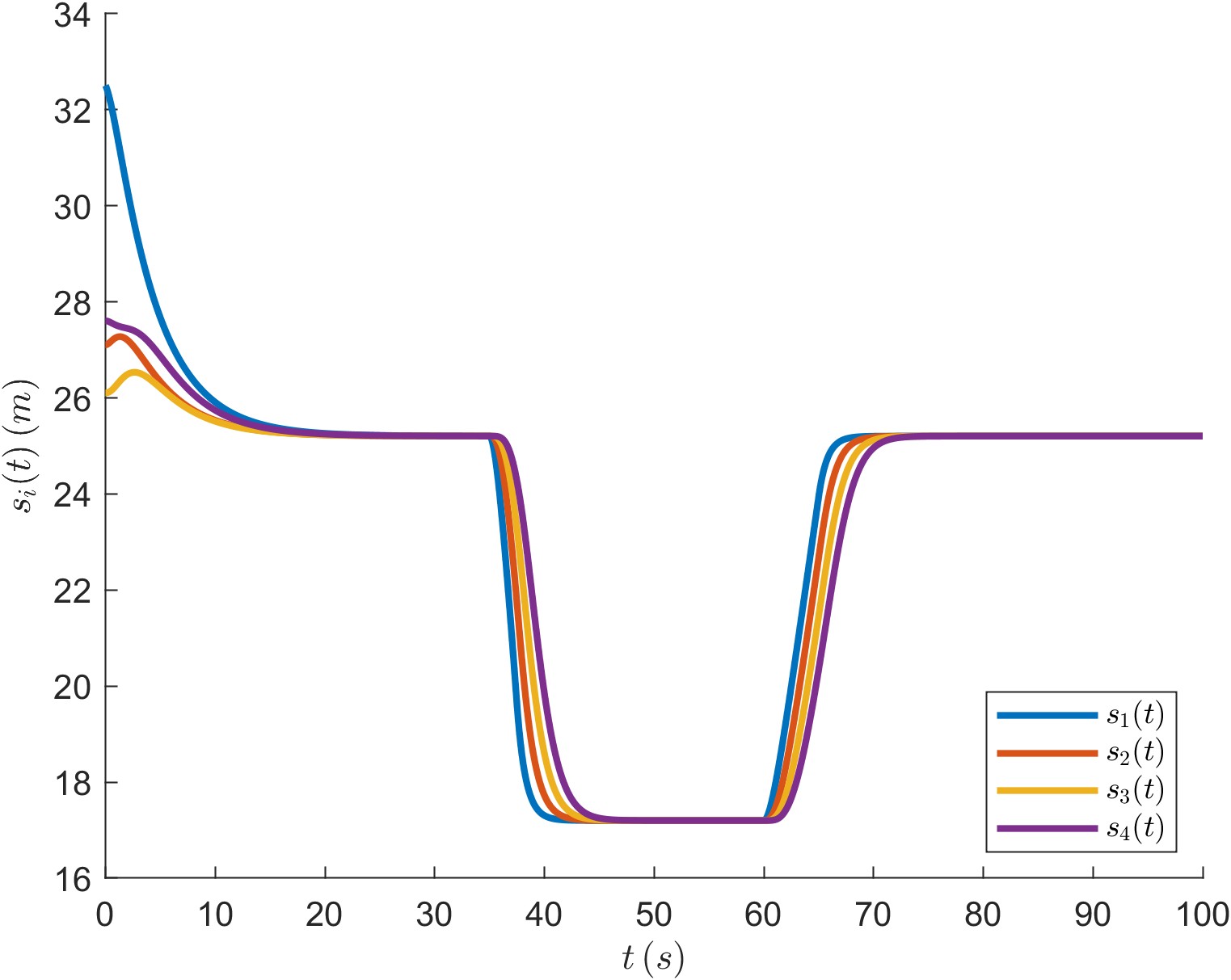}\hfil
\includegraphics[width=0.45\linewidth]{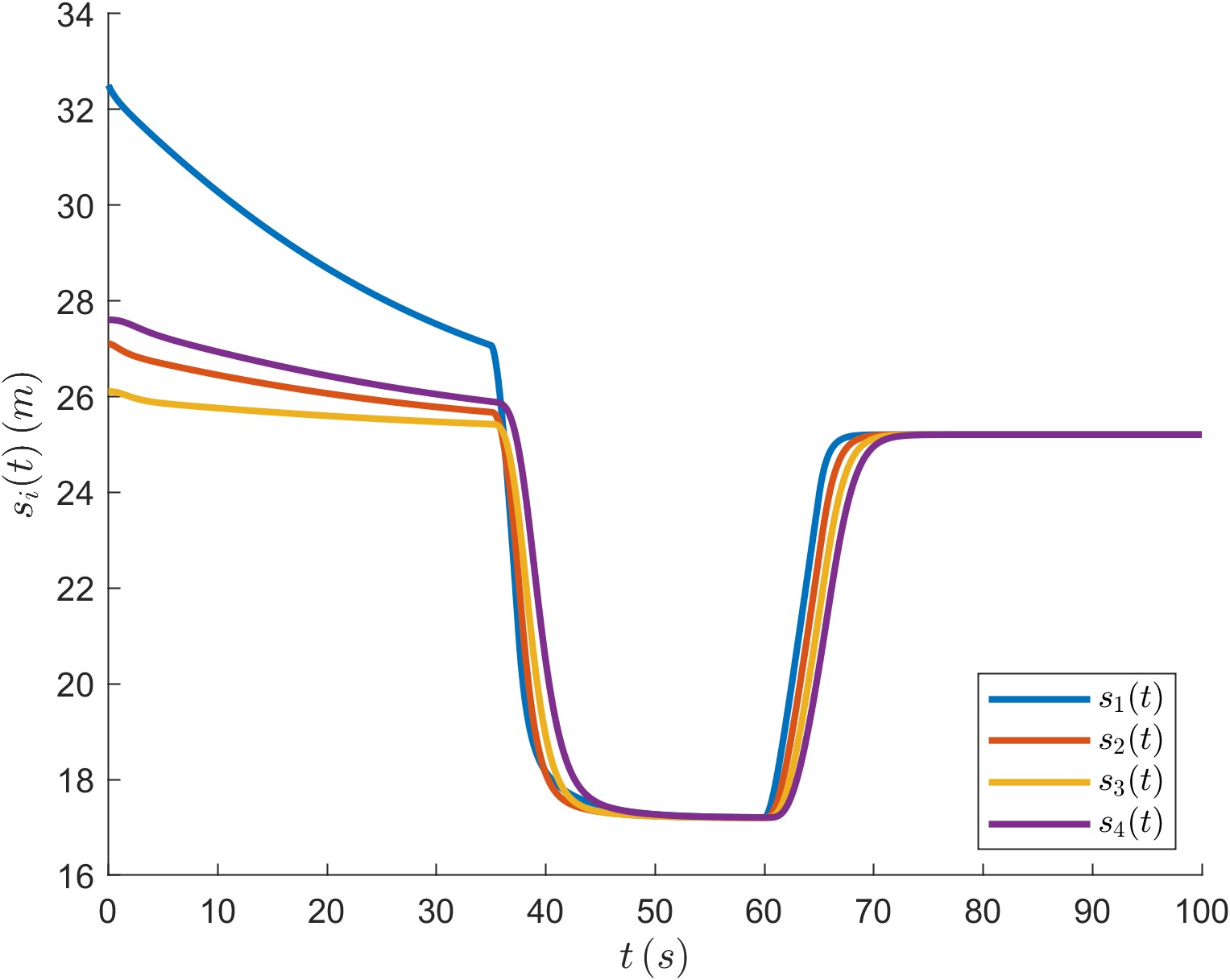}

\caption{Comparison between the linear CTH controller \eqref{CTH:linear} (left) and the nonlinear CTH controller \eqref{GrindEQ__3_1_} with \eqref{sigma:nonlinear} (right).}
\label{Fig1}
\end{figure}

\subsection{Performance Comparison with Actual Traffic Platoons Using  the NGSIM and OpenACC Datasets }
In this subsection, we compare the behavior of real platoons of human-driven vehicles and automated vehicles, extracted from the NGSIM and OpenACC datasets, with that of a platoon sharing the same leader maneuver but with automated followers controlled by the linear and nonlinear CTH-based strategies, respectively. To compare also quantitatively the linear and nonlinear strategies in terms of performance, we will use the following four practical indices (see, e.g., \cite{Akcelik}, \cite{Martinez}, \cite{MWang}):\\

\noindent  a) Safety:
 \begin{align}
 	J_{\rm safety} &= \sum_{i=1}^{n} \int_{0}^{T} \bar{J}_i \left(s_i(t), v_i(t),v_{i-1}(t)\right)\,dt, \label{s_1} \\
 	\bar{J}_i &= \begin{cases}
 		\frac{2}{s_i(t)} \left( v_{i-1}(t) - v_i(t) \right) ^2 , 
 		& {\rm {if}} \quad v_{i-1}(t) \leq v_i(t) \\
 		0, & \rm {otherwise}
 	\end{cases}. \label{s_2}
 \end{align}
 b) Comfort:
 \begin{align}
 	J_{\rm comfort}&=  \sum_{i=1}^{n} \int_{0}^{T} \dot{a}_i(t)^2\,dt. \label{C_3}           
 \end{align}
 where $a_i$ is the acceleration of vehicle $i$.

\noindent c) Fuel consumption:
 \begin{align}
 	J_{\rm fuel} &= \sum_{i=1}^{n} \int_{0}^{T} J_i \left(v_i(t),a_i(t)\right)\,dt, \label{f_1} \\
 	J_i &= \begin{cases}
 		\eta_1 + \eta_2 R_{T_i} \left(v_i(t),a_i(t)\right) v_i(t) 
 		\\+ \eta_3 v_i(t) a_i^2(t),
 		& {\rm {if}} \quad R_{T_i} > 0 \\
 		\eta_1, & {\rm {if}} \quad R_{T_i} \leq 0,
 	\end{cases} \label{f_2} \\
 	R_{T_i} &= \eta_4 + \eta_5 v_i^2(t) + \eta_6 a_i(t), \label{f_3}
 \end{align}
 where $a_i$ is the acceleration of vehicle $i$.
 
\noindent d) Tracking:
 \begin{align}
 	J_{\text{tracking}} = \sum_{i=1}^{n} \int_{0}^{T} \left( v_i(t) - v_{i-1}(t) \right)^2 \, dt .\label{f_4}
 \end{align}
In the following simulation results, the parameters of the fuel consumption indices \eqref{f_1}--\eqref{f_3} are adopted from \cite{Bek} and are given by
$\eta_1=0.666$, $\eta_2=0.0717$, $\eta_3=0.0578$, $\eta_4=0.527$, $\eta_5=0.000948$, and $\eta_6=1.68$.

\subsubsection{Human-Driven Leader Trajectory From the NGSIM Dataset}

First we compare the performance of the linear controller \eqref{CTH:linear} and the nonlinear controller \eqref{GrindEQ__3_1_} with \eqref{sigma:nonlinear} and \eqref{psi:example}, using data from the NGSIM dataset. Specifically, the trajectory of the leading vehicle is extracted from the reconstructed NGSIM dataset in \cite{Montanino2}   and corresponds to a human-driven vehicle (vehicle no. 1601 in the I-80 dataset), while the following vehicles are controlled by the linear or nonlinear ACC strategies.  We choose control gain $k_i={1.2}/{h_i}$, $\psi(v_i(t))=4\frac{v_i(t)}{v_{\rm max}^2}$ for all~$i$, $v_{\max}=20$, $\tilde{H}=  {v_{max}}/{2}$, $h_i =0.8$,  $r_i=10$, and initial conditions $v_{i}(0) \equiv 15.5 \left( m/s \right)$,   $s_{1}(0) = 26.5 \nobreakspace (m)$, $s_{2}(0) = 23.9 \nobreakspace (m)$, $s_{3}(0) = 22.9 \nobreakspace (m)$, $s_{4}(0) = 24.4 \nobreakspace (m)$.

\begin{table}[pos=h]
	\centering
	\caption{Cost values of safety, comfort, fuel consumption, and tracking for both the linear \eqref{CTH:linear} and nonlinear \eqref{GrindEQ__3_1_}, \eqref{sigma:nonlinear} strategies.}
	\scalebox{1}{\begin{tabular}{||c| c  c||} 
			\hline
			\backslashbox{Cost function}{Strategy} & Linear \eqref{CTH:linear} & Nonlinear \eqref{GrindEQ__3_1_}, \eqref{sigma:nonlinear}  \\ 
			\hline\hline
			Safety & $4.58$ &  $4.52$ \\ 
			\hline
			Comfort & $184.99$  &  $184.12$ \\ 
			\hline
			Fuel Consumption & $181.90$  &  $179.51$ \\ 
			\hline
			Tracking & $53.50$  &  $52.70$ \\ 
			\hline
	\end{tabular}}\vspace*{-\baselineskip}
	\label{table_a}
\end{table}

\begin{figure}[pos=ht]
	\begin{center}
	\makebox[0.45\linewidth]{\textbf{Linear Controller \eqref{CTH:linear}}}\hfill
	\makebox[0.45\linewidth]{\textbf{Nonlinear Controller \eqref{GrindEQ__3_1_}, \eqref{sigma:nonlinear}}}
	
	\vspace{0.4em}
	
	\includegraphics[width=0.45\linewidth]{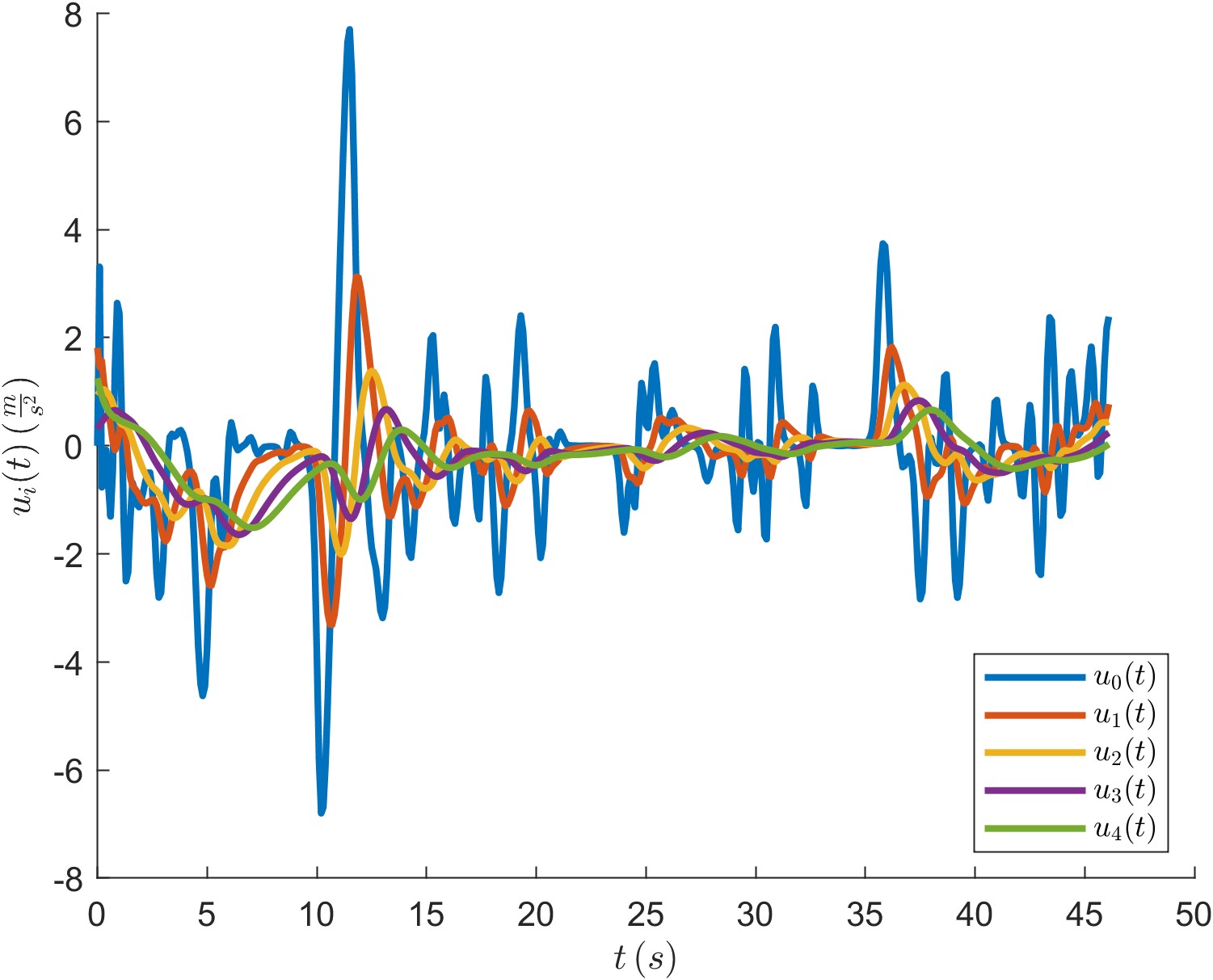}\hfil
	\includegraphics[width=0.45\linewidth]{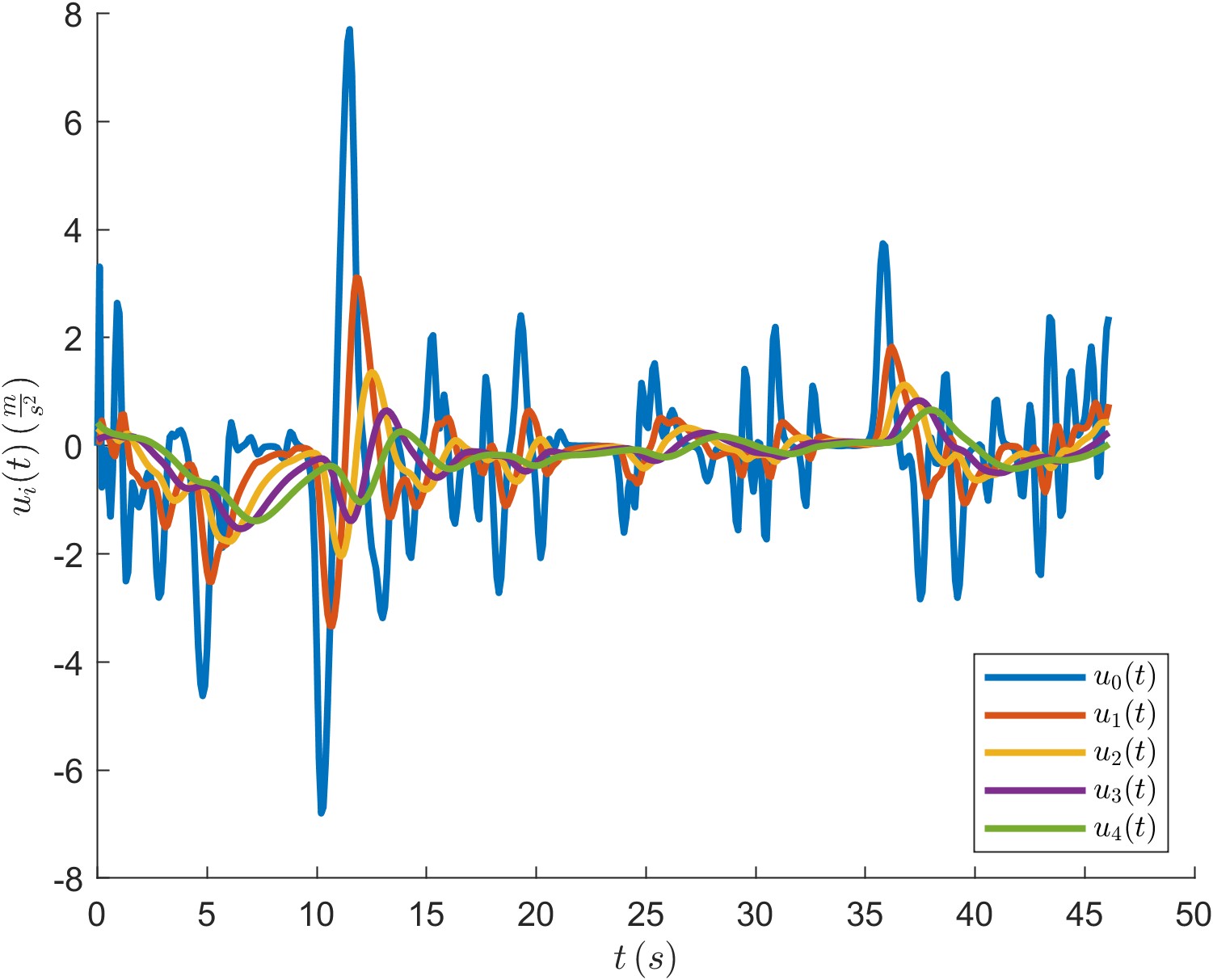} 
	
	\includegraphics[width=0.45\linewidth]{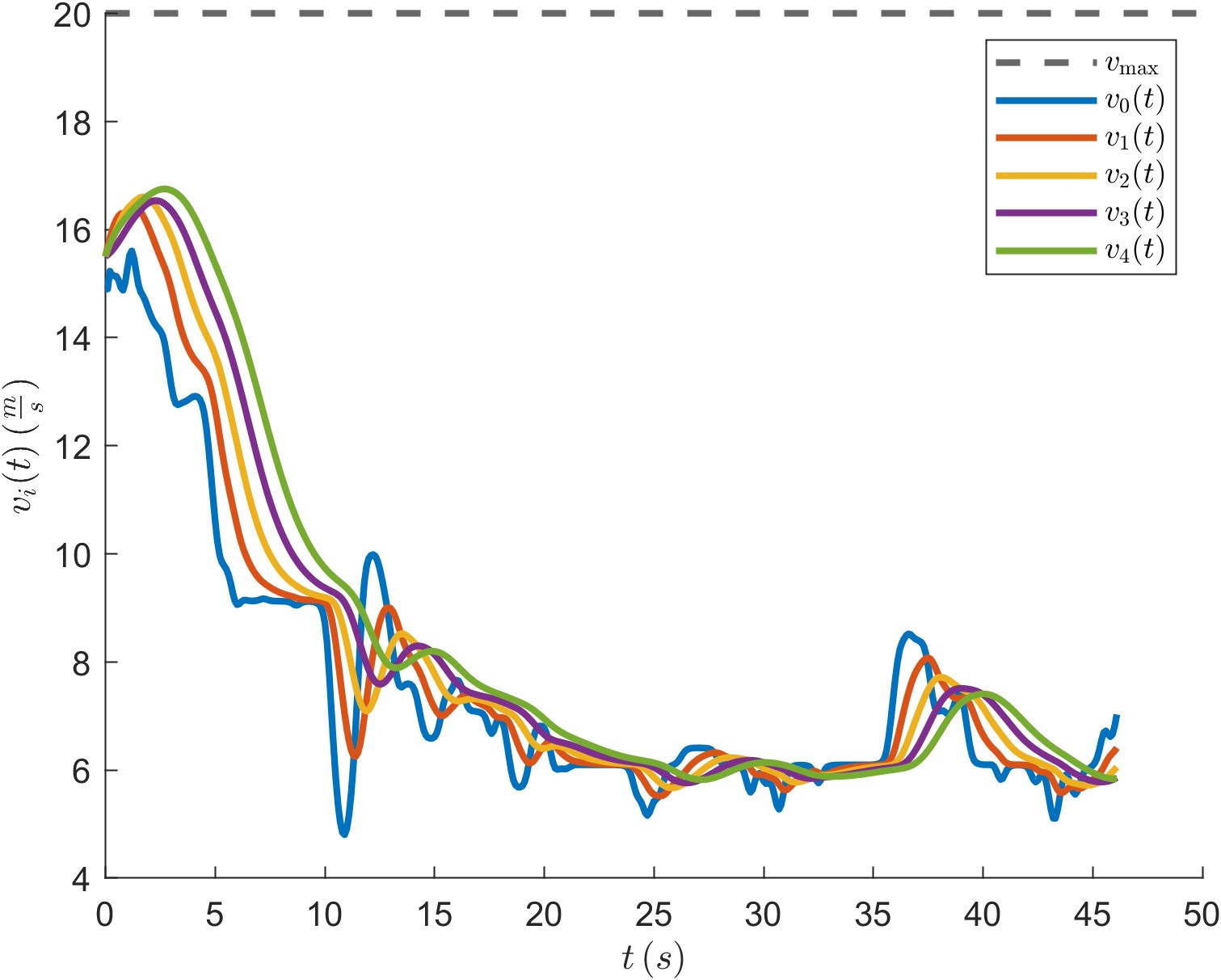}\hfil
	\includegraphics[width=0.45\linewidth]{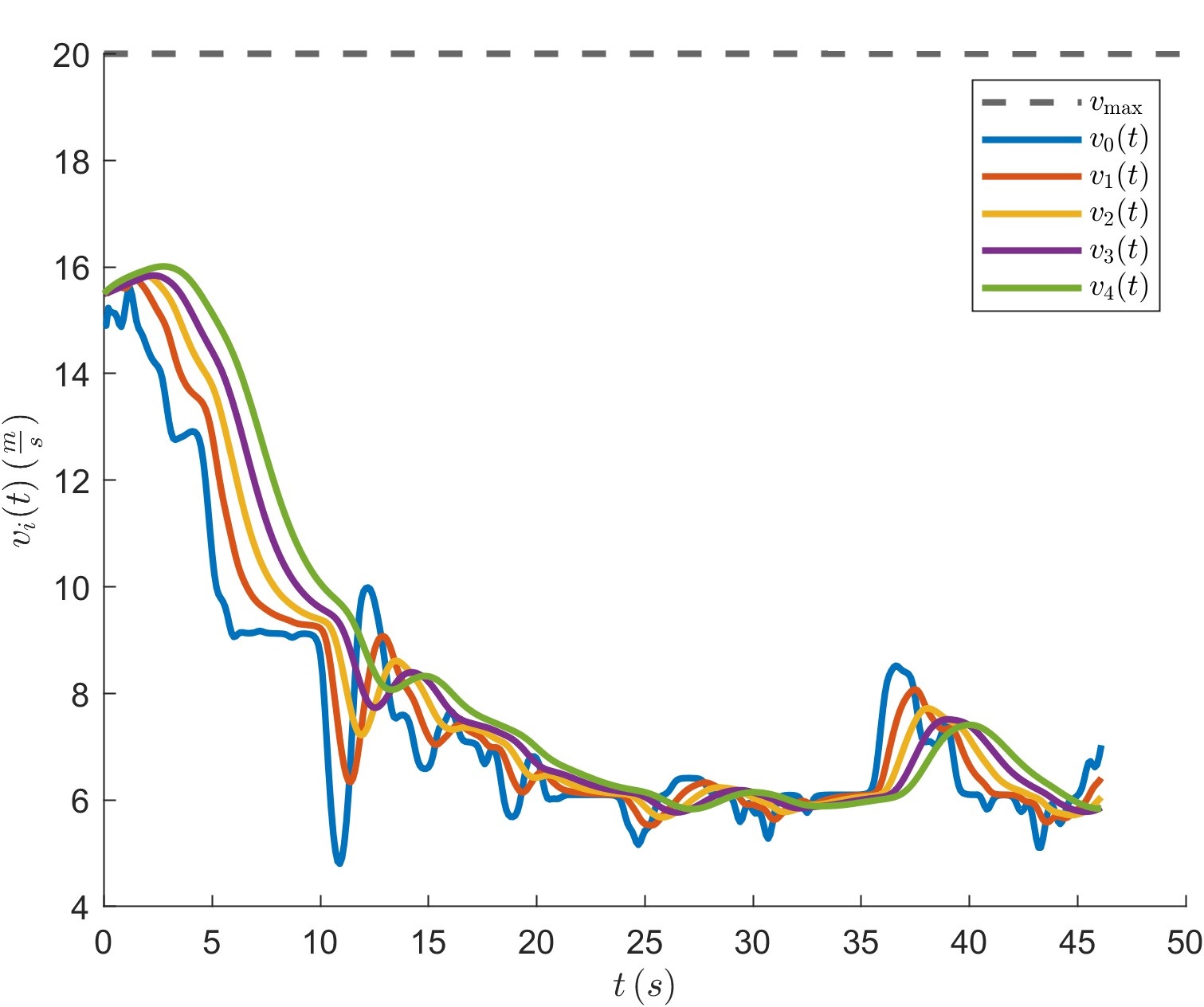} 
	
	\includegraphics[width=0.45\linewidth]{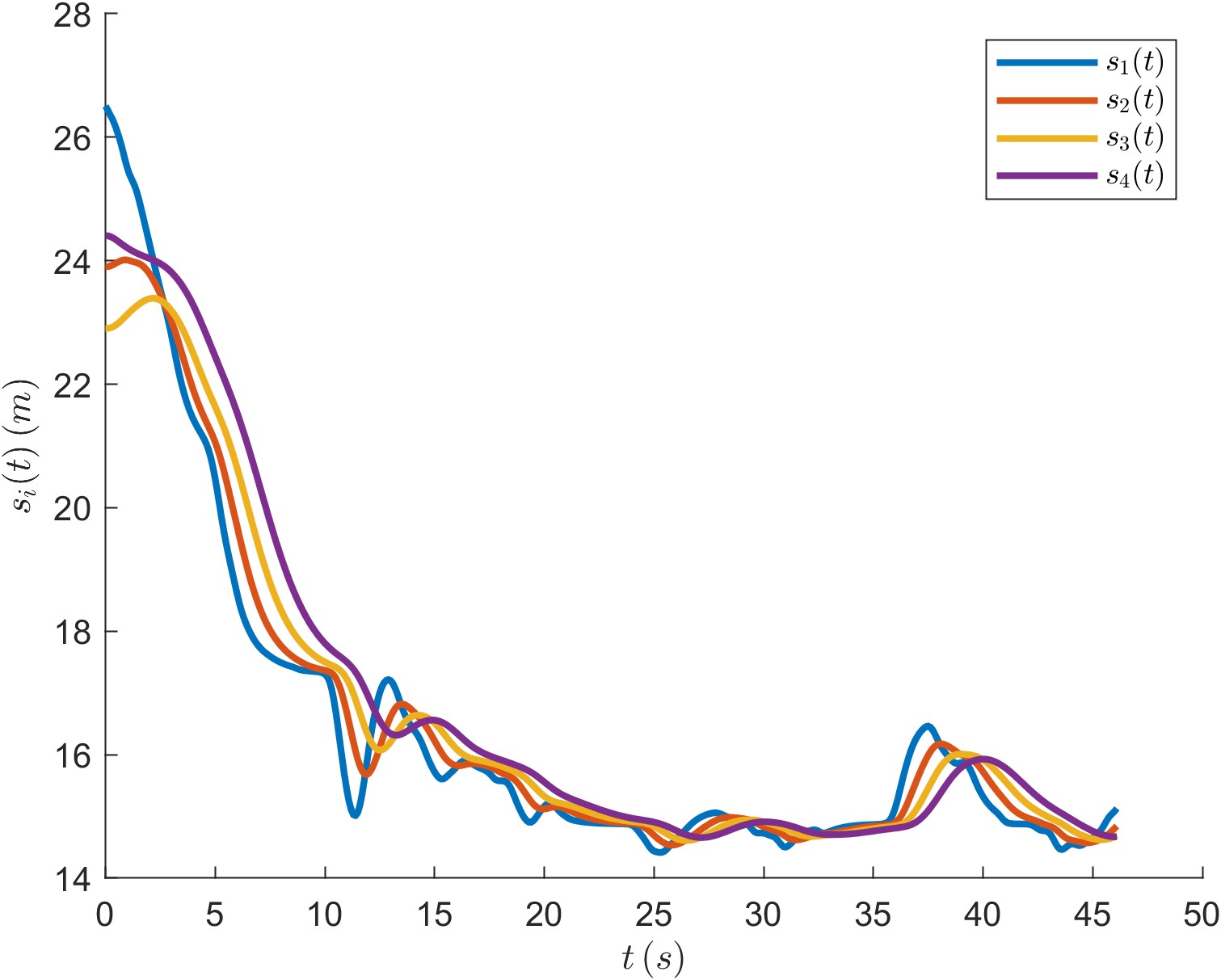}\hfil
	\includegraphics[width=0.45\linewidth]{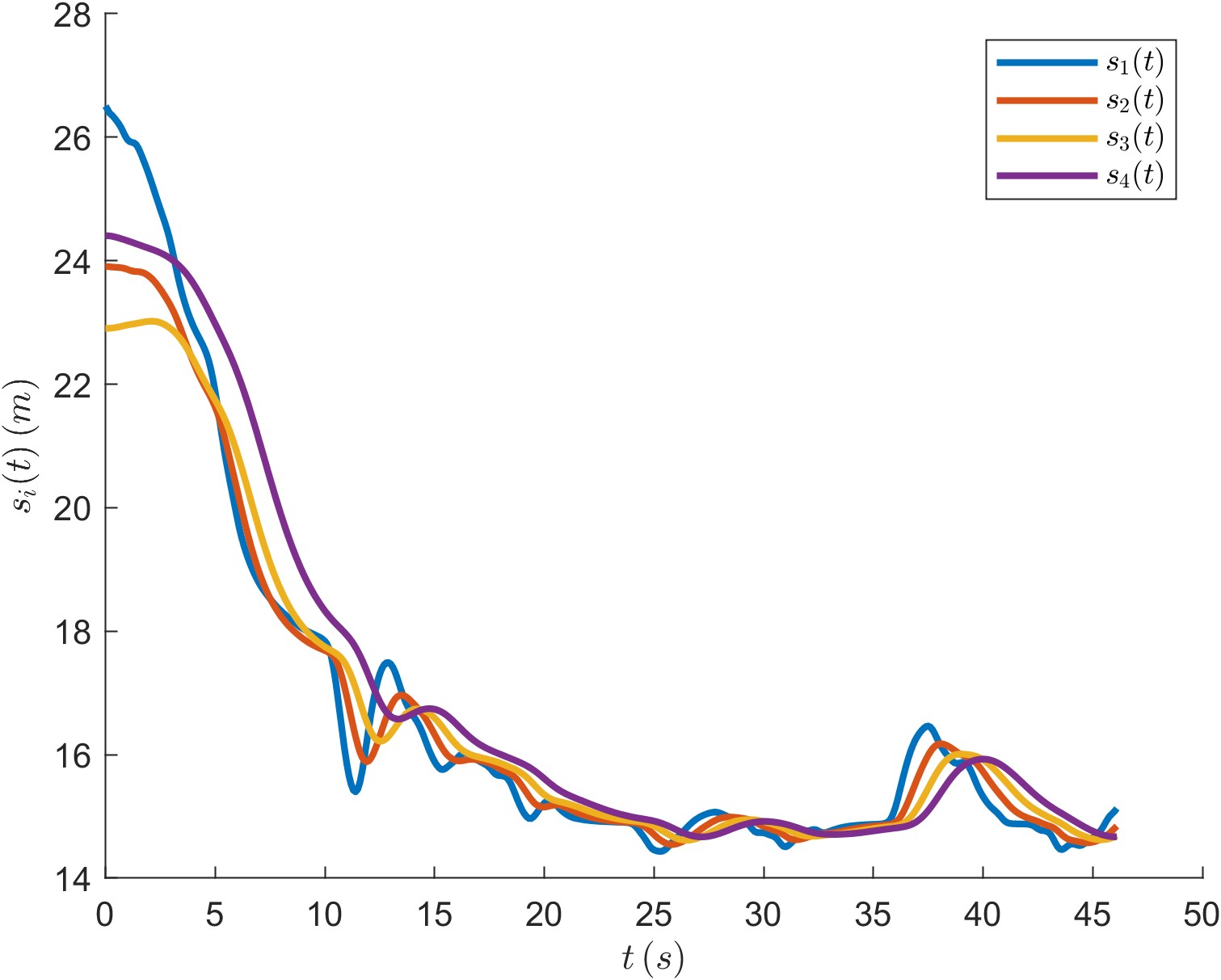} 
	\caption{Comparison between the linear CTH controller \eqref{CTH:linear} (left) and the nonlinear CTH controller \eqref{GrindEQ__3_1_} with \eqref{sigma:nonlinear} (right).  The rows show acceleration (top), speed (middle), and spacing (bottom) when  vehicles share the same time-headway.}
	 \label{Fig111}
		\end{center}
\end{figure}

 \FloatBarrier

Figure \ref{Fig111} shows the acceleration, speed, and spacing trajectories of the automated followers when tracking the human-driven leading vehicle (blue).  The leader exhibits several abrupt acceleration and deceleration maneuvers (due to the emergence of stop-and-go waves, as the NGSIM data are obtained from heavily congested traffic) and both the linear and nonlinear CTH controllers attenuate these disturbances, resulting in progressively smoother speed and spacing profiles for the following vehicles while maintaining safe inter-vehicle spacings throughout the simulation. The transient overshoots visible in the follower trajectories are consistent with the initial-condition perturbations considered in the string-stability analysis, rather than indicating amplification of the leader disturbance. Overall, both controllers produce very similar closed-loop responses.
 
Furthermore, Table~\ref{table_a} quantifies performance of the two different controllers. We observe that both the linear and nonlinear controllers achieve similar performance across all considered metrics, with the nonlinear controller providing a marginal improvement. This is expected since, in the considered scenario, the initial spacing deviations from the desired equilibrium are such that they do not lead to aggressive accelerations in order for the spacings to be regulated rapidly at their desired values. The advantages of the nonlinear controller become more pronounced in scenarios that involve larger spacing deviations, where stronger acceleration commands may be required (this is shown in the following subsection).
\subsubsection{Human-Driven Vehicles Platoon  {With Automated Leader}}

We consider now a platoon maneuver extracted from the OpenACC dataset (AstaZero platoon 3, time interval $[600,800]$), comprising of four human-driven followers and one automated leader. This particular platoon maneuver is shown in Figure \ref {Fig2}. We consider again the linear CTH controller \eqref{CTH:linear} and the nonlinear controller \eqref{GrindEQ__3_1_} with \eqref{sigma:nonlinear}. We choose control gain $k_i=\frac{1.2}{h_i}$,  $v_{\max}=25$, $\tilde{H}= \frac{v_{max}}{2}$, and $\psi(v)$ given by \eqref{psi:example}. The initial conditions (extracted from the dataset) are $v_{0}(0) = 16.08 \left(\frac{m}{s} \right) $, $v_{1} (0)= 16.18 \left(\frac{m}{s} \right) $, $v_{2}(0) = 16.59 \left(\frac{m}{s} \right) $, $v_{3} (0)= 17.10 \left(\frac{m}{s} \right) $, $v_{4} (0) = 16.66 \left(\frac{m}{s} \right) $; $s_{1 } (0) = 28.96 \nobreakspace (m)$, $s_{2}(0) = 51.90 \nobreakspace (m)$, $s_{3}(0) = 55.66 \nobreakspace (m)$, and $s_{4}(0) = 26.22 \nobreakspace (m)$. 

We first consider that all vehicles have the same time-headway $h_i=0.8 \nobreakspace (s)$ and the standstill distance $r_i=10 \nobreakspace (m)$ with the leading vehicle's trajectory taken from the dataset. Figure~\ref{Fig3} shows the accelerations and speeds using the  linear and nonlinear CTH controllers, respectively. Due to the small time-headway, the spacing error $H_i$ is initially large, leading some vehicles using the linear controller to very high acceleration values and speed values that exceed the speed limit $v_{\max}=25 \left(\frac{m}{s} \right)$. Both issues are avoided when using the nonlinear controller \eqref{GrindEQ__3_1_} with \eqref{sigma:nonlinear}. In Figure~\ref{Fig4} is shown the same scenario when vehicles have different time-headways, that is, when $h_1=0.8$, $h_2=1.5$, $h_3=1.5$, $h_4=0.8$, and $r_i=10 \nobreakspace (m)$, $i=1,...,4$.

In Tables \ref{table1} and \ref{table2}  we compare the cost values of safety, comfort, fuel consumption, and tracking for the OpenACC data, the linear controller \eqref{CTH:linear}, and the nonlinear controller \eqref{GrindEQ__3_1_}, \eqref{sigma:nonlinear} for the simulation results presented in Figure~\ref{Fig3} and in Figure~\ref{Fig4}, respectively. In both cases the nonlinear controller presents improved performance, especially in the case of large initial spacing error $H_i$ where the comfort cost values in \eqref{C_3} improve significantly. From Tables \ref{table1} and \ref{table2} we further observe that performance under \eqref{GrindEQ__3_1_} is always better than under \eqref{CTH:linear}. This is explained in the present scenario considered from the fact that \eqref{GrindEQ__3_1_} delivers more cautious acceleration commands. As compared with the OpenACC human-driven platoon both \eqref{CTH:linear} and \eqref{GrindEQ__3_1_} behave better, except in terms of the fuel consumption index for the case corresponding to Figure \ref{Fig3} responses. This is explained from the fact that for large initial spacing deviations from the desired equilibrium, both controllers deliver higher accelerations for driving the platoon at equilibrium.

\begin{figure}[pos=ht]
	\begin{center}
		\includegraphics[width = 0.41\linewidth]{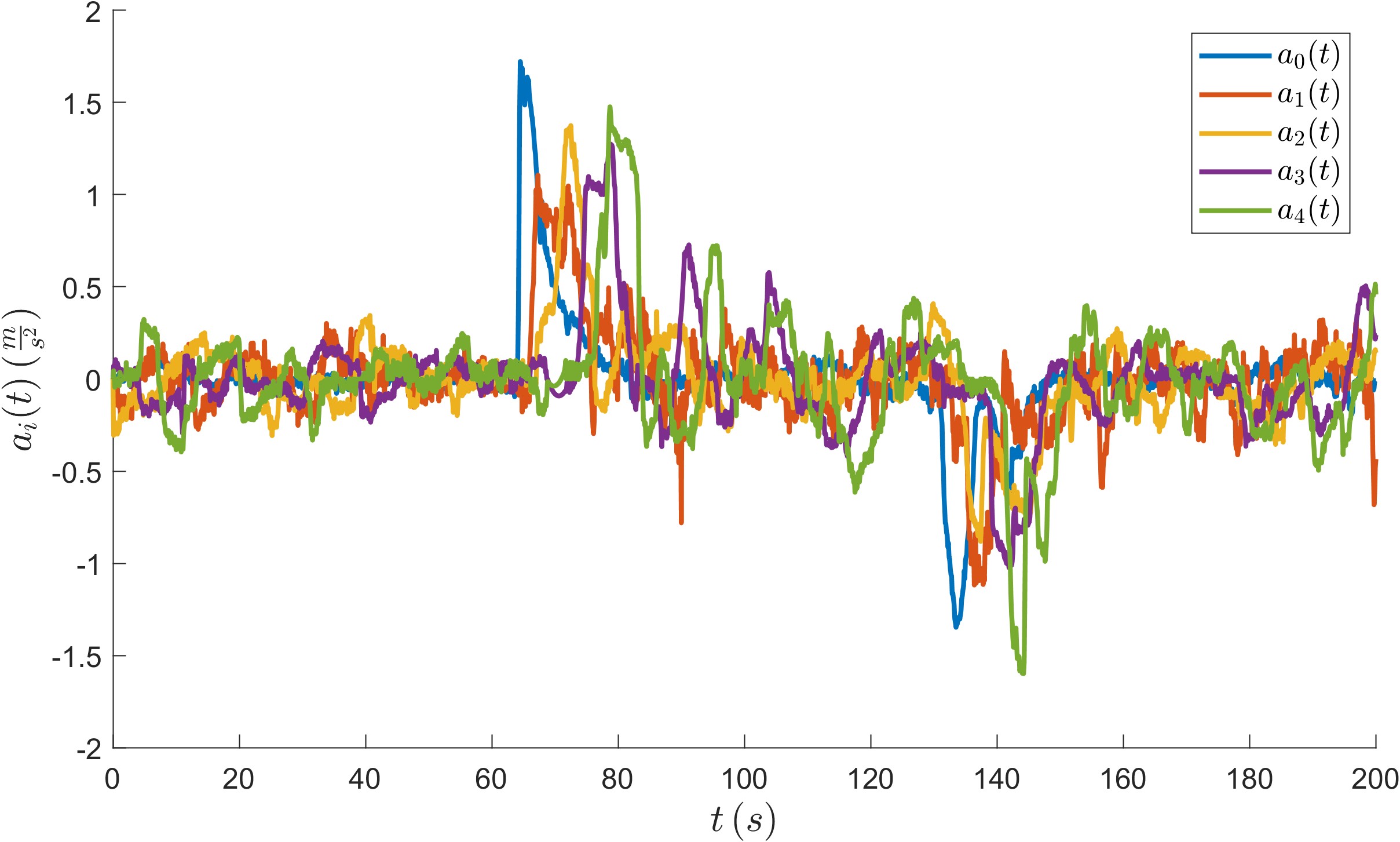}
		\includegraphics[width = 0.41\linewidth]{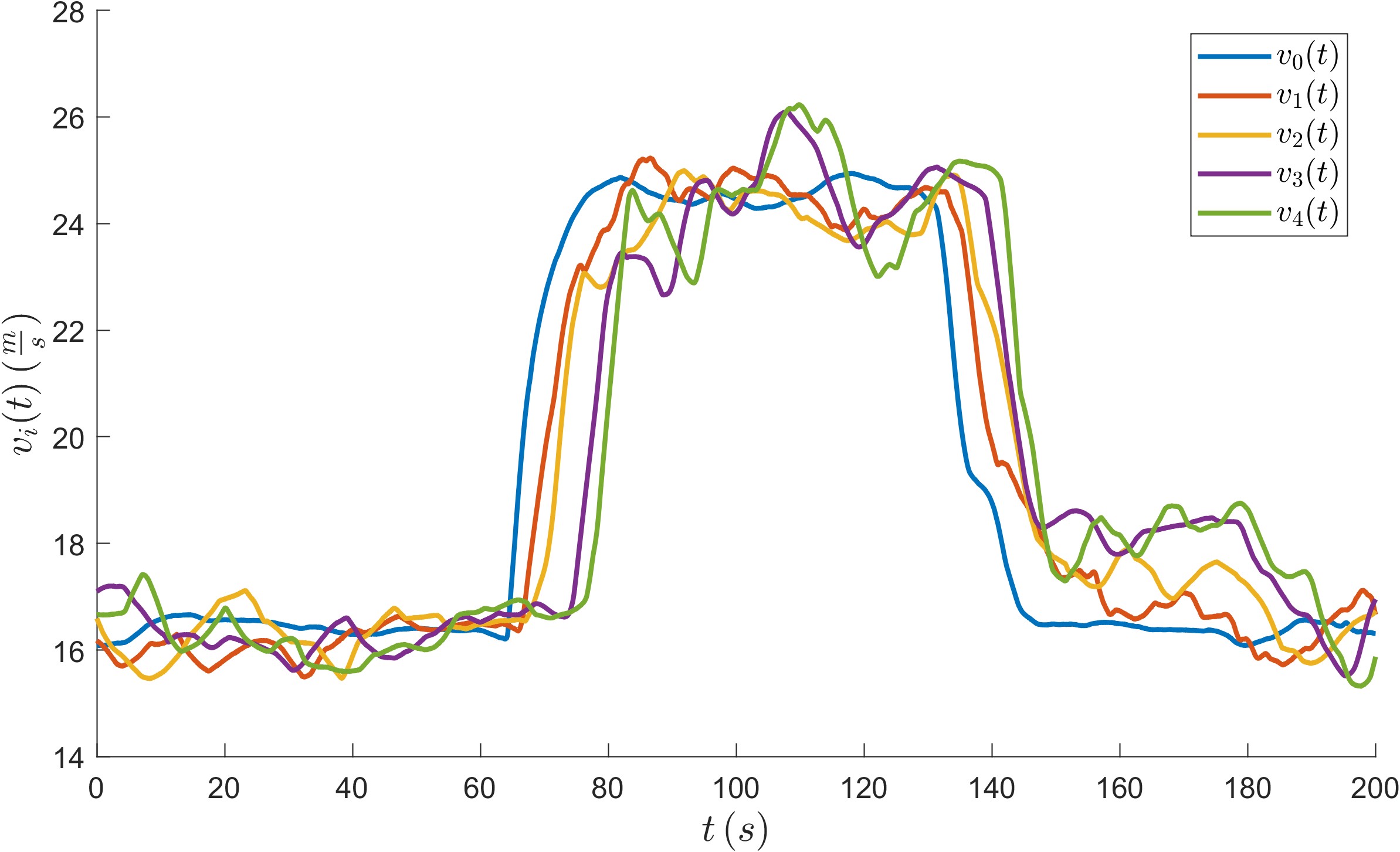}
		\includegraphics[width = 0.41\linewidth]{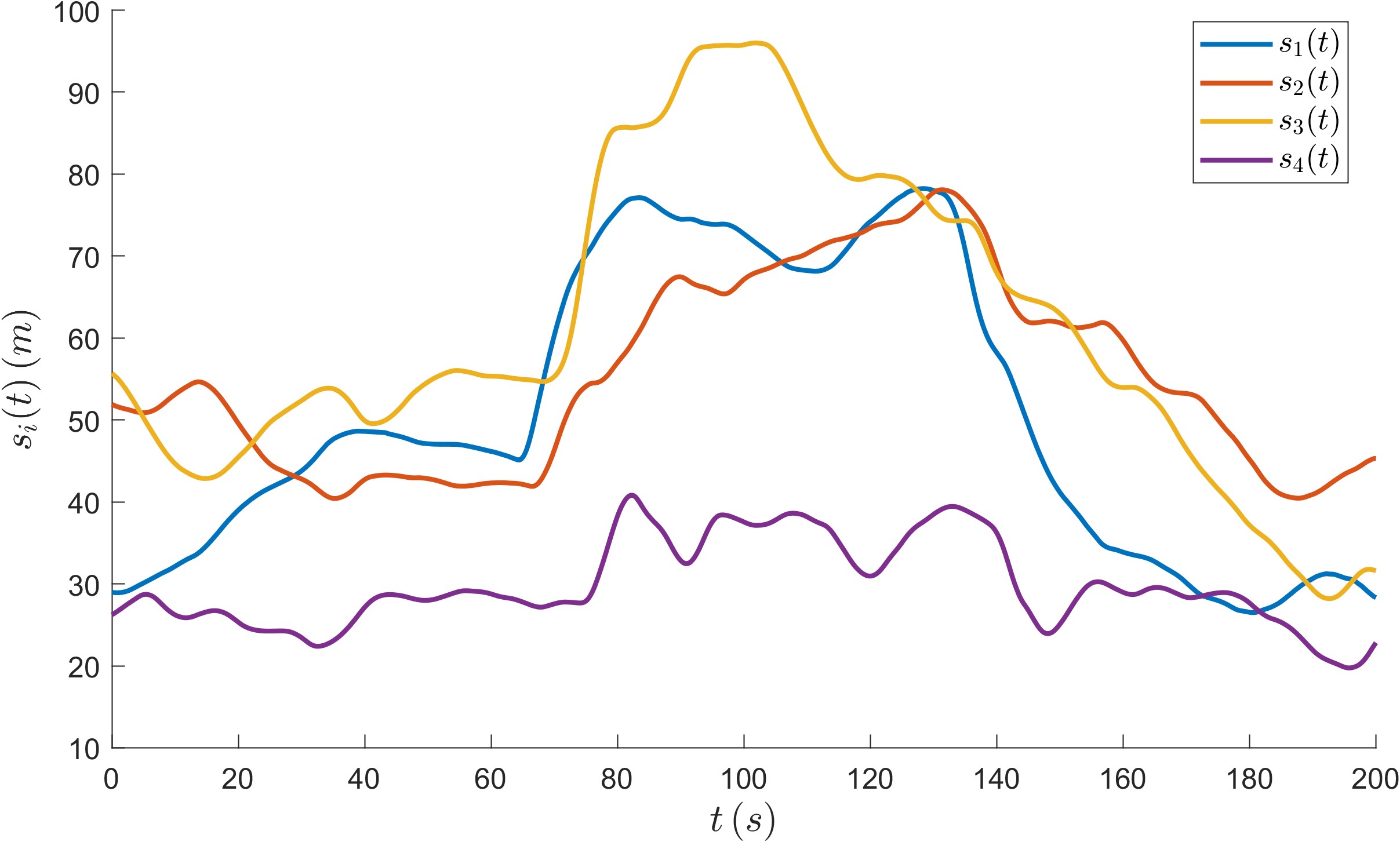}
		\caption{Acceleration (top left), speed (top right), and spacing (bottom) of four human-driven vehicles following a leader with respect to OpenACC dataset.}\label {Fig2}
		\vspace*{-5pt}\end{center}
\end{figure}

\begin{figure}[pos=htbp]
	\begin{center}
	
	\makebox[0.5\linewidth]{\textbf{Linear Controller \eqref{CTH:linear}}}\hfill
	\makebox[0.5\linewidth]{\textbf{Nonlinear Controller \eqref{GrindEQ__3_1_}, \eqref{sigma:nonlinear}}}
	
\includegraphics[width=0.4\linewidth]{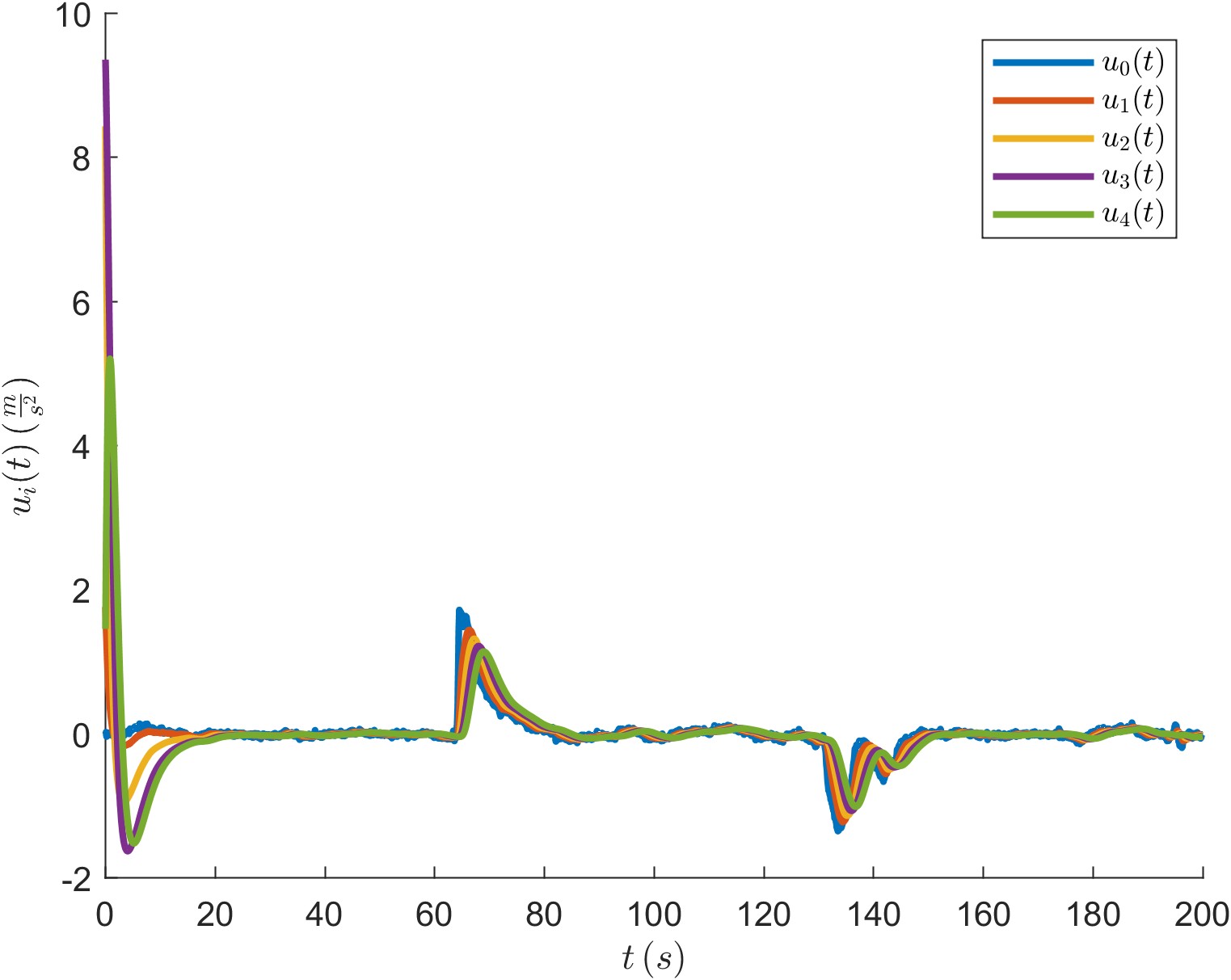}\hfil
\includegraphics[width=0.4\linewidth]{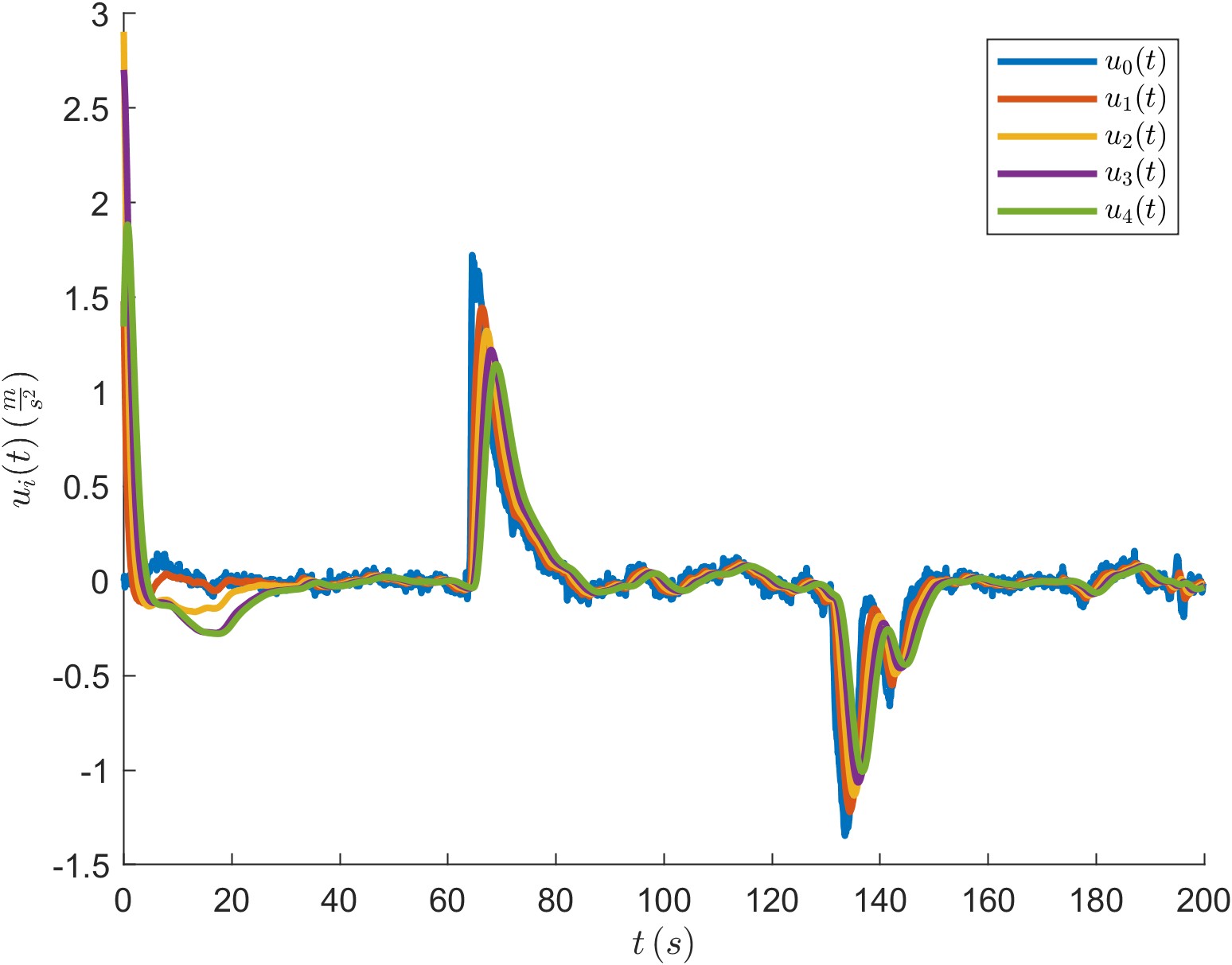} 

\includegraphics[width=0.4\linewidth]{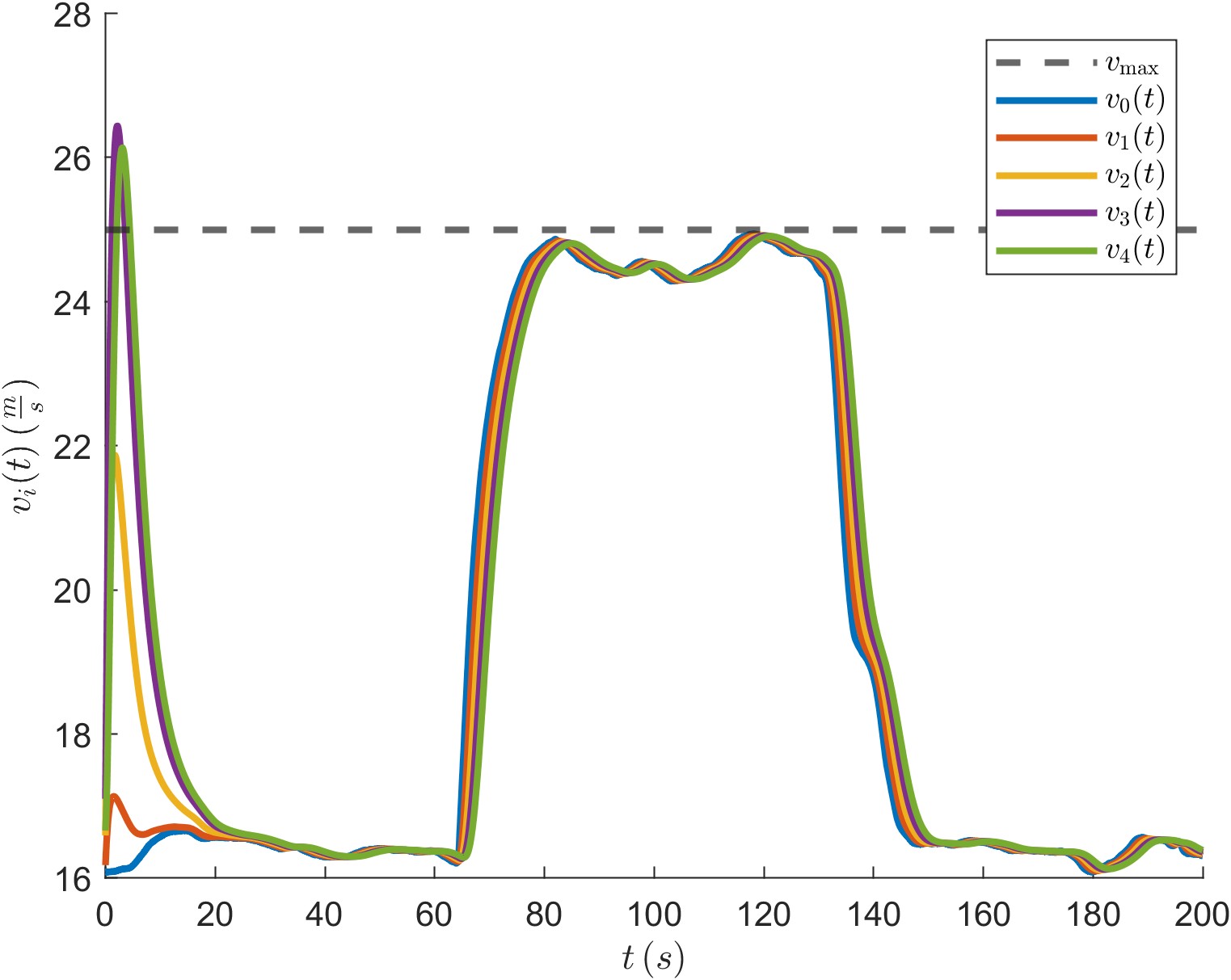}\hfil
\includegraphics[width=0.4\linewidth]{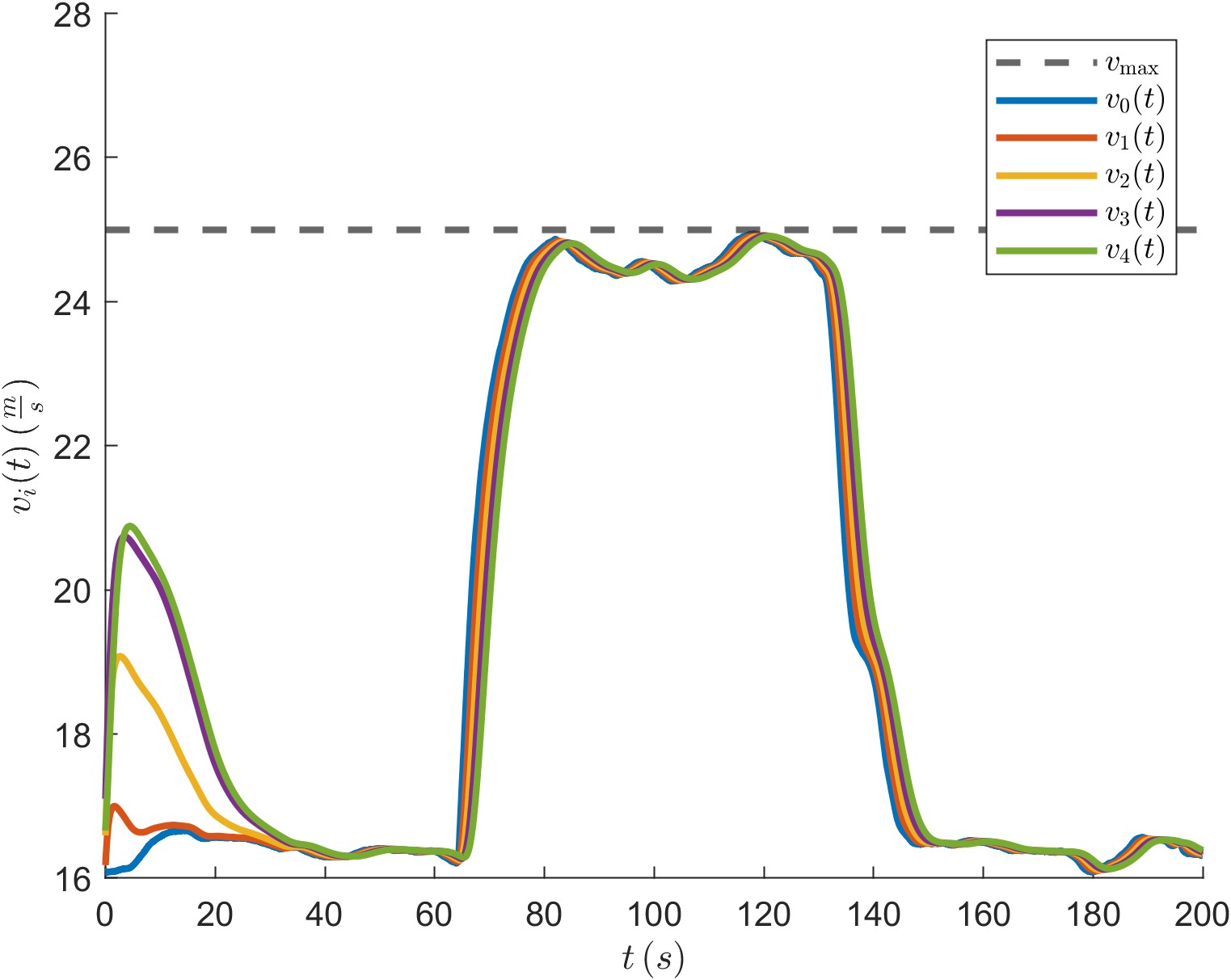} 
		\caption{Comparison between the linear CTH controller \eqref{CTH:linear} (left) and the nonlinear CTH controller \eqref{GrindEQ__3_1_} with \eqref{sigma:nonlinear} (right).  The rows show acceleration (top), speed (bottom) when all vehicles share the same time-headway $h_i=0.8 (s)$.}\label{Fig3}
		\vspace*{-5pt}\end{center}
\end{figure}

 \FloatBarrier
\begin{figure}[pos=htbp]
\centering

\makebox[0.45\linewidth]{\textbf{Linear Controller \eqref{CTH:linear}}}\hfill
\makebox[0.45\linewidth]{\textbf{Nonlinear Controller \eqref{GrindEQ__3_1_}, \eqref{sigma:nonlinear}}}

\vspace{0.4em}

\includegraphics[width=0.45\linewidth]{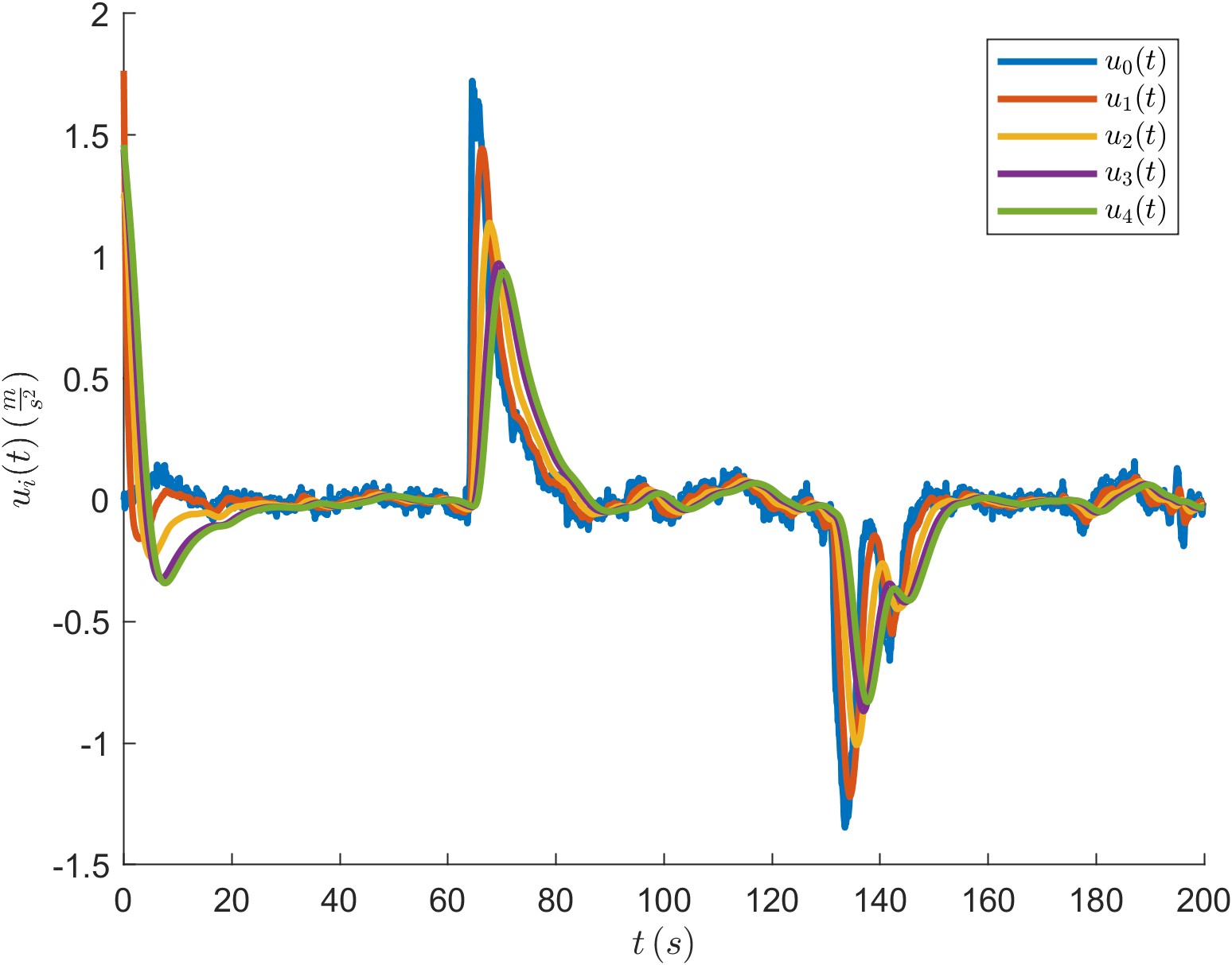}\hfil
\includegraphics[width=0.45\linewidth]{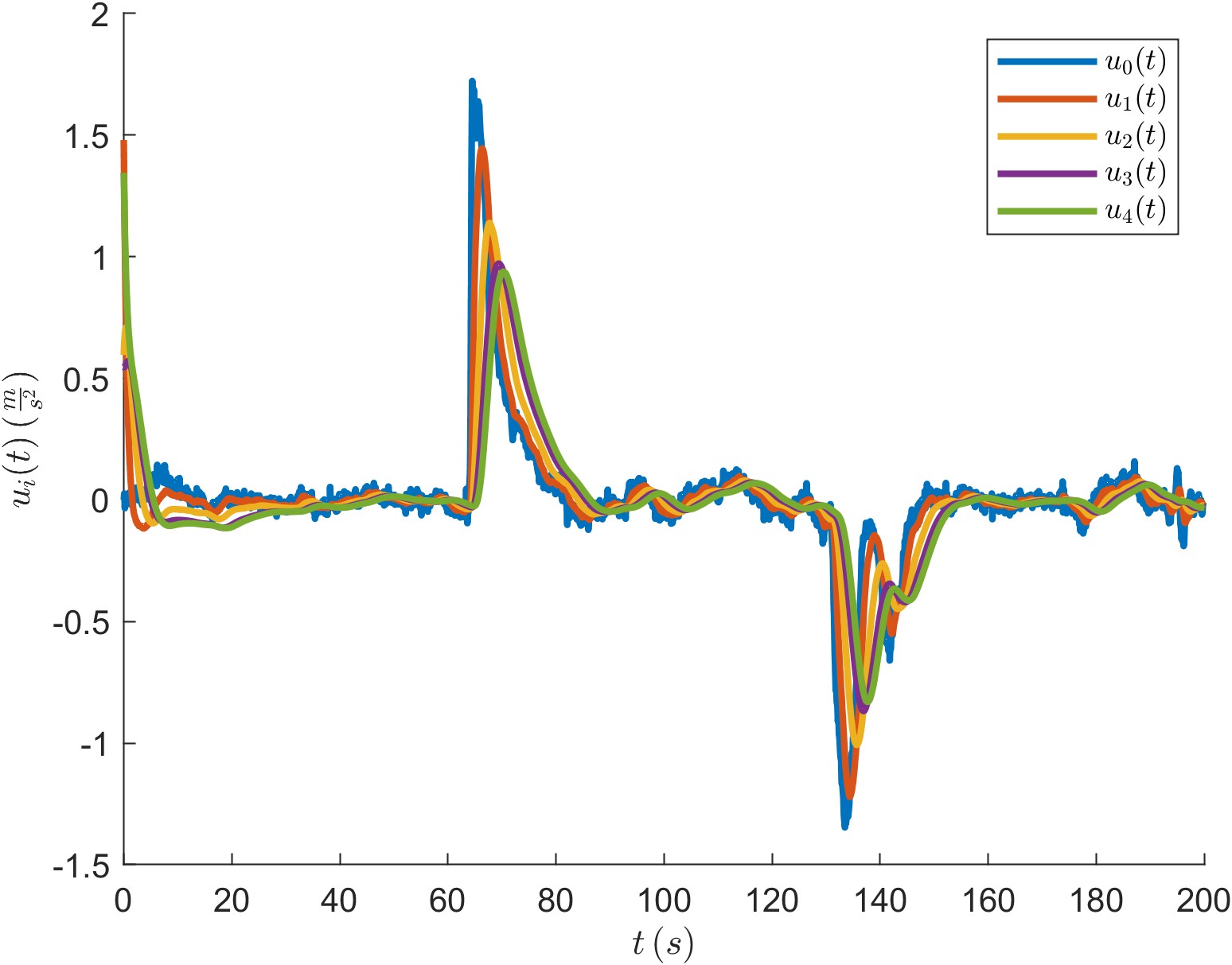} 

\includegraphics[width=0.45\linewidth]{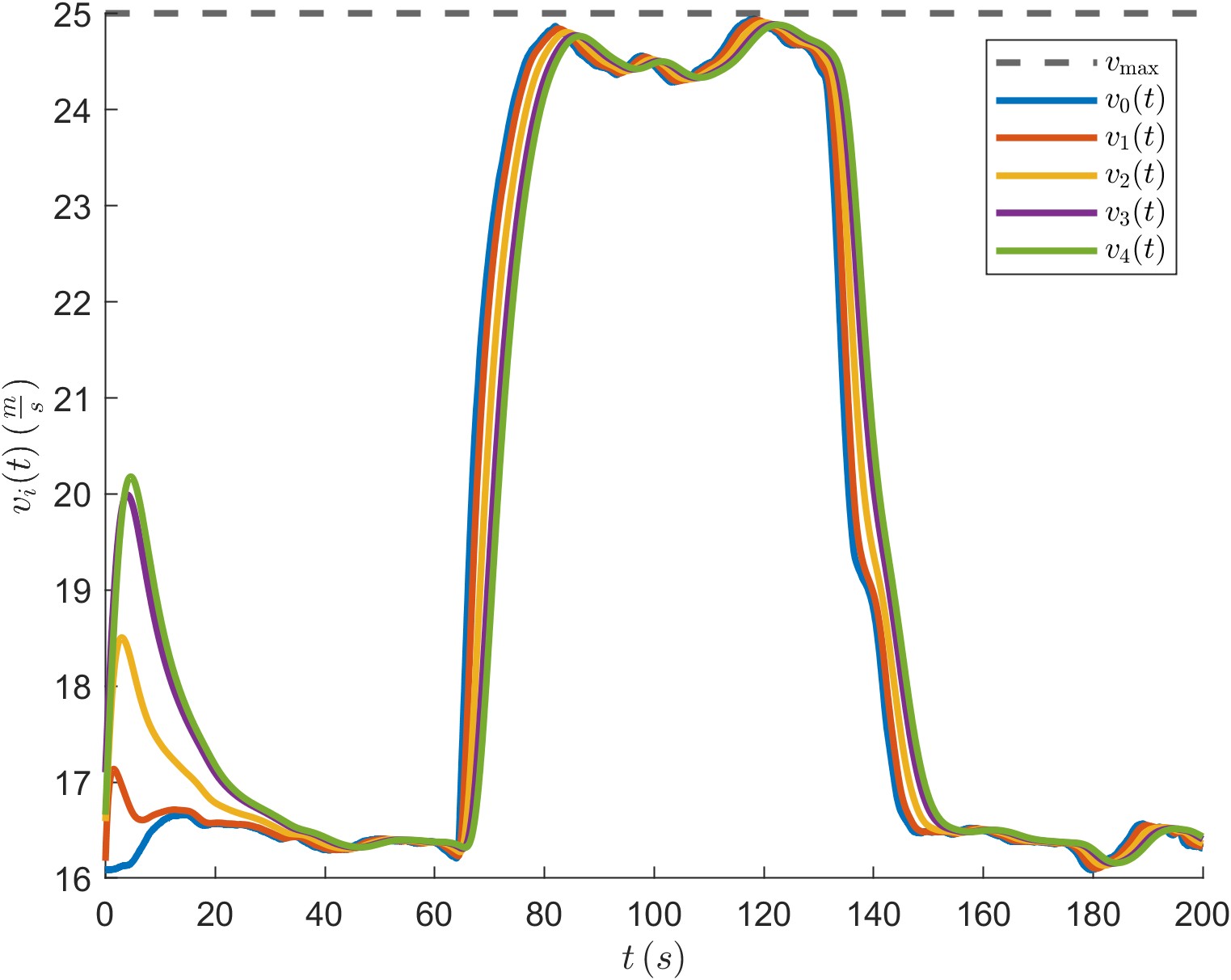}\hfil
\includegraphics[width=0.45\linewidth]{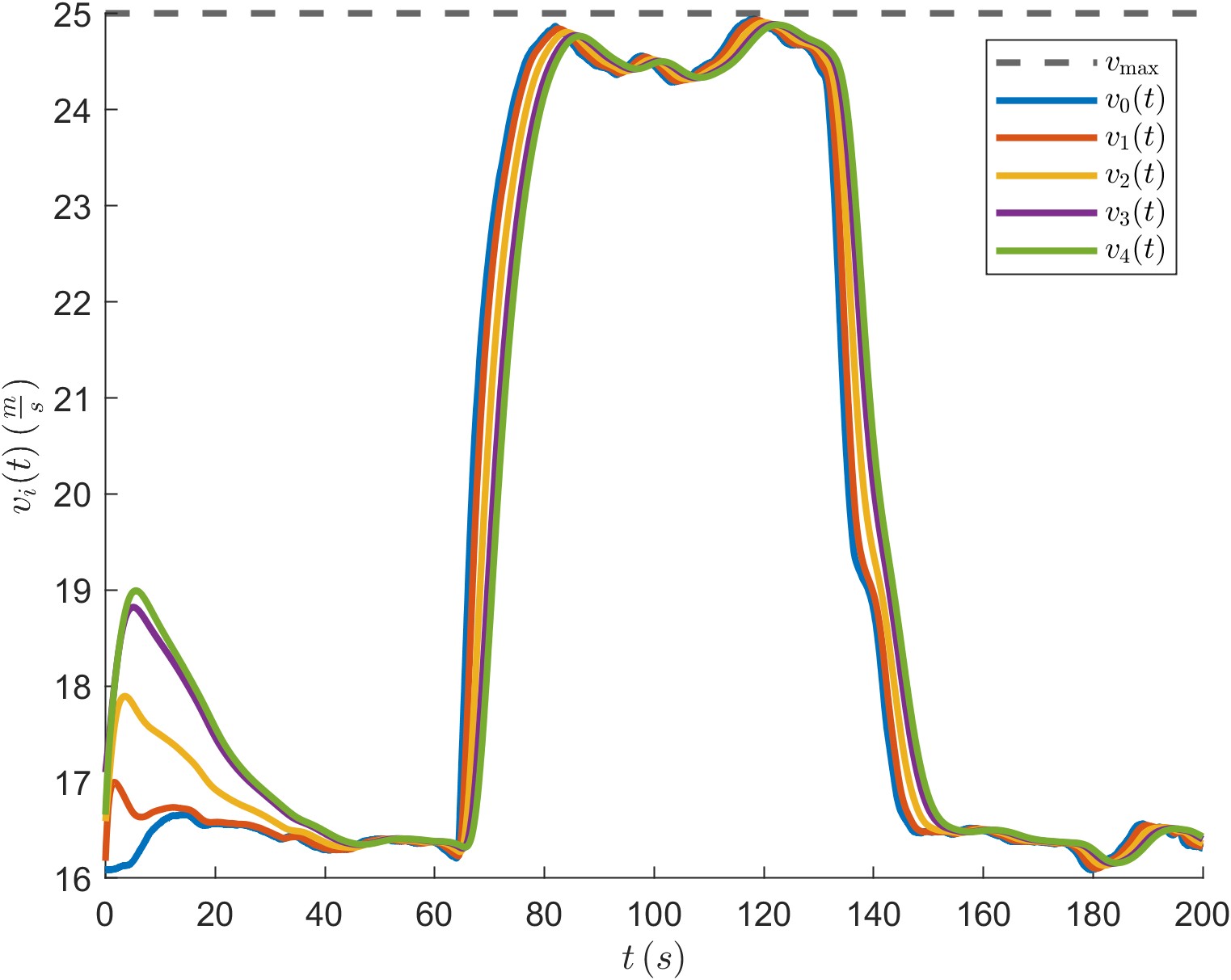} 

\includegraphics[width=0.45\linewidth]{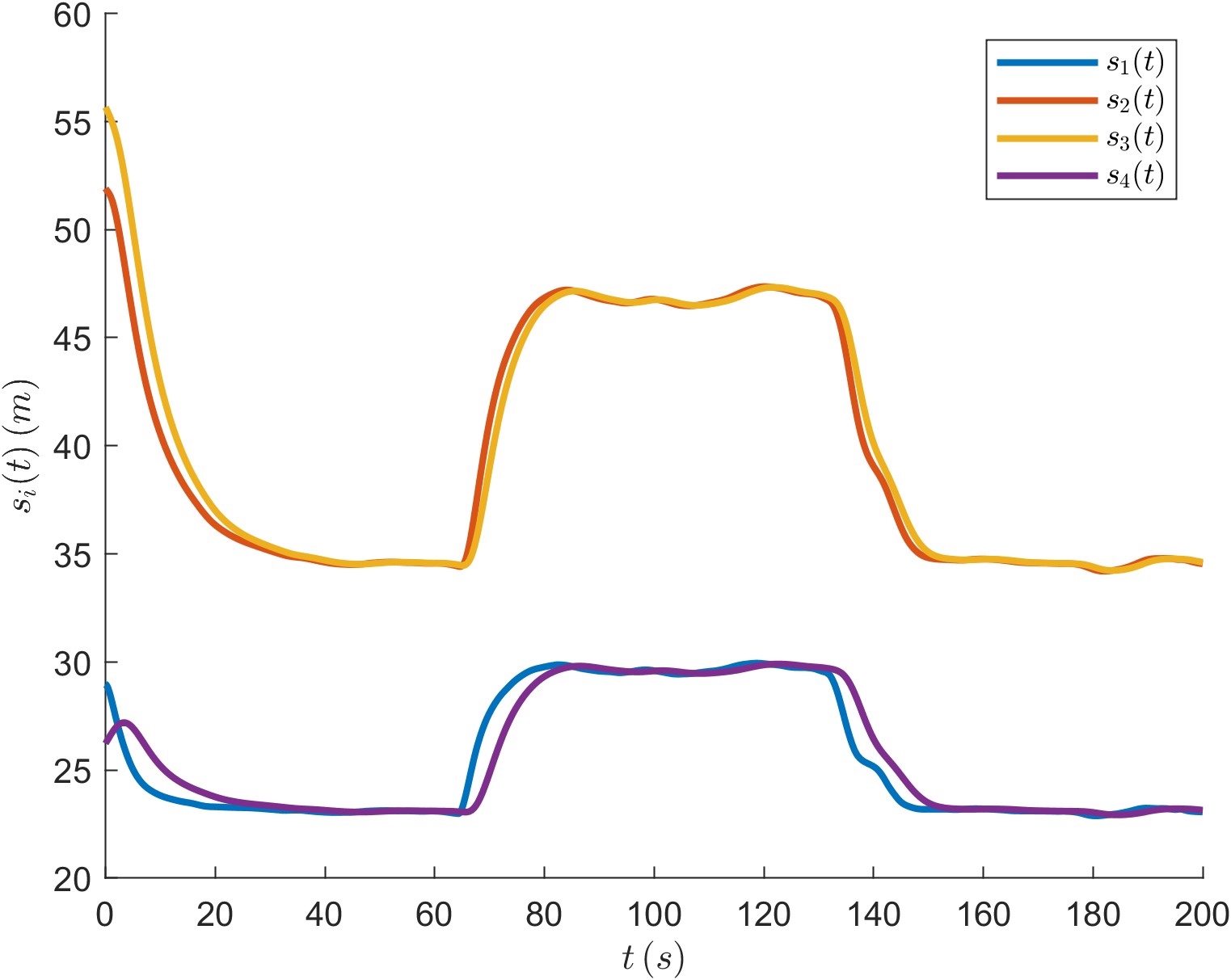}\hfil
\includegraphics[width=0.45\linewidth]{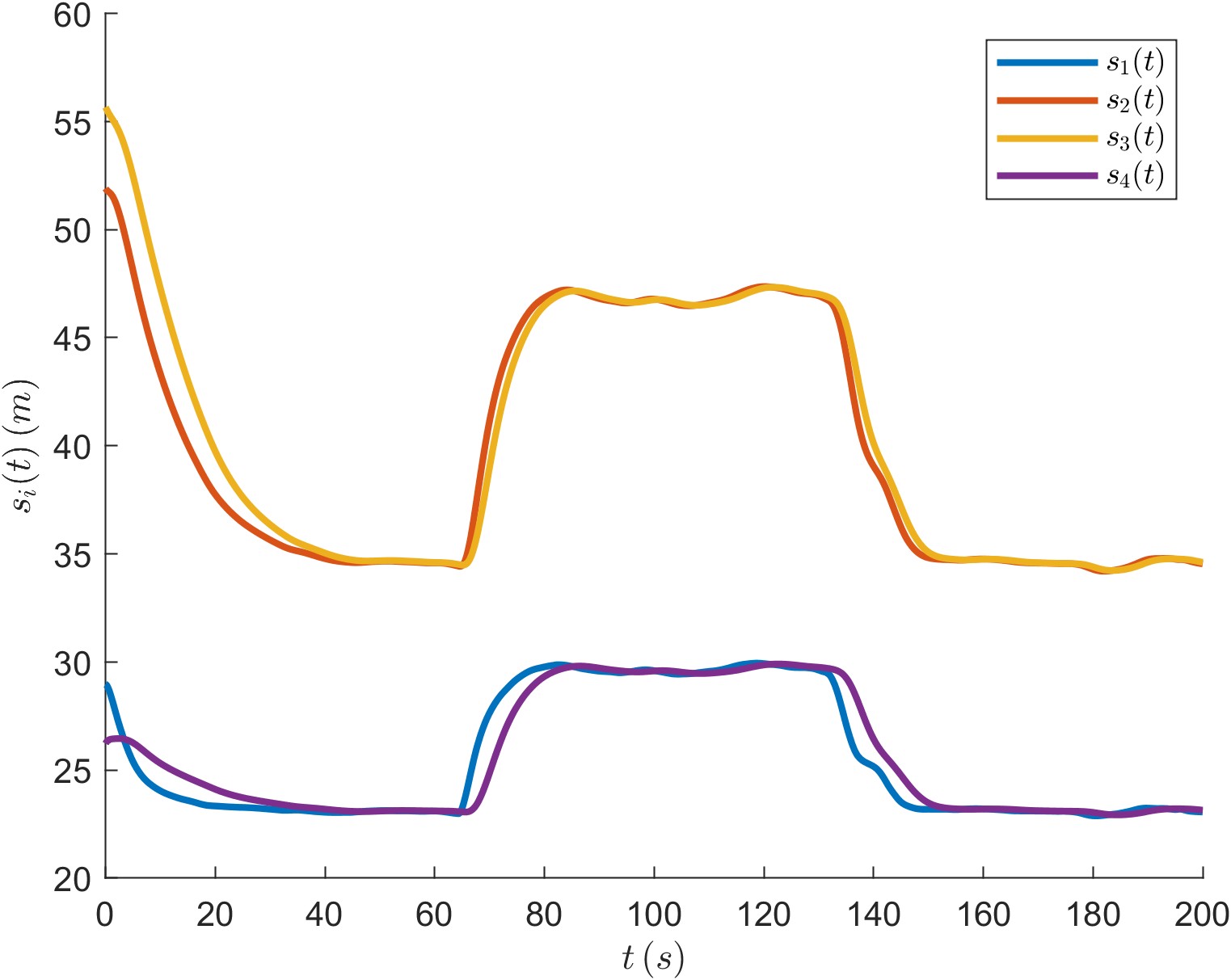} 

\caption{Comparison between the linear CTH controller \eqref{CTH:linear} (left) and the nonlinear CTH controller \eqref{GrindEQ__3_1_} with \eqref{sigma:nonlinear} (right).  The rows show acceleration (top), speed (middle), and spacing (bottom) when  vehicles have a different time-headway.}
\label{Fig4}
\end{figure}

\begin{table}[t]
	\centering
	\caption{Cost values of safety, comfort, fuel consumption, and tracking for OpenACC data, the linear, and the nonlinear CTH-based strategies corresponding to the results shown in Figure~\ref{Fig3}.}
	\scalebox{1}{\begin{tabular}{||c| c  c  c||} 
			\hline
			\backslashbox{Cost function}{Strategy}& OpenACC & Linear \eqref{CTH:linear} & Nonlinear \eqref{GrindEQ__3_1_}, \eqref{sigma:nonlinear} \\ 
			\hline\hline
			Safety & $12.10$ & $10.98$ &  $5.98$ \\ 
			\hline
			Comfort & $290.53$ & $152.70$  &  $18.16$ \\ 
			\hline
			Fuel Consumption & $1624.45$ & $1781.88$  &  $1632.27$ \\ 
			\hline
			Tracking & $644.56$ & $245.99$  &  $128.54$ \\ 
			\hline
	\end{tabular}}
	\label{table1}
\end{table}  

\begin{table}[t]
	\centering
	\caption{Cost values of safety, comfort, fuel consumption, and tracking for OpenACC data, the linear, and the nonlinear CTH-based strategies corresponding to the results shown in Figure~\ref{Fig4}.}
	\scalebox{1}{\begin{tabular}{||c| c  c  c||} 
			\hline
			\backslashbox{Cost function}{Strategy}& OpenACC & Linear \eqref{CTH:linear}  & Nonlinear \eqref{GrindEQ__3_1_}, \eqref{sigma:nonlinear} \\ 
			\hline\hline
			Safety & $12.10$ & $3.47$ &  $2.91$ \\ 
			\hline
			Comfort & $290.53$ & $9.86$  &  $8.14$ \\ 
			\hline
			Fuel Consumption & $1624.45$ & $1611.12$  &  $1603.81$ \\ 
			\hline
			Tracking & $644.56$ & $105.31$  &  $92.48$ \\ 
			\hline
	\end{tabular}}
	\label{table2}
\end{table}
\FloatBarrier

\pagebreak
\subsubsection{Automated Vehicles Platoon}

\begin{figure}[pos=b]
	\begin{center}
		\includegraphics[width = 0.45\linewidth]{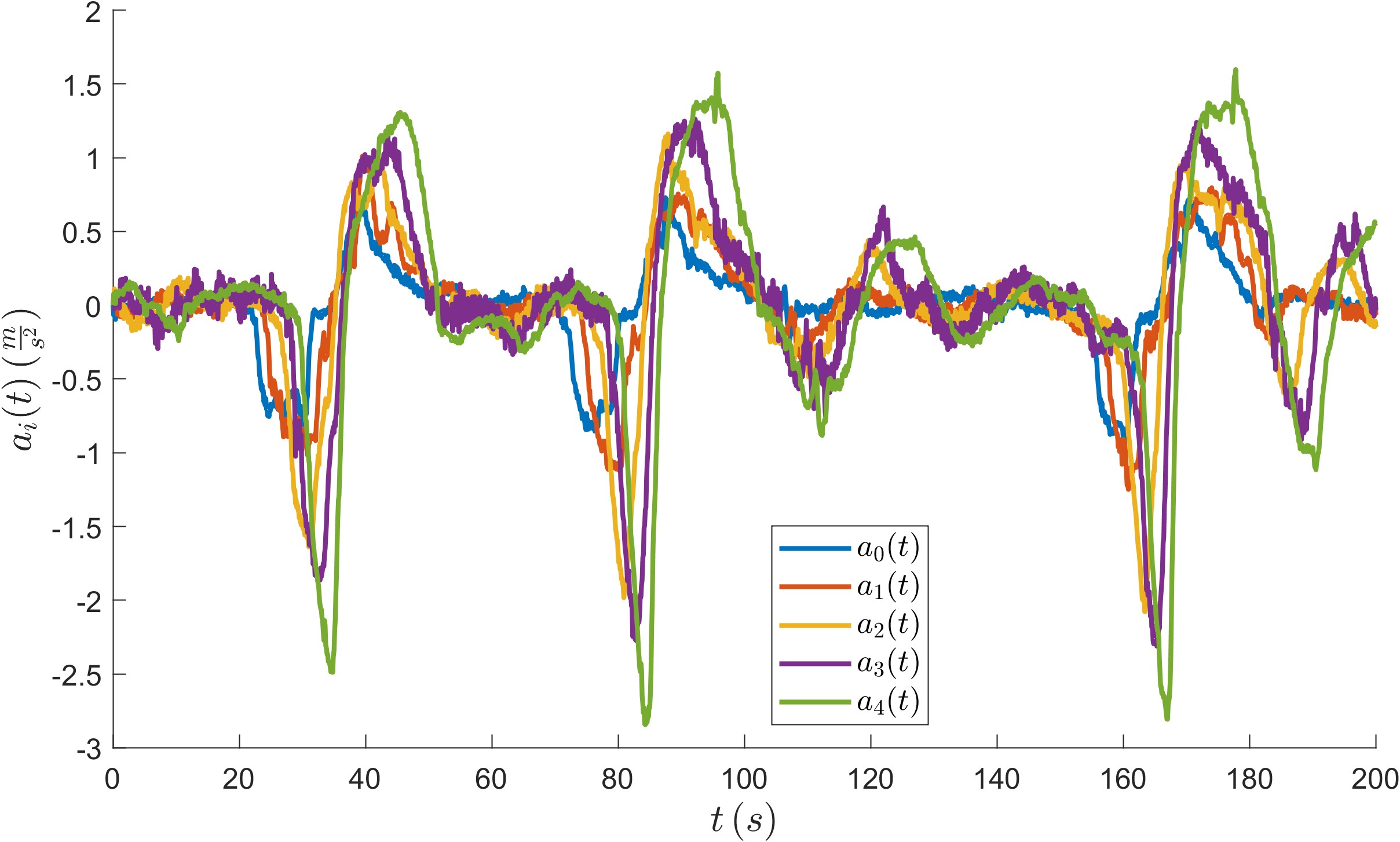}
		\includegraphics[width = 0.45\linewidth]{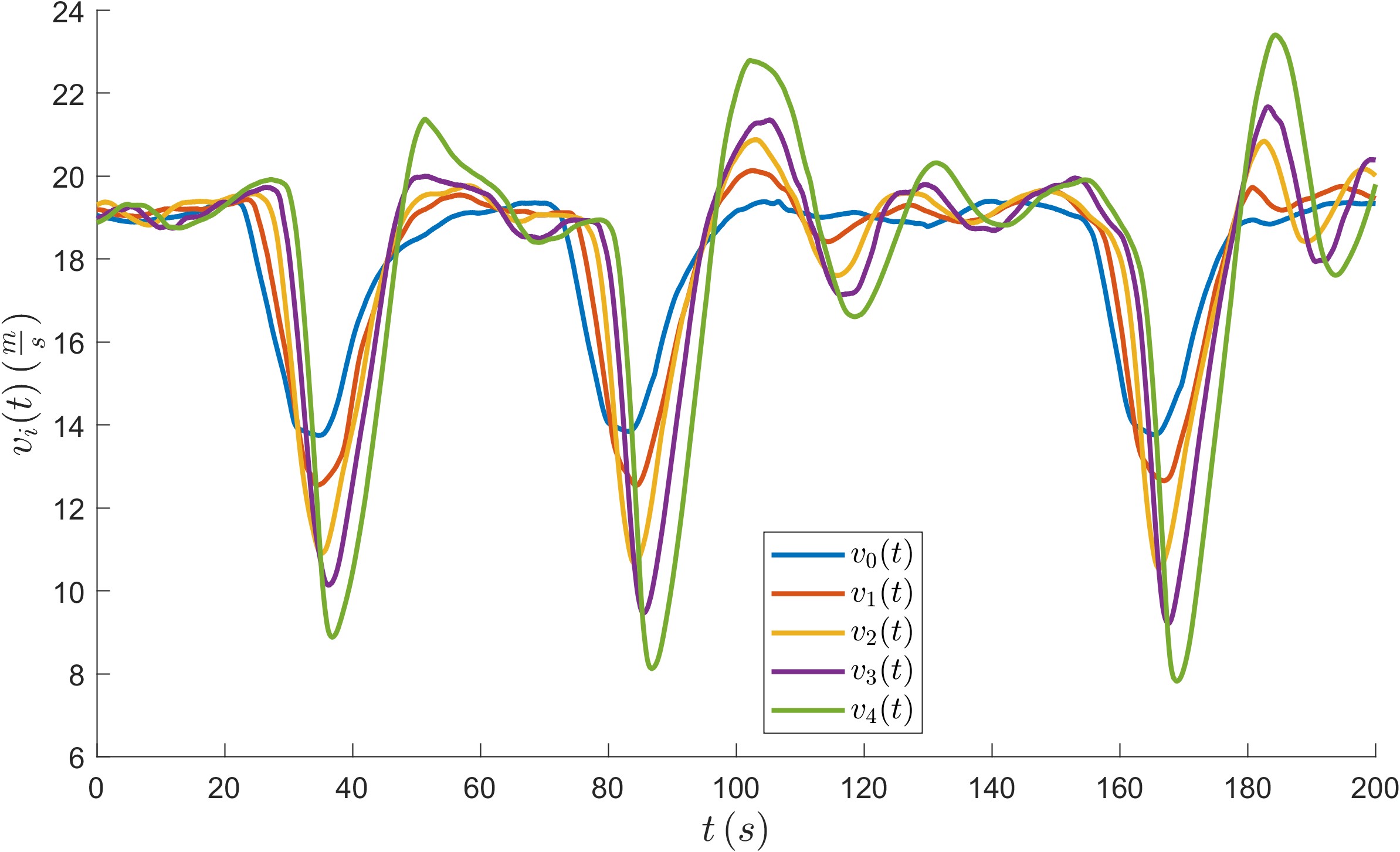}
		\includegraphics[width = 0.45\linewidth]{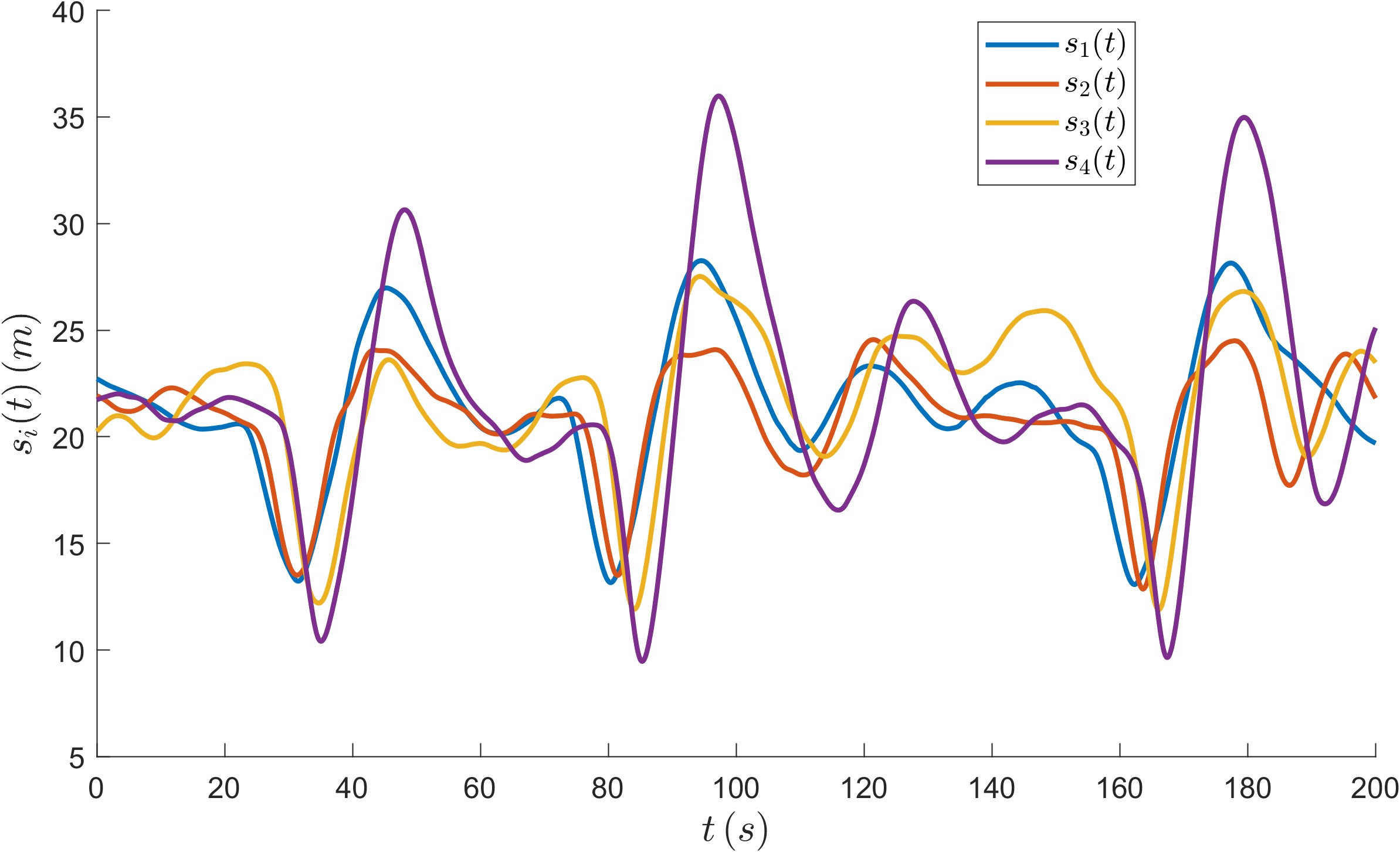}
		\caption{Acceleration (top left), speed (top right), and spacing (bottom) of four vehicles following a leader with trajectories taken from the OpenACC dataset.}\label{Fig5}
		\vspace*{-5pt}\end{center}
\end{figure}
We consider now a platoon maneuver extracted from the OpenACC dataset (AstaZero platoon 2, time interval $[600,800]$), comprising of five commercial automated vehicles ($n=4$ followers and leader). This particular platoon maneuver is shown in Figure \ref {Fig5}. We consider again the linear CTH controller \eqref{CTH:linear} and the nonlinear controller \eqref{GrindEQ__3_1_} with $\sigma$ given by \eqref{sigma:nonlinear}. We choose control gain $k_i=\frac{1.2}{h_i}$, $\beta_i=\frac{1}{h_i}\left(k_i-\frac{1}{h_i}\right)$, $v_{\max}=25$, $\tilde{H}= \frac{v_{max}}{2}$, and $\psi(v)$ given by \eqref{psi:example}. The initial conditions (extracted from the dataset) are  $v_{{\rm 0}_0} = 19.05 \left(\frac{m}{s} \right) $, $v_{{\rm 1}_0} = 19.19 \left(\frac{m}{s} \right) $, $v_{{\rm 2}_0} = 19.30 \left(\frac{m}{s} \right) $, $v_{{\rm 3}_0} = 18.99 \left(\frac{m}{s} \right) $, $v_{{\rm 4}_0} = 18.87 \left(\frac{m}{s} \right) $; $s_{1_0} = 22.71 \nobreakspace (m)$, $s_{2_0} = 21.92 \nobreakspace (m)$, $s_{3_0} = 20.25 \nobreakspace (m)$, $s_{4_0} = 21.73 \nobreakspace (m)$. For the application of linear and nonlinear CTH-based controllers we use the same time-headway $h_i=0.6 \nobreakspace (s)$ and standstill distance $r_i=10 \nobreakspace (m)$ with the leading vehicle's trajectory taken from the dataset.

\begin{figure}[pos=t]
\centering

\makebox[0.45\linewidth]{\textbf{Linear Controller \eqref{CTH:linear}}}\hfill
\makebox[0.45\linewidth]{\textbf{Nonlinear Controller \eqref{GrindEQ__3_1_}, \eqref{sigma:nonlinear}}}

\vspace{0.4em}

\includegraphics[width=0.45\linewidth]{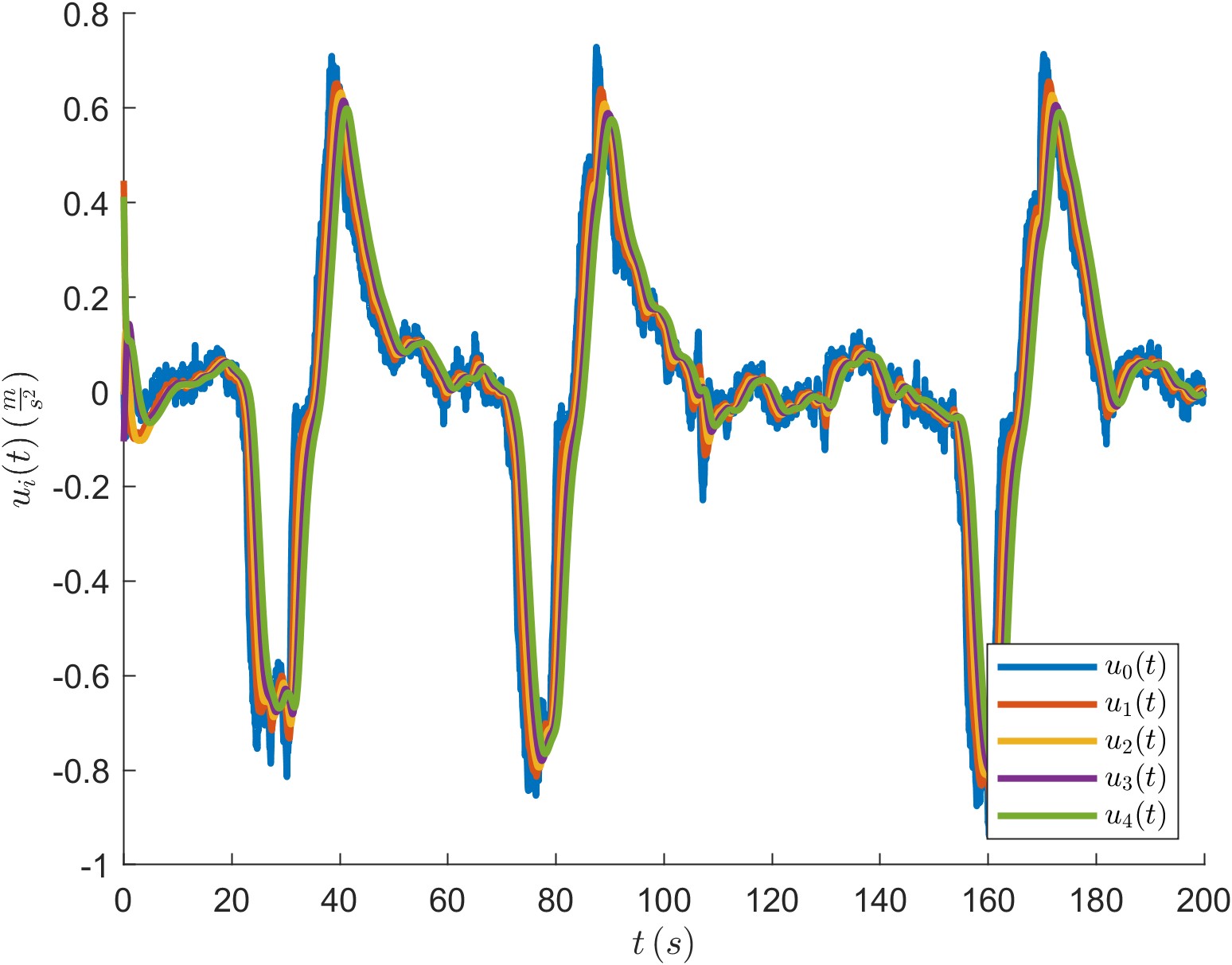}\hfil
\includegraphics[width=0.45\linewidth]{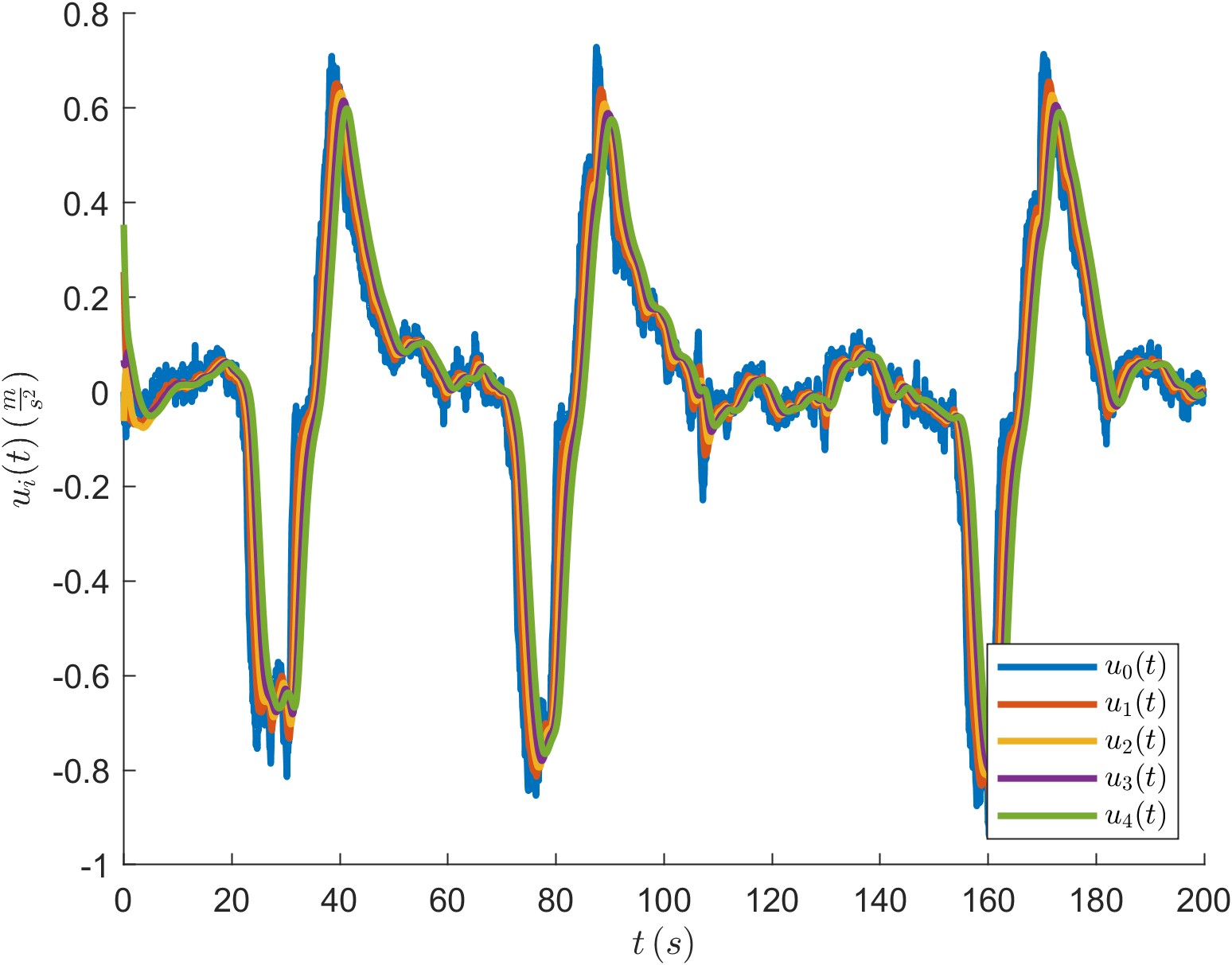} 

\includegraphics[width=0.45\linewidth]{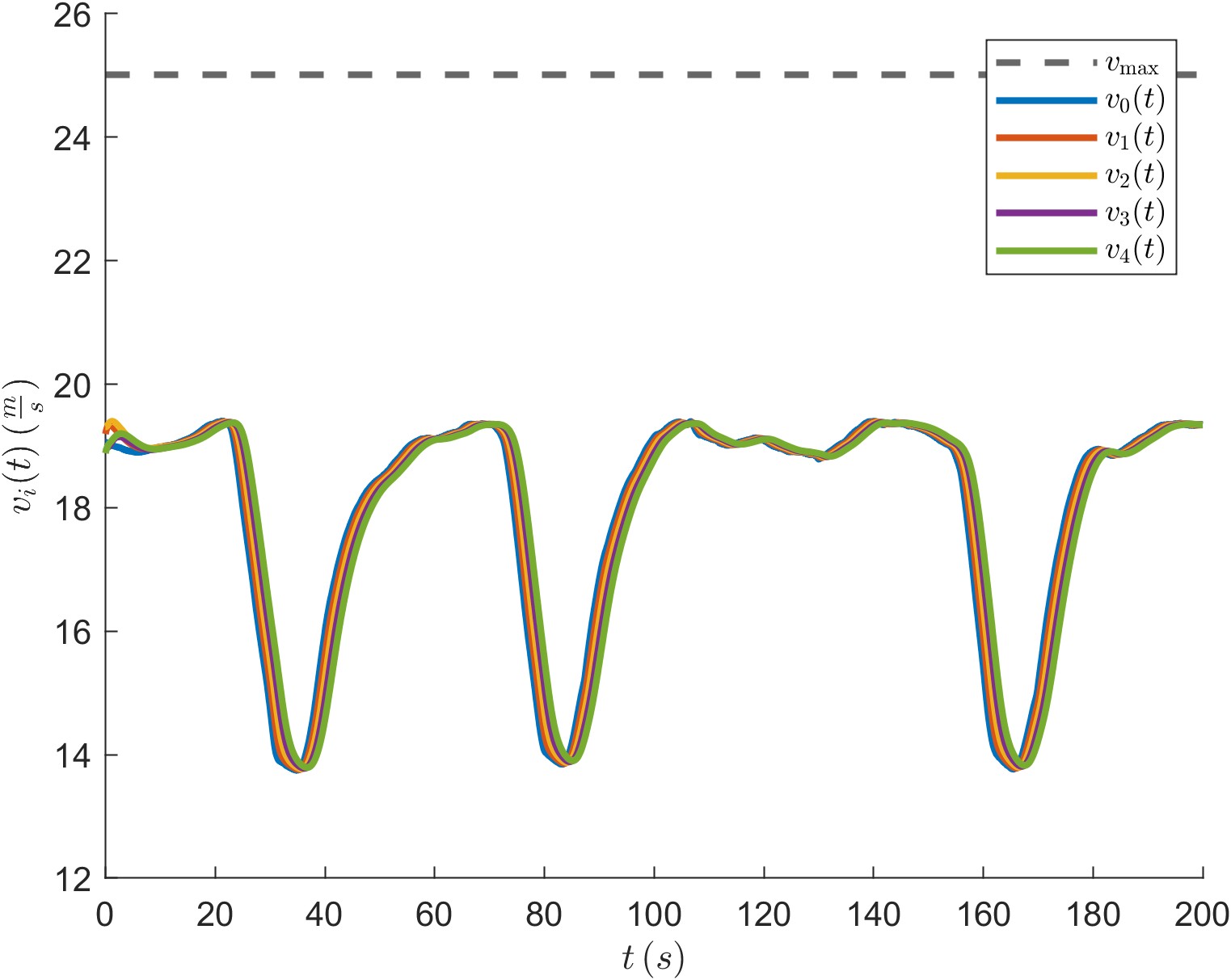}\hfil
\includegraphics[width=0.45\linewidth]{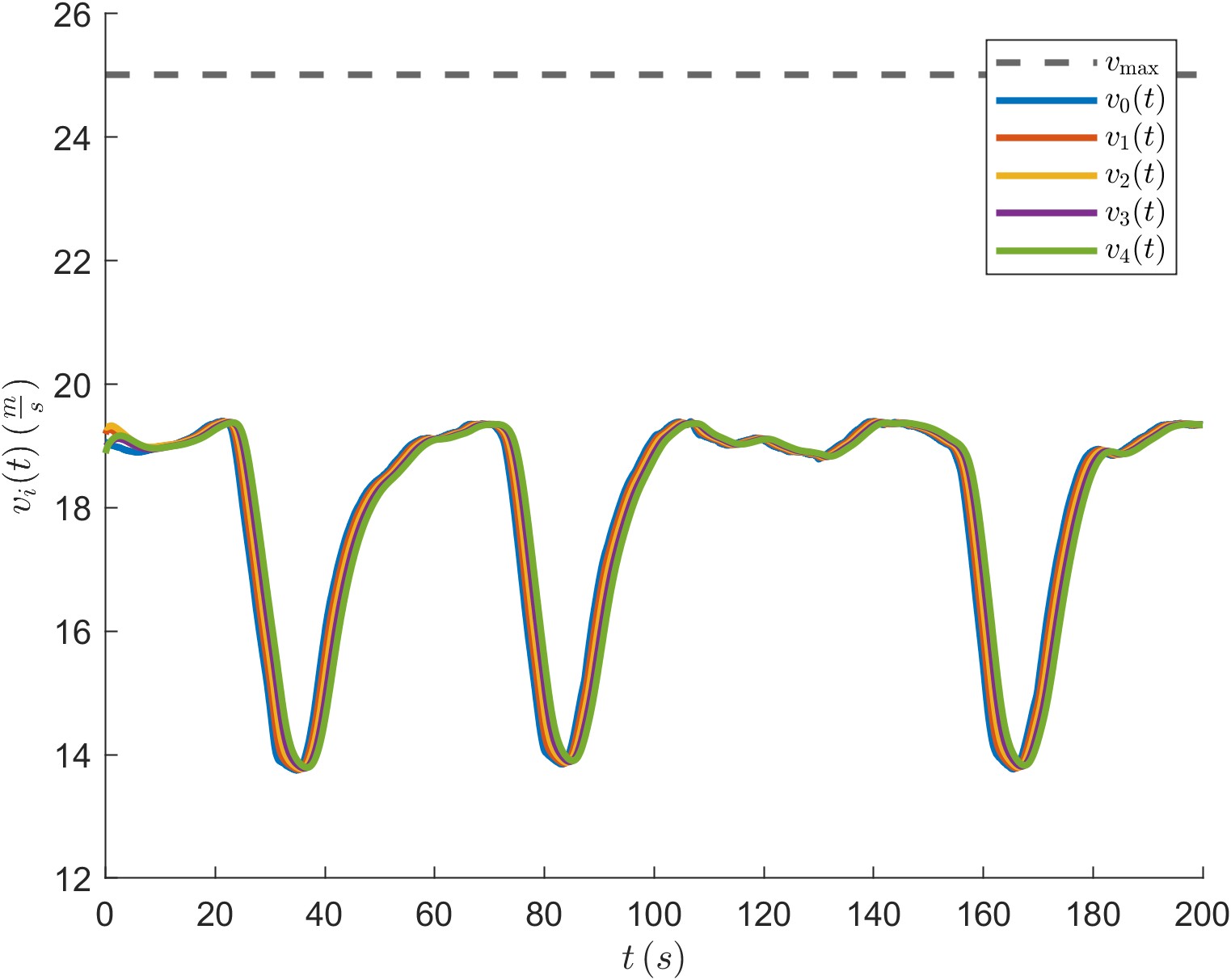} 

\includegraphics[width=0.45\linewidth]{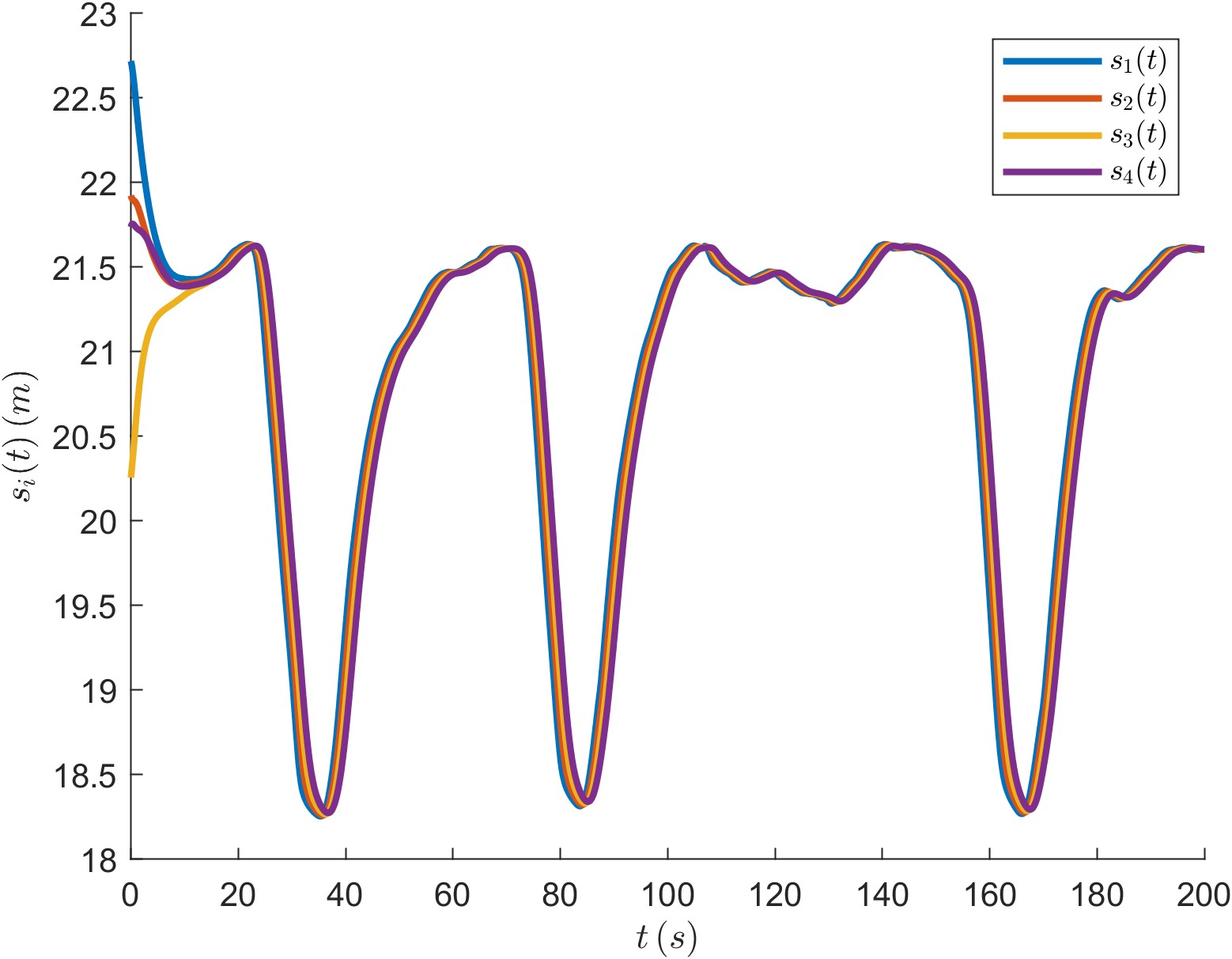}\hfil
\includegraphics[width=0.45\linewidth]{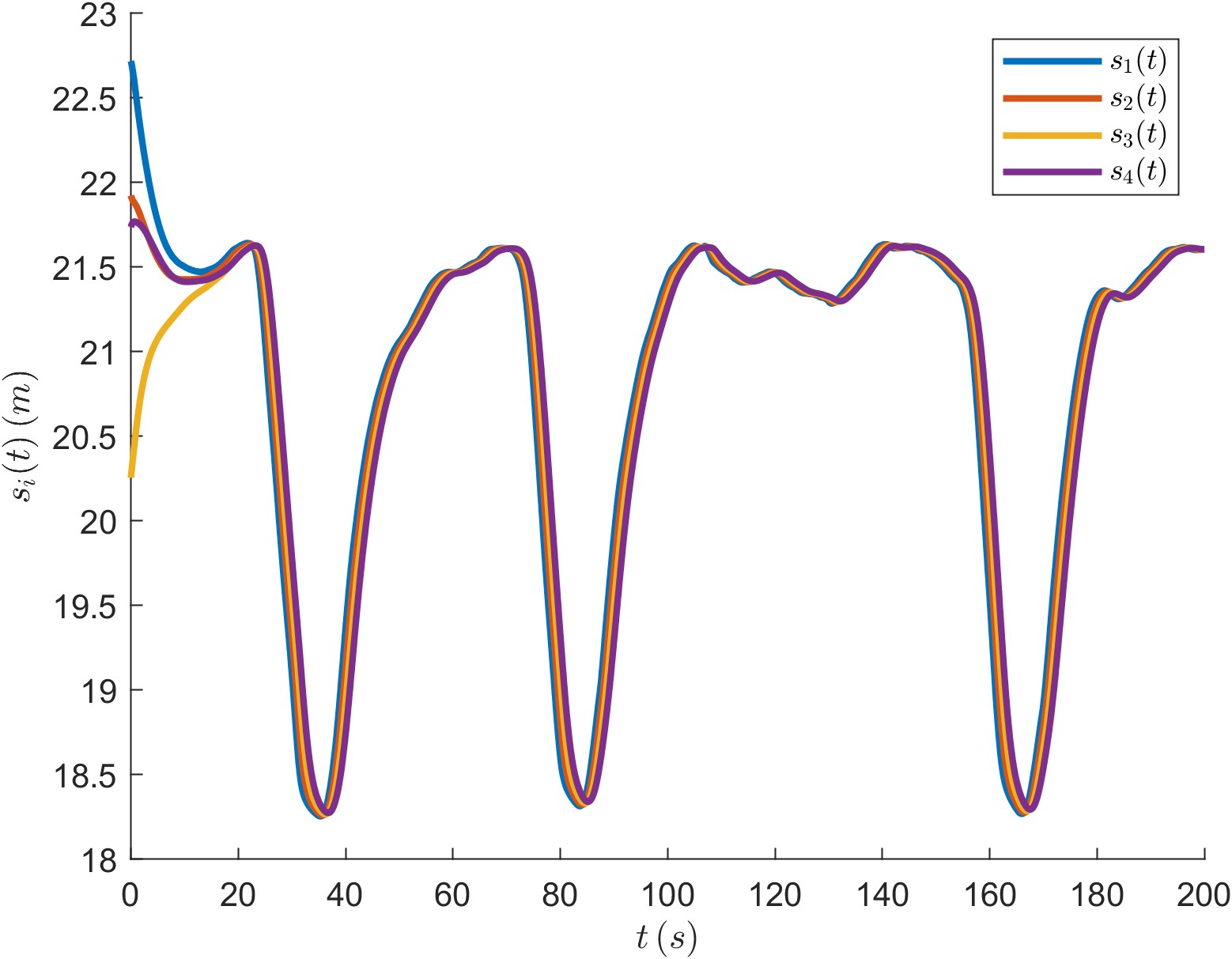} 

\caption{Comparison between the linear CTH controller \eqref{CTH:linear} (left) and the nonlinear CTH controller \eqref{GrindEQ__3_1_} with \eqref{sigma:nonlinear} (right).  The rows show acceleration (top), speed (middle), and spacing (bottom).}
\label{Fig6}
\end{figure}

In Figure~\ref{Fig6} is shown the response of the platoon where the followers use the linear CTH-based controller \eqref{CTH:linear} and the nonlinear CTH-based controller \eqref{GrindEQ__3_1_}, \eqref{sigma:nonlinear}, while the leader's trajectory is extracted from the OpenACC dataset. All stop-and-go waves are suppressed clearly illustrating the string stability guarantees of both the linear and nonlinear CTH controllers. Table \ref{table3} reports the values of the safety, comfort, fuel consumption, and tracking cost functions for the OpenACC data, the linear controller, and the proposed nonlinear controller. Both controllers substantially improve all performance indices compared with the OpenACC trajectories.  This is explained by the fact that the controllers of the commercial, automated vehicles employed in the respective OpenACC experiments, lead to string unstable responses, as also reported in \cite{OpenACC}. Moreover, the linear and nonlinear controllers exhibit very similar performance, with the nonlinear controller providing a slight improvement, particularly in the comfort metric. This behavior is expected since the OpenACC trajectories remain close to the equilibrium operating point, resulting in relatively small spacing errors. Consequently, the nonlinear feedback function operates in its approximately linear region, where $\tanh(x)\approx x$, and the proposed controller behaves similarly to the linear CTH controller. Note however, that for large spacing errors the behavior of the controllers is substantially different (see Figure~\ref{Fig3}).

\begin{table} [t]
	\centering
	\caption{Cost values of safety, comfort, fuel consumption, and tracking for OpenACC data, the linear, and the nonlinear CTH-based strategies corresponding to the results shown in  Figs. \ref{Fig5} and \ref{Fig6}.}
	\scalebox{1}{\begin{tabular}{||c| c  c  c||} 
			\hline
			\backslashbox{Cost function}{Strategy}& OpenACC & Linear \eqref{CTH:linear}  &  Nonlinear \eqref{GrindEQ__3_1_}, \eqref{sigma:nonlinear} \\ 
			\hline\hline
			Safety & $32.22$ & $2.08$ &  $2.07$ \\ 
			\hline
			Comfort & $265.30$ & $7.33$  &  $5.97$ \\ 
			\hline
			Fuel Consumption & $1657.83$ & $1461.68$  &  $1459.46$ \\ 
			\hline
			Tracking & $674.43$ & $44.70$  &  $43.80$ \\ 
			\hline
	\end{tabular}}
	\label{table3}
\end{table}

\subsection{Robustness Simulations }

In this section we numerically study the robustness of the nonlinear controller \eqref{GrindEQ__3_1_} with respect to third-order vehicle dynamics and we then illustrate that the linear CTH-based controller \eqref{CTH:linear} may attain negative speeds for initial conditions selected in $D$ and simultaneously not belonging in $K$.
  
\begin{figure}[pos=t]  
	\begin{center}
		\includegraphics[width = 0.45\linewidth]{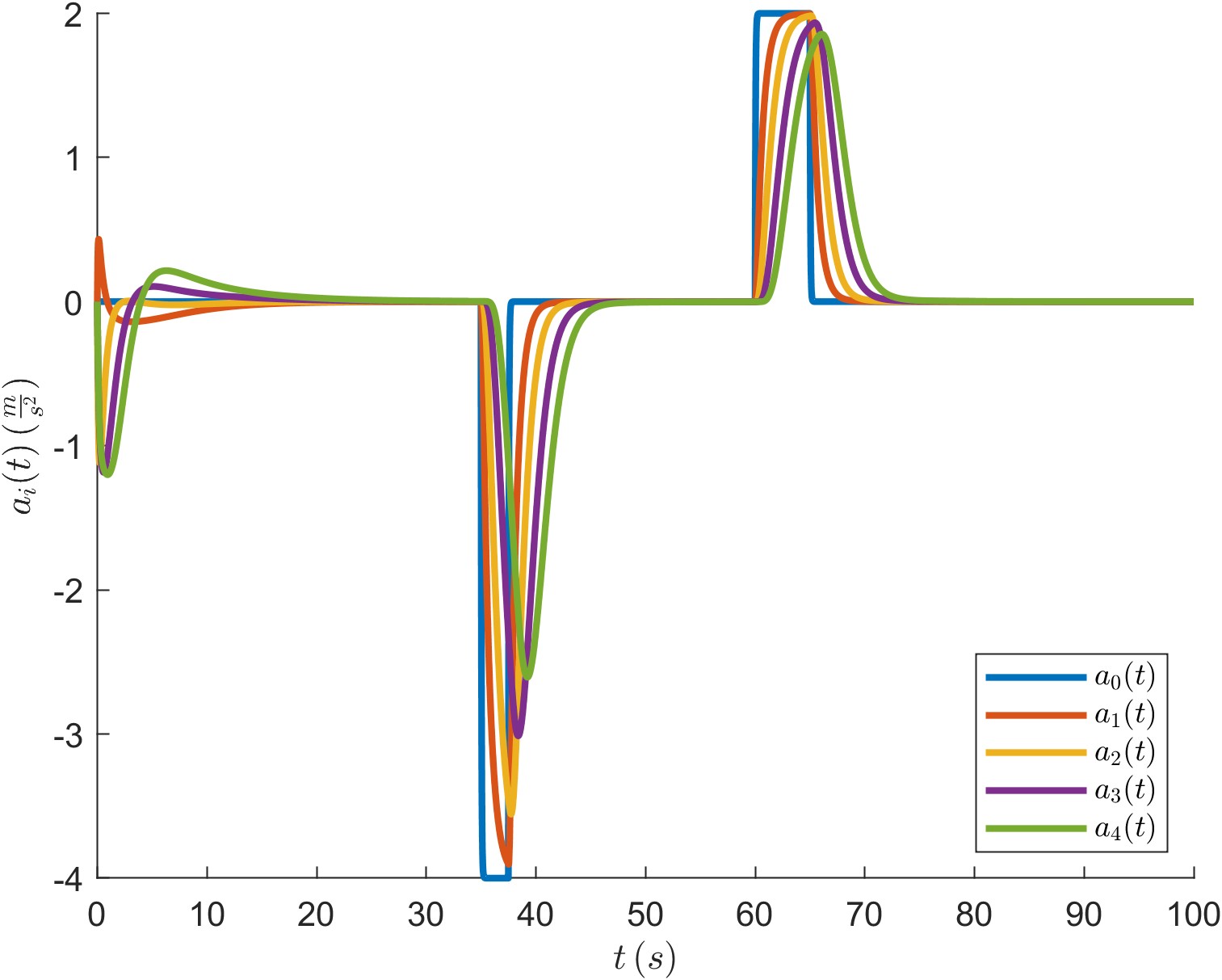}
		\includegraphics[width = 0.45\linewidth]{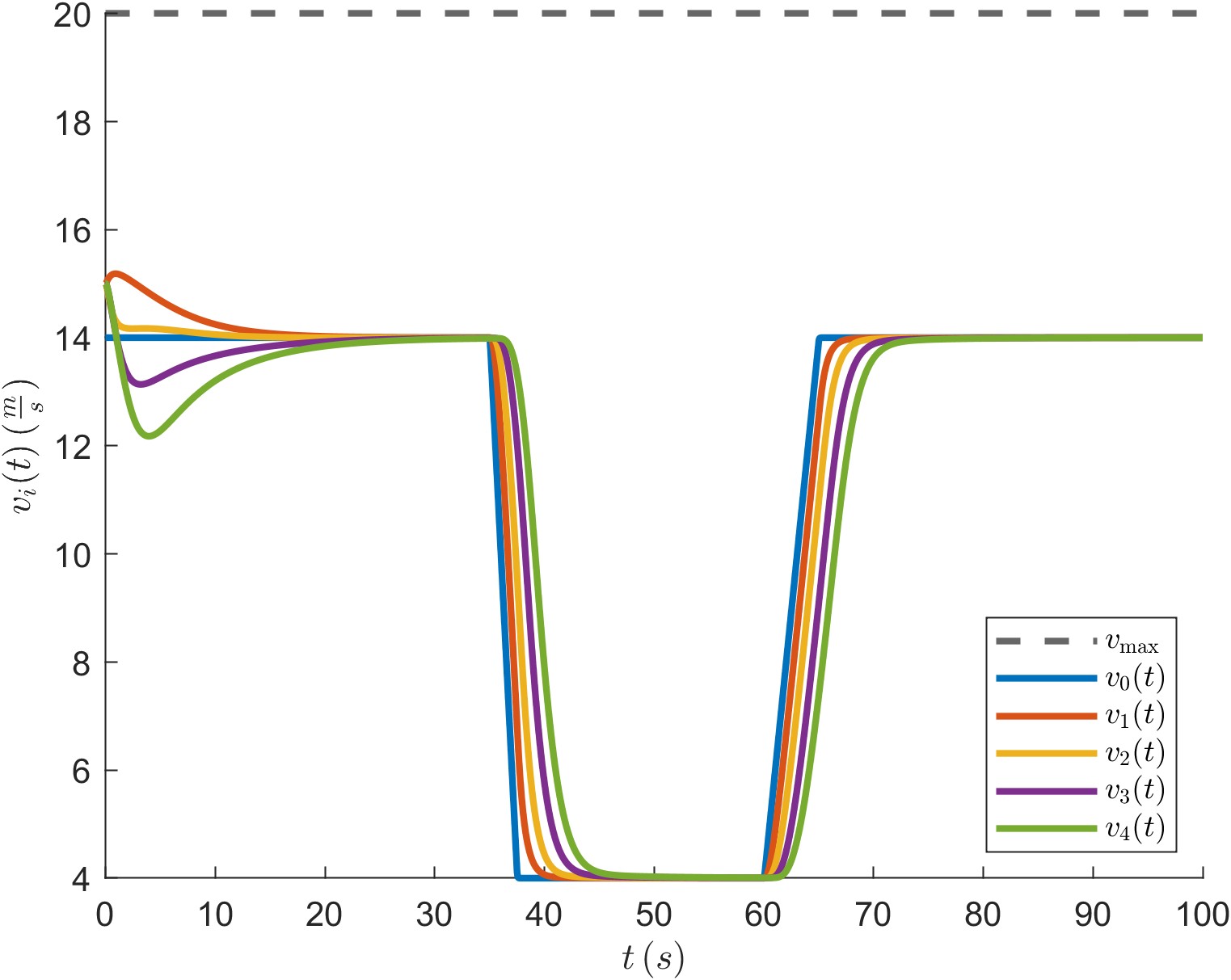}
		\includegraphics[width = 0.45\linewidth]{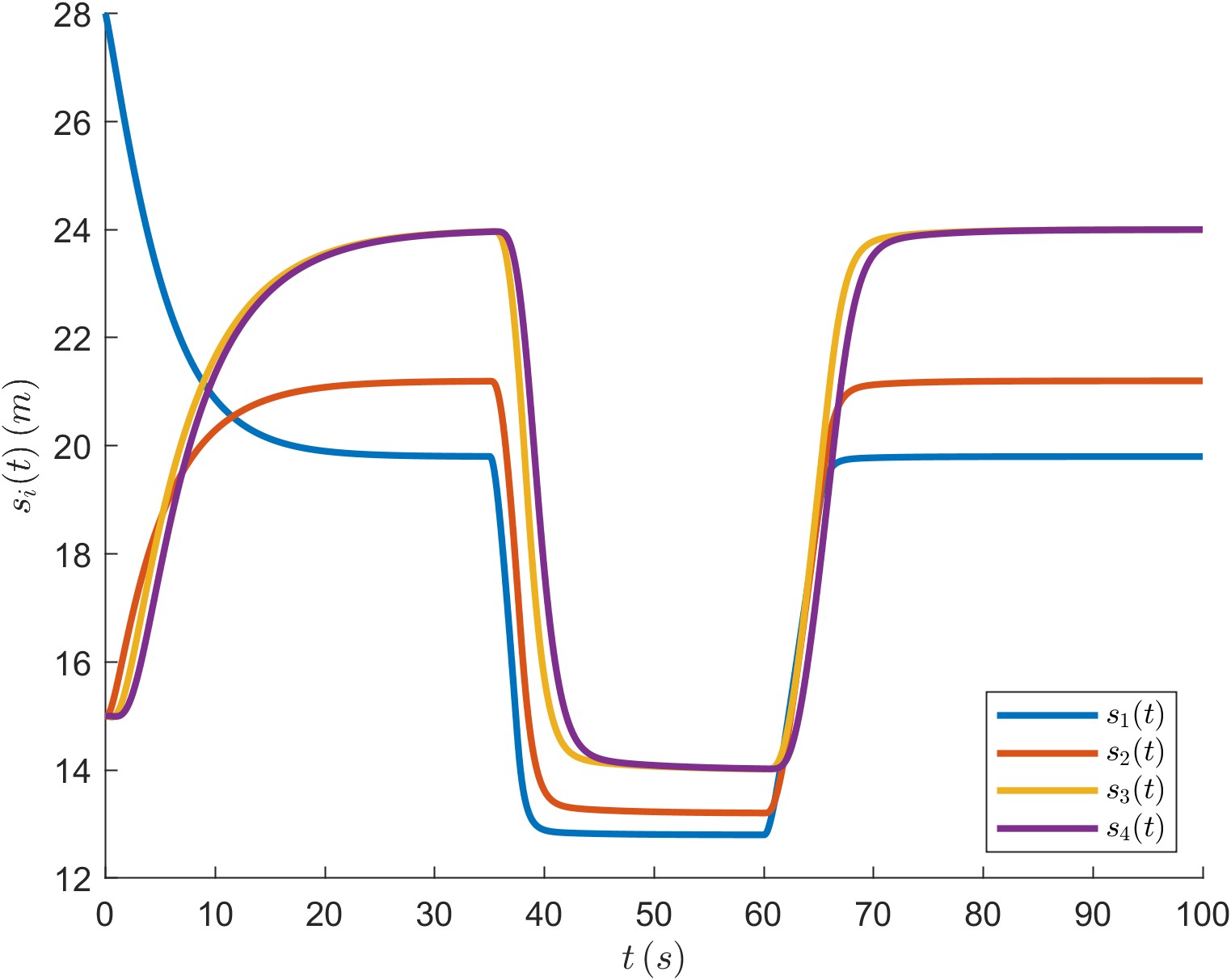}
		\caption{Acceleration (top left), speed (top right), and spacing (bottom) of four vehicles following a leader under the control law \eqref{GrindEQ__3_1_}.\vspace*{-\baselineskip}}\label{Fig7}
	\vspace*{-5pt}\end{center}\vspace*{-5pt}
\end{figure}
\subsubsection{Illustration of Robustness to Third-Order Vehicle Dynamics}
 We consider a platoon of  $n$ vehicles, each one modeled as follows 
\begin{equation}\label{third_order}
\begin{aligned}
	\dot{s}_i(t) =&\nobreakspace v_{i-1}(t) - v_i(t),\\	
	\dot{v}_i(t) =&\nobreakspace a_i (t),\\
	\dot{a}_i(t) =&\nobreakspace -\frac{1} {\tau_i} a_i (t)+\frac{1} {\tau_i}u_i(t),
\end{aligned}, \quad i=\ldots,n,
\end{equation}
where $a_i$ is vehicle acceleration and $\tau_i$ is lag. We consider a heterogeneous platoon of $n=4$ vehicles and a leader using the third-order model \eqref{third_order}, along with initial values outside of admissible invariance set $D$. Here, we apply directly the nonlinear controller \eqref{GrindEQ__3_1_} to the third-order model, and not the backstepping design of Remark \ref{remark:jerk}. We choose control gain $k_i=\frac{1.2}{h_i}$, $\psi(v)$ given by \eqref{psi:example},   $\sigma(H)$ given by \eqref{sigma:nonlinear}, for all~$i$, $v_{\rm max}=20$, and $\tilde{H}= \frac{v_{max}}{2}$. The initial conditions are $v_{i}(0) = 15 \left(\frac{m}{s} \right)$, $v_{0}(0) = 14 \left(\frac{m}{s} \right) $; $s_{i}(0) = 15 \nobreakspace (m)$, $s_{1}(0) = 28 \nobreakspace (m)$, and $a_{i}(0)= 0$ for all $i$. In the present scenario we consider a case in which $\tau_i$, $h_i$, and $r_i$ are set according to Table~\ref{table4}.  Figure~\ref{Fig7}, illustrates that the platoon remains string stable as there is no overshoot in the responses, due to deceleration or acceleration maneuvers performed by the leader.

In other words, the robustness of the nonlinear design is demonstrated: in the presence of the lag $\tau$, vehicle stability, safety, and string stability are preserved. In addition, we show that safety is maintained for initial conditions outside the invariant set $D$.

\begin{table}  [pos=h]
	\centering
	\caption{Parameters used for the simulation results in Figure~\ref{Fig7}.}
	\scalebox{1}{\begin{tabular}{||c| c c  c||} 
			\hline
			\backslashbox{Vehicle No.}{Parameters} & $\tau_{i}$ & $h_i$  & $r_i$ \\ 
			\hline\hline
			0&$0.05$ s & $-$ &  $-$\\ 
			\hline
			1&$0.05$ s & $0.7$ s &  $10$\\ 
			\hline
			2&$0.1$ s & $0.8$ s  & $10$\\
			\hline
			3&$0.15$ s & $1$ s  & $10$\\
			\hline
			4&$0.15$ s & $1$ s  & $10$\\
			\hline
	\end{tabular}}
	\label{table4}
\end{table}

\subsubsection{Comparison of Safety Robustness with Linear CTH Controllers}

We illustrate here the robustness of the safety properties of the platoon to initial conditions outside the constructed invariant sets under both, the linear CTH controller \eqref{CTH:linear}, with the invariant set given by the set $K$ defined by \eqref{GrindEQ__2_15_}, and the nonlinear CTH controller \eqref{GrindEQ__3_1_}, with the invariant set  given by the set $D$ defined in \eqref{GrindEQ__3_16_}. For illustration purposes we only consider $n=2$ identical vehicles in a scenario where all vehicles have low speeds and the leading vehicle comes to a stop, that is, $v^*=0$. We let the length of the vehicles be $\ell=4.5$,  gain $k=0.9$, time-headway headway  $h=1.2$,  $\beta_i=\frac{1}{h_i}\left(k_i-\frac{1}{h_i}\right)$, and $r=6$ which satisfy $k>1/h$, $r>\ell$. Moreover, we consider the initial conditions  $v_1(0)=1$, $v_2(0)=2$,   and $s_1(0)=\ell+0.1+hv_1(0)$, $s_2(0)=\ell+0.1+hv_2(0)$. Notice that the initial spacing $s_1(0)$ and $s_2(0)$ do not satisfy the constraint $s_i\ge r+hv_i$, $i=1,2$ of set $K$ (with $c=h$), while they satisfy the constraint $s_i>\ell_i+h_iv_i$ of the set $D$. Figure~\ref{Fig8} shows that, in this particular scenario, both vehicles attain negative speeds with the linear controller \eqref{CTH:linear}, while the speeds of the vehicles with the nonlinear CTH-based controller \eqref{GrindEQ__3_1_} are non-negative.
\begin{figure} [pos=t]
\centering
\includegraphics[width=0.5\linewidth]{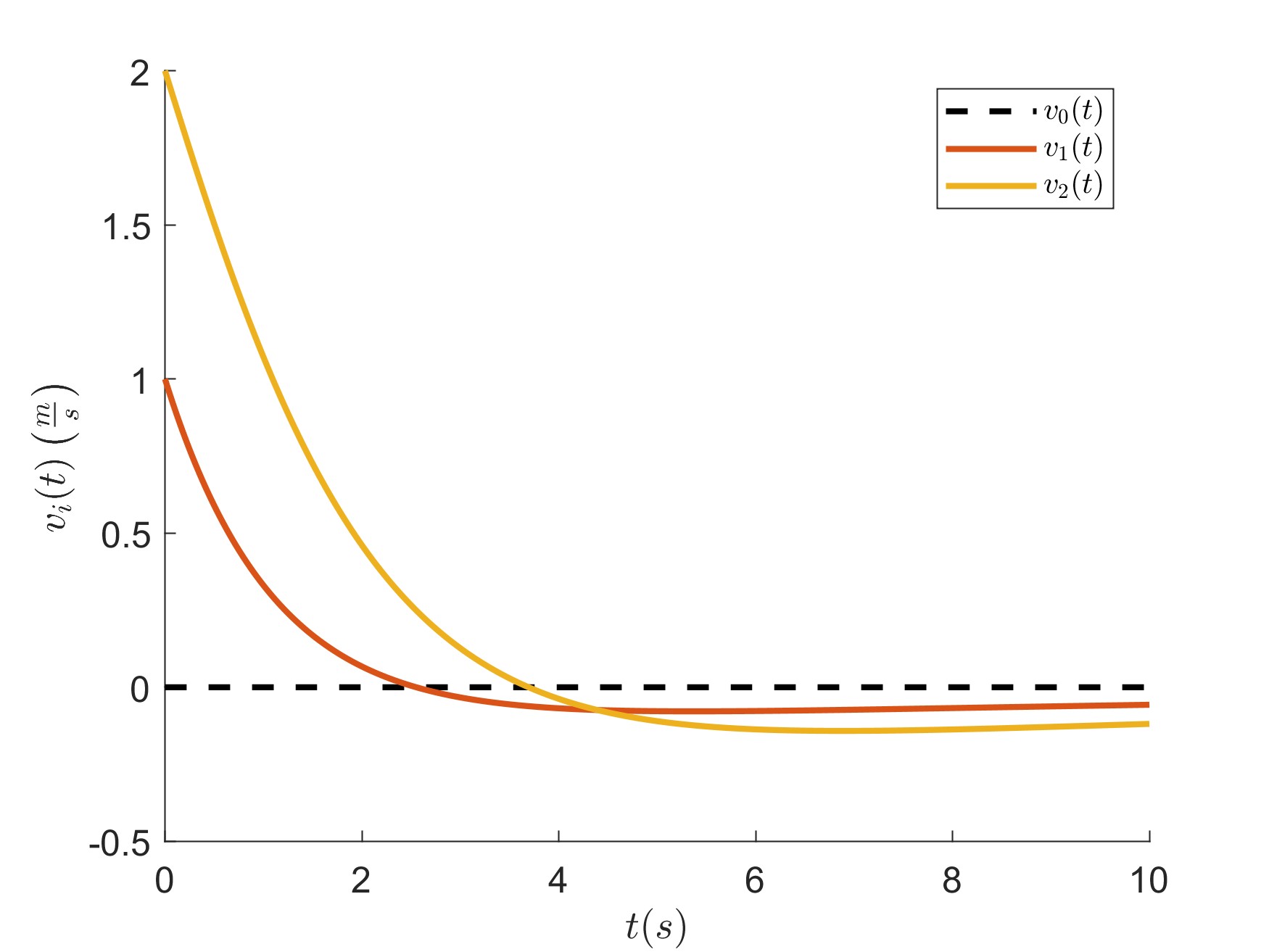}\hfil
\includegraphics[width=0.5\linewidth]{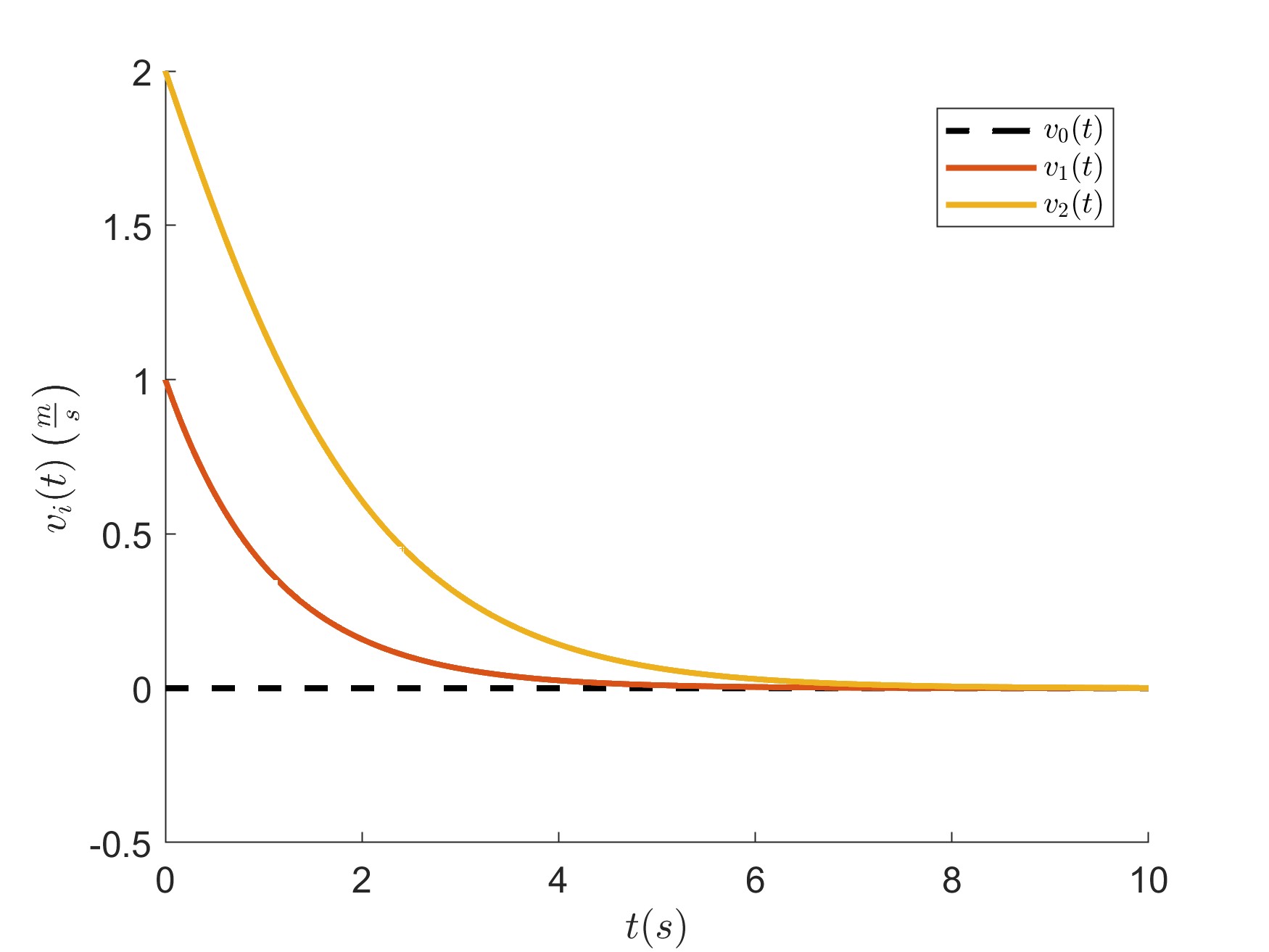}
\caption{Illustration of negative speeds (left) for the linear controller \eqref{CTH:linear} and non-negative speeds (right) with the nonlinear controller \eqref{GrindEQ__3_1_}.}
\label{Fig8}
\end{figure}
 
\pagebreak
\subsection{Comparison with Nonlinear Spacing Policy Controllers of Traffic Throughput}

We consider the controller from \cite{Karafyllis2023}:
\begin{equation} \label{GrindEQ__6_1_} 
u_{i} =F_{NL} \left(s_{i} ,v_{i-1} ,v_{i} \right)=(k-g(s_{i} ))\bar{G}(s_{i} )+g(s_{i} )v_{i-1} -kv_{i}  , \;\;i=1,...,n ,
\end{equation} 
where $\bar{G}(s):=\int _{\ell }^{s}g(l)dl $, $\ell >0$ is the length of the vehicle, $v_{0} \in C^{1} ({\mathbb R}_{+} )$ is the leading vehicle speed, satisfying $\dot{v}_{0} (t)\ge -kv_{0} (t)$, $v_{0} (t)\in (0,v_{\max } )$, for all $t\ge 0$ and $g:{\mathbb R}\to {\mathbb R}_{+} $ is a locally Lipschitz function that satisfies
\begin{align}
& 0<g(s)\le g_{\max }<k  , \textrm{ for }s>\lambda ,\label{GrindEQ__6_3_}\\
& v_{\max } :=\int _{\ell }^{+\infty }g(r)dr <k\left(\lambda -\ell \right) ,\label{GrindEQ__6_4_} \\
& g(s)=0, \textrm{ for }s\in [\ell ,\lambda ] ,\label{GrindEQ__6_5_}
\end{align}
where $\lambda >\ell $ defines a safety region in which the following vehicle starts decelerating at maximum rate (condition \eqref{GrindEQ__6_5_}). Under the above conditions, the nonlinear controller \eqref{GrindEQ__6_1_} guarantees that the following implication holds  $(s(0),v(0))\in \mathrm{{\mathcal D}}(v_{0} (0))\Rightarrow (s(t),v(t))\in \mathrm{{\mathcal D}}(v_{0} (t))$, $t\ge 0$ where $\mathcal{D}(v_0)$ is defined by \eqref{GrindEQ__3_14_}.  Moreover, the above conditions guarantee that a string of vehicles using \eqref{GrindEQ__6_1_} is $L^{2} $ and $L^{\infty } $ string stable, and, in addition, for given $v_{0} (t)\equiv v^{*} \in (0,v_{\max } )$, the equilibrium $(s^{*} 1_{n} ,v^{*} 1_{n} )$ is globally asymptotically stable, where the spacing equilibrium $s^{*} >\lambda $ is given by the solution of $v^{*} =\bar{G}(s^{*} )$, that is $s^{*}=\bar{G}^{-1}(v^*)$ (since $\bar{G}(s)$ is increasing for $s>\ell$).  

Similarly to \cite{Karafyllis2023}, we consider $g$ as follows 
\begin{equation} \label{GrindEQ__6_6_} 
g(s)=\left\{\begin{array}{cc} {0,} & {s\le \lambda }, \\ {L(s-\lambda )}, & {\lambda <s\le \lambda +g_{\max } L^{-1} }, \\ {g_{\max } }, & {\lambda +g_{\max } L^{-1} <s\le \gamma }, \\ {g_{\max } \exp (-\alpha (s-\gamma ))}, & {s>\gamma } \end{array}\right.  
\end{equation} 
where
\begin{equation}\label{NL:par}
\gamma =\lambda +\frac{v_{\max } }{g_{\max } } -\frac{1}{\alpha} +\frac{g_{\max } }{2L}, \quad {L>\frac{g_{\max }^{2} }{2\left(v_{\max } -g_{\max } \alpha ^{-1} \right)} },  \quad {\alpha >\frac{g_{\max } }{v_{\max } } },\quad \textrm{and} \;\; g_{\max}=\eta k, \;\eta\in(0,1).
\end{equation}

The parameters $L$ and $\alpha $ are introduced to shape the growth and decay of $g$. Larger values of $L$ and $\alpha $ allocate a larger fraction of the total area $\int_{a}^{\infty } g\left(s\right)\, ds=v_{{\rm max}} $ to the  region $(\lambda,\lambda+g_{\max}L^{-1}]$. Since $v^{*} =\bar{G}(s^{*} )$, increasing $L$ causes $\bar{G}\left(s\right)$ to accumulate area more rapidly after $s=\lambda $. Therefore, for a given equilibrium speed $v^{*} $, the corresponding spacing $s^{*} =\bar{G}^{-1} \left(v^{*} \right)$ is reduced. In simpler words, a steeper growth of $g$ yields a more spacing-efficient equilibrium relation. We selected $g_{\max}=\eta k$, $\eta \in(0,1)$ so that  the maximum nonlinear feedback increases proportionally with the controller gain, while ensuring that the condition \eqref{GrindEQ__6_3_} holds. Since $\bar{G}$ is increasing for $s>\ell $, we can obtain via \eqref{GrindEQ__6_6_} and definitions $g_{\max}=\eta k$, $\eta\in(0,1)$ and $\gamma$, above, that the equilibrium spacing $s^*$ as a function of the equilibrium speed $v^*$, parameterized by the controller gain $k$ is 
\begin{equation}
s^{*}(v^*;k)=\left\{\begin{array}{cc} {\lambda +\sqrt{2v^{*} L^{-1} } }, & {v^{*} \in \left(0,(\eta k)^{2} L^{-1} /2\right)} \\ 
{\lambda +\frac{\eta k }{2L} +\frac{v^{*} }{\eta k } }, & {v^{*} \in \left((\eta k)^{2} L^{-1} /2,\bar{G}(\gamma)\right)} \\ {\gamma -\frac{1}{\alpha } \ln \left(1-\frac{\alpha(v^{*} -\bar{G}(\gamma))}{\eta k } \right)}, & {v^{*} \in \left(\bar{G}(\gamma),v_{\max } \right)} \end{array}\right. 
\end{equation}
where, for notational convenience,
\begin{equation}
\bar{G}(\gamma;k)=\bar{G}(\gamma)=\eta k\left(\gamma -\lambda \right)-\frac{(\eta k)^{2} }{2L} =v_{\max}-\frac{\eta k}{\alpha}.
\end{equation}
Notice that for $v^{*} \in (\bar{G}(\gamma ;k),v_{\max } )$ it follows that $\alpha\frac{v^{*} -\bar{G}(\gamma;k )}{\eta k } \in (0,1)$. Thus, in all cases we have that $s^{*} (v^*;k) >\lambda $, for every $v^*\in (0,v_{\max})$ and $k>0$. 

It should be noted that conditions \eqref{GrindEQ__6_3_}-\eqref{GrindEQ__6_5_} imply that the controller \eqref{GrindEQ__6_1_} is bounded in $\mathcal{D}(v_0)$, that is, $|F_{NL}(s_i,v_{i-1},v)|\leq k v_{\max}$. Hence, increasing $k$ may result in higher acceleration profiles. A similar observation applies to the nonlinear CTH-based controller  \eqref{GrindEQ__3_1_}  with \eqref{sigma:nonlinear}. Indeed, in the invariant set $D$ given by \eqref{GrindEQ__3_6_}, we have that $|u_i|\leq\frac{1}{h}v_{\max}+ \beta \max_{v\in[0,v_{\max}]}\{\chi(v)\}\tilde{H}$. To facilitate a fair comparison between the two controllers, we select the parameters of the nonlinear CTH controller so that its worst-case acceleration bound coincides with that of the nonlinear controller \eqref{GrindEQ__6_1_}. Specifically, using $\beta=\frac{1}{h}\left(k-\frac{1}{h}\right)$, $k>1/h$, $\psi(v)=v/v_{\max}$, and choosing $\tilde{H}= 4h$, yields  $|u_i|\leq k v_{\max}$. Therefore, in both controllers increasing $k$ may lead to proportionally larger acceleration values. These estimates, however, are conservative and merely provide worst-case bounds. The actual acceleration profiles depend on the leader maneuver and on the initial spacing and relative speeds   conditions.

The safety threshold $\lambda$ of the controller $F_{NL}$ is chosen as a function of the gain $k$, $\lambda(k)=\ell+\frac{v_{\max}}{k}+0.1$. By having $\lambda$ to depend on $k$, we simultaneously guarantee that \eqref{GrindEQ__6_4_} holds, and the spacing equilibrium is reduced for increasing values of $k$ for the controller \eqref{GrindEQ__6_1_}. We use $\eta=0.95$, $\ell=4.5 (m)$, $r=7 (m)$, $v_{\max}=40 (m/s)$, and select  $L=\max(10^2,(\eta k)^2/(v_{\max}-\eta k\alpha^{-1}))$ and $\alpha = \max(10^2,\eta k/v_{\max})$ so that conditions \eqref{NL:par} are satisfied for every value of $k$. For the nonlinear CTH-based controller \eqref{GrindEQ__3_1_} we select $h(k) =1.1/k$ which satisfies $k>1/h $.  

Since the equilibrium spacing depends on the desired equilibrium speed, comparing the controllers at a single $v^*$ may not be representative. We therefore introduce the average equilibrium spacing over the admissible equilibrium-speed range, $v^*\in(0,v_{\max})$, as a single performance metric for each controller gain $k$:
\begin{equation}\label{mean_eq}
\bar{s}(k)=\frac{1}{v_{\max}}\int_0^{v_{\max}}s^*(v;k)dv.
\end{equation}
Since the comparison is performed over the equilibrium speed range $v^{*}\in(0,v_{\max})$, the average spacing of the CTH-based controller \eqref{GrindEQ__3_1_} is obtained from its equilibrium relation
$s ^* =r  + h(k) v^{*} $ (see \eqref{GrindEQ__5_1_}) yielding 
$\bar{s}(k)=r+\frac{v_{\max}}{2}h(k)$.

\begin{figure}[pos=t]
\centering
\includegraphics[width=0.47\linewidth]{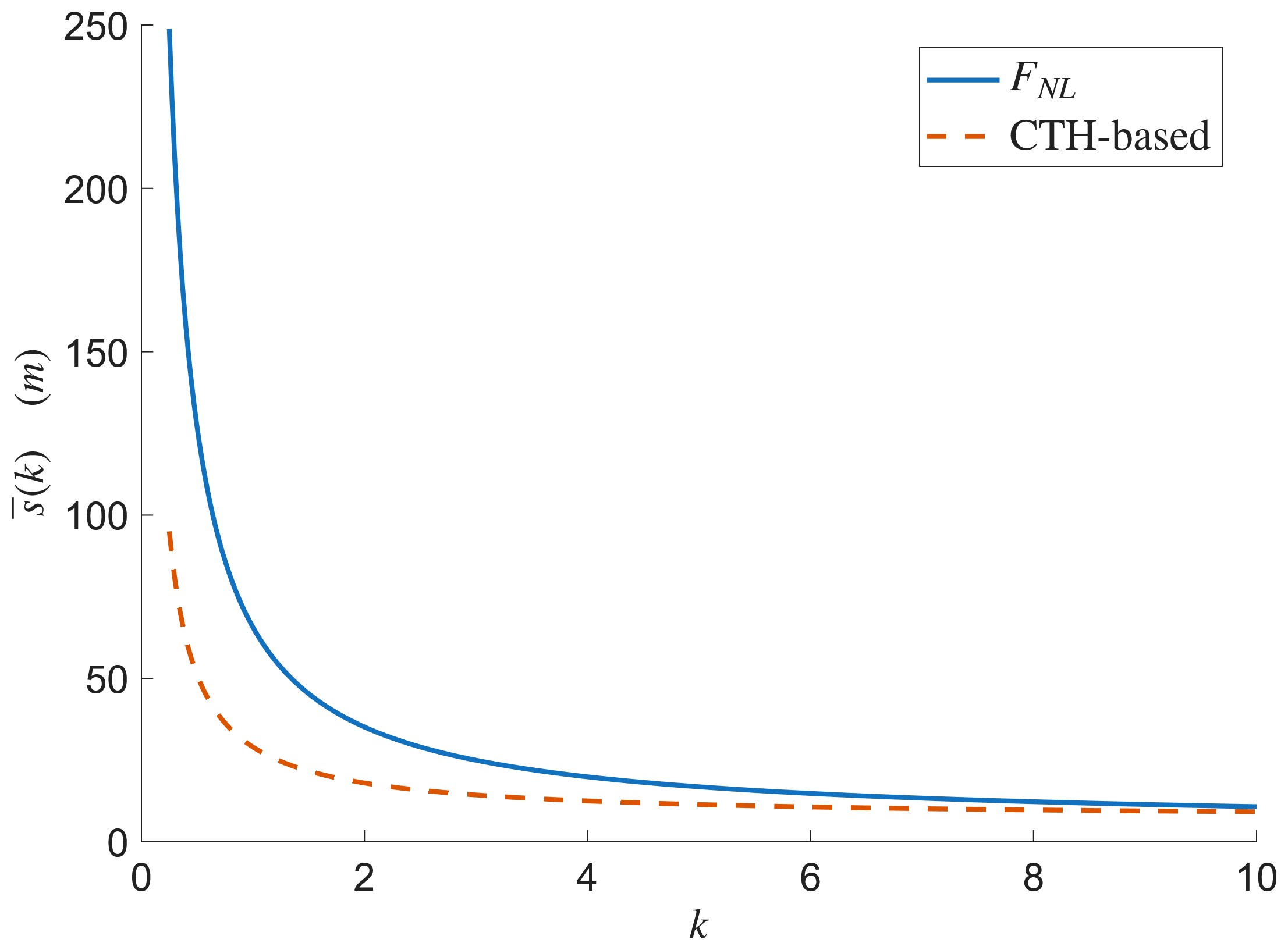}
\includegraphics[width=0.47\linewidth]{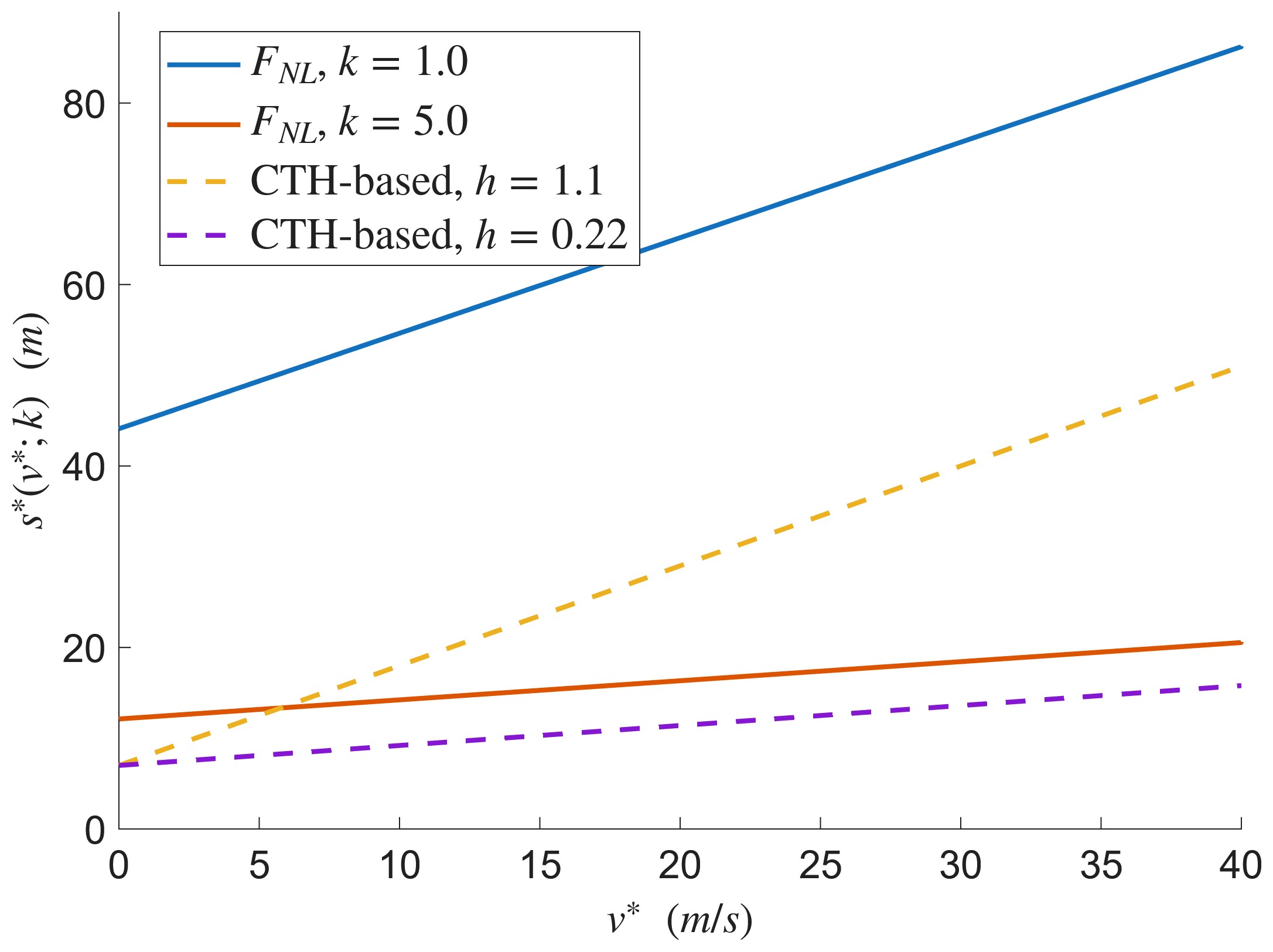}
\caption{Equilibrium spacing comparison between the nonlinear controller  \eqref{GrindEQ__6_1_} and the nonlinear CTH-based controller \eqref{GrindEQ__3_1_}. (Left) Average equilibrium spacing versus the controller gain $k$. (Right) Equilibrium spacing as a function of the desired speed for $k=1$, $k=5$,  $h=1.1$, and $h=0.22$.}
\label{fig:comp}
\end{figure} 
Figure \ref{fig:comp} (left) shows that both average spacings decrease as $k$ increases. For small and moderate values of $k$, the CTH-based controller \eqref{GrindEQ__3_1_} gives smaller average spacing. As $k$ increases, the gap decreases, and the two controllers have comparable spacing requirements. The  results also indicate that increasing $k$ improves the spacing efficiency of both controllers while reducing the performance gap between them. These results suggest that, over the considered range of gains, controller \eqref{GrindEQ__3_1_} can accommodate smaller equilibrium spacings on average and therefore may achieve higher equilibrium traffic throughput. Figure \ref{fig:comp}  (right) further illustrates the equilibrium-spacing relations for two representative gain values, with the CTH time headway selected according to $h(k) =1.1/k$. For $k=1$, the CTH-based controller requires substantially smaller spacing over the entire equilibrium-speed range. For $k=5$, the spacing requirements for both controllers become comparable.

The equilibrium-spacing comparison demonstrates that the proposed CTH-based controller generally requires smaller equilibrium spacing than the nonlinear controller \eqref{GrindEQ__6_1_}, suggesting improved equilibrium traffic throughput.  In contrast, the nonlinear controller \eqref{GrindEQ__6_1_} admits a larger positively invariant state space and provides the additional capability of shaping the equilibrium traffic fundamental diagram through the choice of the nonlinear spacing policy, see \cite{Karafyllis2023}.

\section{Proofs}\label{sec:proofs}

In this section we provide the proofs of all results of the paper.\\

\begin{proof}[Proof of Theorem \ref{thm:1}]  Let $\left(s(0),v(0)\right)\in K$. Since the closed-loop system \eqref{GrindEQ__2_1_}, \eqref{GrindEQ__2_2_}  is a linear system with continuous input $v_0\in C^{0}(\mathbb{R}_+;\mathbb{R}_+)$, the corresponding initial value problem admits a unique solution $(s(t),v(t))\in {\mathbb R}^{2n} $ that is defined for all $t\ge 0$. We will show that $(s(t),v(t))\in K$ for all $t\ge 0$. Define 
\begin{equation} \label{GrindEQ__7_1_} 
Q_{i} =s_{i} -r_i-c_i v_{i}  , i=1,\ldots,n 
\end{equation} 
Then, since, $\left(s(0),v(0)\right)\in K$, definition \eqref{GrindEQ__2_15_} and \eqref{GrindEQ__7_1_} imply that $v_{i} (0)\ge 0$ and $Q_{i} (0)\ge 0$ for $i=1,\ldots,n$. Using \eqref{GrindEQ__2_1_}, \eqref{GrindEQ__2_2_}, and \eqref{GrindEQ__7_1_} we obtain
\begin{equation} \label{GrindEQ__7_2_} 
\begin{aligned} 
&{\dot{v}_{i} =  k_{v,i} v_{i-1} + k_{p,i} Q_{i} -\left( k_{v,i}- k_{p,i} (c_i-h_i)\right)v_{i} } \\
&{\dot{Q}_{i} =\left(1-c_i k_{v,i}\right)v_{i-1} +\left(c_i k_{v,i}-1+c_i k_{p,i}(h_i-c_i)\right)v_{i} -c_i k_{p,i} Q_{i} } \end{aligned} ,\;\; i=1,\ldots,n 
\end{equation} 
Define along the solution,
\begin{equation} \label{GrindEQ__7_4_} 
W(t)=\frac{1}{2} \sum _{i=1}^{n}\left(\left(v_{i}^{-} (t)\right)^{2} +\left(Q_{i}^{-} (t)\right)^{2} \right)  , t\ge 0 .
\end{equation} 
Notice that $W\in C^{1}(\mathbb{R}_+;\mathbb{R}_+)$. In the following we omit the time-dependence for brevity. We get from \eqref{GrindEQ__7_2_}, \eqref{GrindEQ__7_4_}, and the identity $\frac{d}{dt}(x^-)^2=-x^-\dot{x}$, that
\begin{equation} \label{GrindEQ__7_5_} 
\begin{aligned}  \dot{W}&=-\sum _{i=1}^{n}v_{i}^{-} \left( k_{v,i} v_{i-1} +k_{p,i} Q_{i} -\left( k_{v,i}- k_{p,i} (c_i-h_i)\right)v_{i} \right) \\
&-\sum _{i=1}^{n}Q_{i}^{-} \left(\left(1-c_i k_{v,i}\right)v_{i-1} +\left(c_i k_{v,i}-1+c_ik_{p,i}(h_i-c_i)\right)v_{i} -c_i k_{p,i} Q_{i} \right)   .\end{aligned} 
\end{equation} 
Notice that assumption \eqref{GrindEQ__2_16_}  gives that $ k_{v,i}-k_{p,i}(c_i-h_i)\ge  k_{v,i}+ k_{p,i}\left(h_i-\frac{1}{ k_{v,i}}\right)\ge  k_{v,i}>0$. Using the inequalities $-x^{-} y\le x^{-} y^{-} $ and $x^{-} x=-\left(x^{-} \right)^{2} $ (that hold for all $x,y\in {\mathbb R}$)\footnote{Since  $x^-\ge0$ and $y^-\ge-y$ for all $x,y\in\mathbb{R}$ it follows that $-x^-y\leq x^-y^-$. For the second property, for $x\ge0$ we have $x^-=0$. Thus, $x^- x=0=-(x^-)^2$. For $x<0$, we have that $x^-=-x$. Thus, $x^-x=-xx=-x^2=-(-x)^2=-(x^-)^2$.}, we get from   \eqref{GrindEQ__7_5_} and the fact that $c_i,k_{p,i}>0$, that 
\begin{equation} \label{GrindEQ__7_6_} 
\begin{aligned}   
&\dot{W}\le  \sum _{i=1}^{n} k_{v,i}v_{i}^{-} v_{i-1}^{-}  + \sum _{i=1}^{n} k_{p,i}v_{i}^{-} Q_{i}^{-}  +\sum _{i=1}^{n}\left(1-c_i k_{v,i} \right)Q_{i}^{-} v_{i-1}^{-}   \\
&+\sum _{i=1}^{n}\left(c k_{v,i}-1+c_ik_{p,i}(h_i-c_i)\right)Q_{i}^{-} v_{i}^{-}  -\sum _{i=1}^{n}\left( k_{v,i}- k_{p,i} (c_i-h_i)\right)\left(v_{i}^{-} \right)^{2}  -\sum _{i=1}^{n}c_ik_{p,i}\left(Q_{i}^{-} \right)^{2}  \\ 
&\le  \sum _{i=1}^{n} k_{v,i}v_{i}^{-} v_{i-1}^{-}    +\sum _{i=1}^{n}a_iQ_{i}^{-} v_{i-1}^{-}  +\sum _{i=1}^{n}b_iQ_{i}^{-} v_{i}^{-}   \end{aligned} 
\end{equation} 
where $a_i = \left(1-c_i k_{v,i} \right)$, $b_i= \left(c_i k_{v,i}-1 + c_ik_{p,i}(h_i-c_i) +k_{p,i}\right)$, $i=1,\ldots,n$. Notice that due to conditions \eqref{GrindEQ__2_16_}, we have that $a_i>0$ and $b_i>0$ for $i=1,\ldots,n$.
From Young's inequality $ab\le \frac{1}{2} a^{2} +\frac{1}{2} b^{2} $, $a,b\in {\mathbb R}$, the fact that  $v_{0}^{-} =0$ (since $v_{0} (t)\ge 0$), and the notational convention $k_{v,n+1}=0$, $a_{n+1}=1$, we have
\begin{equation} \label{GrindEQ__7_7_} 
\begin{aligned} &{  \sum _{i=1}^{n}k_{v,i}v_{i}^{-} v_{i-1}^{-}  \le  \sum _{i=1}^{n}\frac{ k_{v,i}}{2}\left(v_{i}^{-} \right)^{2}  +\sum _{i=1}^{n}\frac{ k_{v,i}}{2} \left(v_{i-1}^{-} \right)^{2}  =\sum _{i=1}^{n}\frac{ k_{v,i}}{2} \left(v_{i}^{-} \right)^{2}  + \sum _{i=0}^{n-1}\frac{ k_{v,i+1}}{2}\left(v_{i}^{-} \right)^{2}  } \\ 
& \le  \sum _{i=1}^{n}\frac{ k_{v,i}}{2}\left(v_{i}^{-} \right)^{2}  + \sum _{i=1}^{n}\frac{ k_{v,i+1}}{2}\left(v_{i}^{-} \right)^{2}  ,\end{aligned} 
\end{equation} 
\begin{equation} \label{GrindEQ__7_8_} 
v_{i}^{-} Q_{i}^{-}  \le \frac{1}{2} \left(v_{i}^{-} \right)^{2}  + \frac{1}{2}\left(Q_{i}^{-} \right)^{2}   ,\,  i=1,\ldots,n,
\end{equation} 
and
\begin{equation} \label{GrindEQ__7_9_} 
\sum _{i=1}^{n}a_iQ_{i}^{-} v_{i-1}^{-}  \le \frac{1}{2} \sum _{i=1}^{n}a_i\left(v_{i-1}^{-} \right)^{2}  +\frac{1}{2}\sum _{i=1}^{n}a_i \left(Q_{i}^{-} \right)^{2}  \le \frac{1}{2} \sum _{i=1}^{n}a_i\left(Q_{i}^{-} \right)^{2}  +\frac{1}{2} \sum _{i=1}^{n}a_{i+1}\left(v_{i}^{-} \right)^{2}   .
\end{equation} 
Combining \eqref{GrindEQ__7_6_} with \eqref{GrindEQ__7_7_}, \eqref{GrindEQ__7_8_}, and \eqref{GrindEQ__7_9_} we get
\begin{equation} \label{GrindEQ__7_10_} 
\dot{W}\le \frac{q}{2} \sum _{i=1}^{n}\left(\left(v_{i}^{-} \right)^{2}  + \left(Q_{i}^{-} \right)^{2}\right)
\end{equation} 
where $q=\max_{i=1,\ldots,n}\{k_{v,i}+k_{v,i+1}+b_i+a_{i+1}, a_i+b_i\}>0$,
with $k_{v,n+1}=0$, $a_{n+1}=1$. Next, using \eqref{GrindEQ__7_10_} and definition \eqref{GrindEQ__7_4_} we get
\begin{equation} \label{GrindEQ__7_11_} 
\dot{W}(t)\le qW(t), t\ge 0 
\end{equation} 
Since $\left(s(0),v(0)\right)\in K$, it follows that $W(0)=0$ and due to Gronwall's inequality (see \cite{Khalil2002}), \eqref{GrindEQ__7_11_} implies that $W(t)=0$ for all $t\ge 0$. The latter and definitions \eqref{GrindEQ__7_1_} and \eqref{GrindEQ__7_4_} give that $v_{i} (t)\ge 0$ and $s_{i} (t)\ge r+cv_{i} (t)$  for $i=1,\ldots,n$ and $t\ge 0$. Thus $(s(t),v(t))\in K$ for all $t\ge 0$. The proof is complete. 
\end{proof}
 
\begin{proof}[Proof of Theorem \ref{thm:2}]
{Let $\left(s_{1,0} ,\ldots,s_{n,0} ,v_{1,0} ,\ldots,v_{n,0} \right)\in D$ be given. Since $\psi,\sigma \in C^{1}  $ and $v_0\in C^0$, by standard local existence and uniqueness theorems (see \cite{Khalil2002}), there} exists a unique maximal solution $\left(s\left(t\right),v\left(t\right)\right)\in \mathrm{{\mathbb R}}^{2n} $, of the initial value problem \eqref{GrindEQ__2_1_}, \eqref{GrindEQ__3_1_} with initial conditions $(s(0),v(0))=\left(s_{1,0} ,\ldots,s_{n,0} ,v_{1,0} ,\ldots,v_{n,0} \right)\in D$ defined on $[0,T_{\max } )$. If  $T_{\max} <+\infty $, then ${\mathop{\lim \sup }\limits_{t\to T_{\max }^{-} }} \left(\left|(s(t),v(t))\right|\right)=+\infty $. 

\noindent Define now
\begin{equation} \label{GrindEQ__7_12_} 
\tilde{\chi }(v)=\left\{\begin{array}{cc} {0} & {v<0} \\ {(v_{\max } -v)\psi (v)} & {0\le v\le v_{\max } } \\ {0} & {v>v_{\max } } \end{array}\right.  
\end{equation} 
Since $\chi \in C^{1} \left([0,v_{\max } ]\right)$ with $\chi (0)=\chi (v_{\max } )=0$, the continuous extension $\tilde{\chi }$, is globally Lipschitz and bounded, that is, there exists $C_{\chi } >0$ such that $\left|\tilde{\chi }(v)\right|\le C_{\chi } $, $v\in {\mathbb R}$. Consider also the system
\begin{equation} \label{GrindEQ__7_13_} 
\begin{aligned} &{\dot{s}_{i} =v_{i-1} -v_{i} } \\ 
&{\dot{v}_{i} =\frac{1}{h_{i} } \left(v_{i-1} -v_{i} \right)+\beta _{i} \tilde{\chi }(v_{i} ) \sigma(H_{i})}  \end{aligned}, \quad i=1,\ldots,n 
\end{equation} 
where $(s,v)\in {\mathbb R}^{2n} $ and the vector field of \eqref{GrindEQ__7_13_} is globally defined and locally Lipschitz on ${\mathbb R}^{2n} $. Since $\tilde{\chi }$ is bounded, definition \eqref{GrindEQ__3_2_} and property \eqref{s3} give for $i=1,\ldots,n$ that
\[\left|\dot{s}_{i} \right|\le \left|v_{i-1} \right|+\left|v_{i} \right|,\] 
\[
\begin{aligned} 
 \left|\dot{v}_{i} \right|&\le \frac{1}{h_{i} } \left(\left|v_{i-1} \right|+\left|v_{i} \right|\right)+\beta _{i} C_{\chi } \left|\sigma(H_i)\right|  \le \frac{1}{h_{i} } \left(\left|v_{i-1} \right|+\left|v_{i} \right|\right)+\beta _{i} C_{\chi } \left| H_i \right|  \\ 
&\le \frac{1}{h_{i} } \left(\left|v_{i-1} \right|+\left|v_{i} \right|\right)+\beta _{i} C_{\chi } \left(\left|s_{i} \right|+r_{i} +h_i\left|v_{i} \right|\right) {=\beta _{i} C_{\chi } \left|s_{i} \right|+\frac{1}{h_{i} } \left|v_{i-1} \right|+\left(\frac{1}{h_{i} } +\beta _{i} C_{\chi } \right)\left|v_{i} \right|+\beta _{i} C_{\chi } r_{i} }. \end{aligned}
\] 
Taking into account the previous inequalities and assumption \eqref{GrindEQ__3_4_} we can conclude that system \eqref{GrindEQ__7_13_} satisfies a linear growth condition. Thus by \cite{Khalil2002}, the solution $(s(t),v(t))\in {\mathbb R}^{2n} $ of the initial value problem \eqref{GrindEQ__7_13_} with $(s(0),v(0))=\left(s_{1,0} ,\ldots,s_{n,0} ,v_{1,0} ,\ldots,v_{n,0} \right)\in D$ exists for all $t\ge 0$. We now prove that this global solution satisfies $0\le v_{i} (t)\le v_{\max } $ for all $t\ge 0$ and $i=1,\ldots,n$.

Let $T>0$ (arbitrary), $i\in \left\{1,\ldots,n\right\}$, and assume that 
\begin{equation}\label{GrindEQ__7_14_}
0\le v_{i-1} (t)\le v_{\max }\textrm{ for }t\in [0,T].
\end{equation} 
 
Define 
\begin{equation} \label{GrindEQ__7_15_} 
E_i(t)=\frac{1}{2}\left(v_{i}^{-} (t)\right)^{2} + \frac{1}{2}\left((v_{i} (t)-v_{\max } )^{+} \right)^{2} , t\in [0,T] ,
\end{equation} 
which is $C^{1} $ and satisfies $E_i(t)=0$ if and only if $0\le v_{i} (t)\le v_{\max } $, $t\in [0,T]$. In the following we omit the time-dependence for brevity. Using \eqref{GrindEQ__7_13_} and \eqref{GrindEQ__7_15_} we have
\begin{equation} \label{GrindEQ__7_16_} 
\begin{aligned} &{\dot{E}_i=-v_{i}^{-} \dot{v}_{i} +\left(v_{i} -v_{\max } \right)^{+} \dot{v}_{i} } \\ 
&=\frac{1}{h_i} \left(-v_{i}^{-} +\left(v_{i} -v_{\max } \right)^{+} \right)\left(v_{i-1} -v_{i} \right)+\left(-v_{i}^{-} +\left(v_{i} -v_{\max } \right)^{+} \right)\beta _{i} \tilde{\chi }(v_{i} ) {\sigma(H_{i})} \end{aligned} 
\end{equation} 
Using the properties  $x^{-} x=-\left(x^{-} \right)^{2} $ and $x^{+} x=\left(x^{+} \right)^{2} $ (that hold for all $x\in {\mathbb R}$)\footnote{For $x\ge0$ we have $x^+=x$. Thus, $x^+x=x^2=(x^+)^2$. For $x<0$, we have that $x^+=0$. Thus, $x^+x=0 =(x^+)^2$.}  and by taking into account \eqref{GrindEQ__7_14_} we get
\begin{equation} \label{GrindEQ__7_17_} 
\begin{aligned} &{\left(-v_{i}^{-} +\left(v_{i} -v_{\max } \right)^{+} \right)\left(v_{i-1} -v_{i} \right)=-v_{i}^{-} v_{i-1} +v_{i}^{-} v_{i} +\left(v_{i} -v_{\max } \right)^{+} \left(v_{i-1} -v_{i} \right)} \\ &{\le -\left(v_{i}^{-} \right)^{2} -\left(v_{i} -v_{\max } \right)^{+} \left(v_{i} -v_{\max } \right) =-\left(v_{i}^{-} \right)^{2} -\left(\left(v_{i} -v_{\max } \right)^{+} \right)^{2} \le 0}. \end{aligned} 
\end{equation} 
Moreover, we have 
\begin{equation} \label{GrindEQ__7_18_} 
\begin{aligned} &{\left|\left(-v_{i}^{-} +\left(v_{i} -v_{\max } \right)^{+} \right)\tilde{\chi }(v_{i} )\right|\le \left(\left|-v_{i}^{-} \right|+\left|\left(v_{i} -v_{\max } \right)^{+} \right|\right)\left|\tilde{\chi }(v_{i} )\right|} \\ &{=\left(v_{i}^{-} +\left(v_{i} -v_{\max } \right)^{+} \right)\left|\tilde{\chi }(v_{i} )\right|}. \end{aligned} 
\end{equation} 
Since $\tilde{\chi }$ is globally Lipschitz, with $\tilde{\chi }(0)=\tilde{\chi }(v_{\max } )=0$, there exists a constant $L>0$ such that 
\begin{equation} \label{GrindEQ__7_19_} 
\left|\tilde{\chi }(v)\right|\le L\left|v\right|  \textrm{ and }   \left|\tilde{\chi }(v)\right|=\left|\tilde{\chi }(v)-\tilde{\chi }(v_{\max } )\right|\le L\left|v-v_{\max } \right| .
\end{equation} 
Hence, using the properties $\left|x\right|=x^{-} +x^{+} $, $x^{-} \left|x\right|=\left(x^{-} \right)^{2} $ and $x^{+} \left|x\right|=\left(x^{+} \right)^{2} $ (that hold for all $x\in\mathbb{R}$)\footnote{For $x\ge0$, we have that $x^-=0$, $x^+=x$, and $|x|=x$. Thus, $x^-+x^+=x=|x|$, $x^-|x|=0=(x^-)^2$, and $x^+|x|=x^2=(x^+)^2$. While for $x<0$ we have that $x^-=x$, $x^+=0$, and $|x|=-x$. Thus, $x^-+x^+=-x=|x|$, $x^-|x|=(-x)(-x)=(-x)^2=(x^-)^2$, and $x^+|x|=0=(x^+)^2$.}, we get
\begin{equation} \label{GrindEQ__7_20_} 
\begin{aligned} &{v_{i}^{-} \left|\tilde{\chi }(v_{i} )\right|\le Lv_{i}^{-} \left|v_{i} \right|=L\left(v_{i}^{-} \right)^{2} },  \\ &{\left(v_{i} -v_{\max } \right)^{+} \left|\tilde{\chi }(v_{i} )\right|\le L\left(\left(v_{i} -v_{\max } \right)^{+} \right)^{2} } .\end{aligned} 
\end{equation} 
{Combining \eqref{s3}, \eqref{GrindEQ__7_20_}, \eqref{GrindEQ__7_17_}, \eqref{GrindEQ__7_18_}, and \eqref{GrindEQ__7_16_} we get
\[ {\dot{E}_i\le \beta _{i} \left|-v_{i}^{-} +\left(v_{i} -v_{\max } \right)^{+} \right|\left|\tilde{\chi }(v_{i} )\right|\left|\sigma(H_{i})\right|} {\le \beta _{i} L\left|H_{i}  \right|\left(\left(v_{i}^{-} \right)^{2} +\left(\left(v_{i} -v_{\max } \right)^{+} \right)^{2} \right)}.\] }
Thus, definition \eqref{GrindEQ__7_15_} gives
\[\dot{E}_i(t)\le 2\beta _{i} L\left|H_{i} (t)\right|E_i(t),\;\; t\in [0,T].\] 
 Since $(s(0),v(0))\in D$, we have $v_{i} (0)\in [0,v_{\max } ]$, and consequently, $E_i(0)=0$. By Gronwall's inequality (see \cite{Khalil2002}) it follows that
\begin{equation} \label{GrindEQ__7_21_} 
E_i(t)=0,\, \, \, t\in [0,T]. 
\end{equation} 
Hence, $0\le v_{i} (t)\le v_{\max } $ for $t\in [0,T]$. Since assumption \eqref{GrindEQ__3_4_} holds, a simple induction argument yields $0\le v_{i} (t)\le v_{\max } $ for all $t\in [0,T]$ and $i=1,\ldots,n$. 
Finally, since $T>0$ was arbitrary, it follows that 
\begin{equation} \label{GrindEQ__7_22_} 
0\le v_{i} (t)\le v_{\max }  , \;\;\; t\ge 0 ,\; i=1,\ldots,n ,
\end{equation} 
for the solution of the extended system \eqref{GrindEQ__7_13_}.

{Using \eqref{GrindEQ__7_13_} and definition \eqref{GrindEQ__3_2_}, we get
\begin{equation}\label{dotH}
\dot{H}_i=\dot{s}_i-h_i \dot{v}_i =-h_i\beta_i\tilde{\chi}(v_i)\sigma(H_i)
\end{equation}
and therefore from \eqref{s2} and since $h_i$, $\beta_i$, $\tilde{\chi}\ge0$, we have that
\begin{equation}
H_i\dot{H}_i = -h_i\beta_i \tilde{\chi}H_i\sigma(H_i)\leq0
\end{equation} 
The latter implies that
\begin{equation}
\frac{d}{dt}\left(\frac{1}{2}H_i^2(t)\right)=H_i(t)\dot{H}_i(t)\leq0,
\end{equation}  
and thus,
\begin{equation}\label{Hi_L}
|H_i(t)|\leq |H_i(0)|,\; i=1,\ldots,n,\;\;t\ge0.
\end{equation}}

We next show that \eqref{GrindEQ__3_8_} holds. To that end, we show first that for any $H_i(0)\in \mathbb{R}$, $H_i(t)$ does not change sign. Suppose on the contrary, that $H_i$ changes sign. Then, there exists a time $t_0 >0$ such that $H_i(t_0)=0$. Since $\sigma(0)=0$ (property \eqref{s1}), we have that  $\dot{H}_i(t_0)=0$. Then, since $\tilde{\chi},\sigma\in C^1$, for the initial value problem \eqref{dotH} with $H(t_0)=0$ we get that $H_i(t)=0$ for all $t\ge t_0$. The latter and the semigroup property implies that $H_i$ cannot change sign for system \eqref{GrindEQ__7_13_} with $(s(0),v(0))\in D$. Thus, we get from \eqref{Hi_L} and the fact that $G_i(0)=s_i(0)-\ell_i-h_iv_i(0)>0$ and $H_i(0)=s_i(0)-r_i-h_iv_i(0)$, that \eqref{GrindEQ__3_8_} holds.

{Next, we prove that $G_{i} (t)>0$, $t\ge 0$, $i=1,\ldots,n$ for the extended system \eqref{GrindEQ__7_13_}. Since $G_i(t)=H_i(t)+r_i-\ell_i$ and since $G_i(0)>0$ (due to $(s(0),v(0))\in D$), we have that $H_i(0)>-(r_i-\ell_i)$. Moreover, due to \eqref{GrindEQ__3_8_} it holds that $H_i(t)>-(r_i-\ell_i)$ for all $t\ge0$ and $i=1,\ldots,n$. Thus, the fact that $G_i(t)=H_i(t)+r_i-\ell_i$, gives immediately that 
\begin{equation} \label{GrindEQ__7_24_} 
G_{i} \left(t\right)>0, \textrm{ for all } t>0 \textrm{ and }i=1,\ldots,n.
\end{equation} }
  
Inequalities \eqref{GrindEQ__7_22_} and \eqref{GrindEQ__7_24_} imply that  the solution $(s(t),v(t))\in {\mathbb R}^{2n} $ of the initial value problem \eqref{GrindEQ__7_13_} with $(s(0),v(0))=\left(s_{1,0} ,\ldots,s_{n,0} ,v_{1,0} ,\ldots,v_{n,0} \right)\in D$ exists for all $t\ge 0$ and satisfies $(s(t),v(t))\in D$ for all $t\ge 0$. Since $v_{i} (t)\in [0,v_{\max } ]$ for all $t\ge 0$, it follows from definition \eqref{GrindEQ__7_12_} that $\tilde{\chi }(v_{i} (t))=\chi (v_{i} (t))$ for all $t\ge 0$. Therefore, the solution of \eqref{GrindEQ__7_13_} is also a solution of the original system \eqref{GrindEQ__2_1_}, \eqref{GrindEQ__3_1_}. By uniqueness of solutions, the solutions coincide on the maximal interval of existence $[0,T_{\max } )$ of the solution of \eqref{GrindEQ__2_1_}, \eqref{GrindEQ__3_1_}. The latter shows that $T_{\max } =+\infty $. Hence the solution of the initial value problem \eqref{GrindEQ__2_1_}, \eqref{GrindEQ__3_1_} with initial condition $(s(0),v(0))=\left(s_{1,0} ,\ldots,s_{n,0} ,v_{1,0} ,\ldots,v_{n,0} \right)\in D$ satisfies $(s(t),v(t))\in D$ for all $t\ge 0$. Finally, inequality \eqref{GrindEQ__3_9_} is a consequence of \eqref{GrindEQ__7_22_}, \eqref{GrindEQ__3_8_} and the facts that $G_i=H_i+r_i-\ell_i$, and $s_{i} (t)=\ell _{i} +h_i v_{i} (t)+G_{i} (t)$,  $i=1,\ldots,n$, $t\ge 0$. The proof is complete.  
\end{proof}  

\begin{proof}[Proof of Theorem \ref{thm:3}]
By virtue of Theorem \ref{thm:2}, for any initial condition $\left(s(0),v(0)\right)\in D$ and any leader speed $v_0(t)$ satisfying \eqref{GrindEQ__3_4_}, the solution of \eqref{GrindEQ__2_1_}, \eqref{GrindEQ__3_1_} is defined for all $t\ge 0$ and satisfies $(s(t),v(t))\in D$ for all $t\ge 0$. Let $v^*\in[0,v_{\max}]$. For each $i=1,\ldots,n$, define 
\begin{equation} \label{GrindEQ__7_27_} 
p_{i} =v_{i} -v^{*}  .
\end{equation} 
Since $v^{*} \in [0,v_{\max } ]$ is constant, we get via \eqref{GrindEQ__7_27_} and \eqref{GrindEQ__3_1_} that 
\begin{equation} \label{GrindEQ__7_28_} 
\dot{p}_{i} =\dot{v}_{i} = \frac{1}{h_i}  \left(p_{i-1} -p_{i} \right)+\beta _{i} \chi \left(v_{i} \right) {\sigma(H_{i})}. 
\end{equation} 
Equation \eqref{GrindEQ__7_28_}, gives for each $i=1,\ldots,n$ and for all $t\ge 0$ that
\begin{equation} \label{GrindEQ__7_45_} 
\begin{aligned} &{p_{i} \left(t\right)=\exp \left(- {\frac{1}{h_i}}t \right)p_{i} \left(0\right)+ {\frac{1}{h_i}}\int _{0}^{t}\exp \left(- {\frac{1}{h_i}}(t-\tau ) \right)p_{i-1} \left(\tau \right)  d\tau  } \\
&{+\beta _{i} \int _{0}^{t}\exp \left(- {\frac{1}{h_i}}(t-\tau ) \right)\chi \left(v_{i} \left(\tau \right)\right) {\sigma(H_{i} \left(\tau \right))}d\tau  } .\end{aligned} 
\end{equation} 
We first show that \eqref{GrindEQ__4_1_} holds.  Define for $i=1,\ldots,n$ and $t\ge0$,
\begin{equation}\label{kappa_theta}
\begin{aligned}
\gamma_i(t)&=\frac{1}{h_i}\exp\left( -\frac{1}{h_i}t\right),\\
\theta_i(t)&=\exp\left( -\frac{1}{h_i}t\right).
\end{aligned}
\end{equation}
Then, for $p\in[1,\infty)$ and $t\ge0$ we have that
\begin{equation}\label{kappa:1}
\| \gamma_i\|_{[0,t],1}= \int_0^t \frac{1}{h_i}\exp\left(-\frac{1}{h_i}\tau\right)d\tau=1-\exp\left( -\frac{1}{h_i}t\right)\leq 1
\end{equation}
and
\begin{equation}\label{theta:p}
\| \theta_i\|_{[0,t],p}=\left(\int_0^t \exp\left(-\frac{p}{h_i}\tau\right)d\tau\right)^{1/p}=\left(\frac{h_i}{p}\left(1-\exp\left(-\frac{p}{h_i}t\right)\right)\right)^{1/p}\leq \left(\frac{h_i}{p}\right)^{1/p}.
\end{equation}
 {Recall that Theorem \ref{thm:2} gives that $|H_i(t)|\leq |H_i(0)|$ for $t\ge0$. Moreover, notice that condition \eqref{s4} and the fact that $\sigma\in C^1$ and $\sigma(0)=0$ imply that $\sigma'(0)>0$. Let $\kappa:\mathbb{R}_+\to(0,\infty)$ defined by
\begin{equation}\label{kappa:inf}
\kappa(R):=\left\{\begin{array}{ll}\inf_{0<|H|<R}\left(\frac{\sigma(H)}{H}\right), & R>0  \\ \sigma'(0), & R=0\end{array}\right.
\end{equation}
which is non-increasing with $\kappa(R)>0$ for $R\ge0$ (due to assumption \eqref{s4}).} Let $|H_{i}(0)|\ge0$ and define $R_i=|H_i(0)| $. Then, since $|H_i(t)|\leq |H_i(0)|\leq R_i$, definition of $\kappa$ above gives
\begin{equation}
H_i(t)\sigma(H_i(t))\ge \kappa(R_i) H_i^2(t), \;\; t\ge0.
\end{equation}
Thus, we get 
\begin{equation}
\frac{d}{dt}\left(\frac{1}{2}H_i^2\right)=-h_i\beta_i \chi(v_i)H_i\sigma(H_i)\leq -h_i \beta_i \kappa(R_i)\chi(v_i)H_i^2.
\end{equation}
Let $y_i(t):=H_i^2(t)$, $t\ge0$. Then, the previous inequality gives $\dot{y}_i(t)\leq -2h_i \beta_i \kappa(R_i)\chi(v_i(t))y_i(t)$ and from Gronwall's Inequality (see \cite{Khalil2002}) we get for $i=1,\ldots,n$ and $t\ge0$ that
\begin{equation}
H_i^2(t)\leq H_i^2(0) \exp\left(-2h_i \beta_i \kappa(R_i)\int_0^t\chi(v_i(\tau))d\tau\right).
\end{equation}
Thus,
\begin{equation}\label{Hi}
|H_i(t)|\leq |H_i(0)|\exp\left(-h_i \beta_i \kappa(R_i)\int_0^t\chi(v_i(\tau))d\tau\right).
\end{equation}
Define
\begin{equation} \label{GrindEQ__7_42_} 
Z_i(t):=\exp \left(- h_i \beta _{i}\kappa(R_i) \int _{0}^{t}\chi (v_{i} (s))ds \right), t\ge 0 ,
\end{equation} 
which gives
\begin{equation} \label{GrindEQ__7_43_} 
\dot{Z}_i(t)=-  h_i \beta _{i} \kappa(R_i) \chi (v_{i} (t))Z_i(t), t\ge 0 .
\end{equation} 
Using \eqref{s3}, \eqref{Hi}, \eqref{GrindEQ__7_43_}, and the fact that $\chi (v)\ge 0$, $v\in [0,v_{\max } ]$, $\beta _{i} >0$, $h_{i} >0$ we get for $t\ge 0$ that
\begin{equation} \label{GrindEQ__7_44_} 
\begin{aligned} &{\beta _{i} \int _{0}^{t}\chi (v_{i} (\tau ))\left|\sigma(H_{i} (\tau ))\right|d\tau \leq\beta _{i} \int _{0}^{t}\chi (v_{i} (\tau ))\left|H_{i} (\tau )\right|d\tau  =\beta _{i} \left|H_{i} (0)\right|\int _{0}^{t}\chi (v_{i} (\tau ))Z_i(\tau )d\tau  } \\
 &{=\frac{ \left|H_{i} (0)\right|}{h_{i} \kappa(R_i)} \int _{0}^{t}\left(-\frac{d}{d\tau } \left(Z_i(\tau )\right)\right)d\tau  =\frac{ \left|H_{i} (0)\right|}{h_{i} \kappa(R_i)} \left(1-Z_i(t)\right)\le \frac{ \left|H_{i} (0)\right|}{h_{i} \kappa(R_i)} } .\end{aligned} 
\end{equation} 
Using \eqref{GrindEQ__7_44_} and applying Young's convolution inequality: $\|f\ast g\|_{r}\leq\|f\|_p\|g\|_q$ for $p,q,r\ge1$ and $1+\frac{1}{r}=\frac{1}{q}+\frac{1}{p}$, $f\in L^p$, $g\in L^q$, (see for instance \cite[Theorem 4.33]{Brezis}) 
we obtain
\begin{equation}\label{conv}
\begin{aligned}
&\|\gamma_i \ast p_{i-1}\|_{[0,t],p}\leq \|\gamma_i\|_{[0,t],1}\|p_{i-1}\|_{[0,t],p}\leq\|p_{i-1}\|_{[0,t],p},\\
&\|\theta_i \ast (\beta_i\chi(v_i)\sigma(H_i))\|_{[0,t],p}\leq \|\theta_i\|_{[0,t],p}\|\beta_i\chi(v_i)\sigma(H_i)\|_{[0,t],1}\leq \left(\frac{h_i}{p}\right)^{1/p}\frac{|H_{i}(0)|}{h_i  \kappa(R_i)}.
\end{aligned}
\end{equation}
Thus \eqref{GrindEQ__7_45_}, and Young's convolution inequality gives
\begin{equation}
\|p_i\|_{[0,t],p}\leq \|\theta_i\|_{[0,t],p}|p_{i}(0)|+\|\gamma_i\|_{[0,t],1}\|p_{i-1}\|_{[0,t],p}+\|\theta_i\|_{[0,t],p}\|\beta_i \chi(v_i)\sigma(H_i)\|_{[0,t],1}
\end{equation}
and from \eqref{kappa:1}, \eqref{theta:p}, and \eqref{conv} we get estimate \eqref{GrindEQ__4_1_}.

Next, we show that estimate \eqref{GrindEQ__4_2_} holds. Using \eqref{s2}, \eqref{GrindEQ__7_44_} and \eqref{GrindEQ__7_45_} we get for $i=1,\ldots,n$ and $t\ge 0$
\begin{equation} \label{GrindEQ__7_46_} 
\begin{aligned} &{\left|p_{i} (t)\right|\le \exp \left(-  {\frac{1}{h_i}}t \right)\left|p_{i} (0)\right|+ {\frac{1}{h_i}} \int _{0}^{t}\exp \left(- {\frac{1}{h_i}}(t-\tau )  \right)\left|p_{i-1} (\tau )\right|d\tau  } \\ 
&{+\beta _{i} \int _{0}^{t}\exp \left(- {\frac{1}{h_i}}(t-\tau )  \right)\chi (v_{i} (\tau ))\left|H_{i} (\tau )\right|d\tau  } \\ 
&{\le \left|p_{i} (0)\right|+\sup_{0\le \tau \le t} \left(\left|p_{i-1} \left(\tau \right)\right|\right) {\frac{1}{h_i}}\int _{0}^{t}\exp  \left(-\frac{1}{h_i}(t-\tau )  \right)d\tau  +\beta _{i} \int _{0}^{t}\chi (v_{i} (\tau ))\left|H_{i} (\tau )\right|d\tau  } \\ 
&{\le \left|p_{i} (0)\right|+\sup_{0\le \tau \le t} \left(\left|p_{i-1} \left(\tau \right)\right|\right)\left(1-\exp \left(- {\frac{1}{h_i}}t  \right)\right)+\frac{ \left|H_{i} (0)\right|}{h_{i}\kappa(R_i) } } \\
&{\le \left|p_{i} (0)\right|+ \sup_{0\le \tau \le t} \left(\left|p_{i-1} \left(\tau \right)\right|\right)+\frac{  \left|H_{i} (0)\right|}{h_{i}\kappa(R_i) } } \end{aligned} 
\end{equation} 
Estimate \eqref{GrindEQ__4_2_} is a direct consequence of inequality \eqref{GrindEQ__7_46_} and definition \eqref{GrindEQ__7_27_}.

It remains to prove that estimates \eqref{spacing:i} and \eqref{spacing:1} hold. By definition of $\delta_i=s_i-r_i-h_iv^*$, $i=1,...,n$ and \eqref{GrindEQ__3_2_} we have that 
\begin{equation}\label{delta:i}
\delta_i=H_i+ h_i(v_i-v^*).
\end{equation}
Using estimates \eqref{GrindEQ__3_8_} and \eqref{GrindEQ__4_2_} we get that
\begin{equation}\label{s:1}
\begin{aligned}
\left\| \delta_i \right \|_{[0,t],\infty}&\leq |H_i(0)|+ h_i  \left\|v_i-v^*\right\|_{[0,t],\infty}\\
&\leq h_i \left\|v_{i-1}-v^*\right\|_{[0,t],\infty} + h_i |v_i(0)-v^*| +\left(1+\frac{1}{\kappa(R_i)}\right)|H_i(0)|.
\end{aligned}
\end{equation}
Notice also that \eqref{delta:i} gives
$v_{i-1}-v^*=\frac{1}{h_{i-1}}(\delta_{i-1}-H_{i-1})$.
Consequently, using \eqref{GrindEQ__3_8_} we have that
\begin{equation}\label{last}
\|v_{i-1}-v^*\|_{[0,t],\infty}\leq \frac{1}{h_{i-1}}\|\delta_{i-1}\|_{[0,t],\infty}+\frac{ |H_{i-1}(0)|}{h_{i-1}}.
\end{equation}  
Combining \eqref{last} with \eqref{s:1} gives estimate \eqref{spacing:i}. Moreover, \eqref{spacing:1} follows directly from definition of $\delta_1$ and estimate \eqref{GrindEQ__4_2_}  for $i=1$. Finally, \eqref{H} is a direct consequence of \eqref{Hi}. The proof is complete. \end{proof}

\begin{proof}[Proof of Lemma \ref{lem:1}]
Let \textit{$(s(0),v(0))\in \interior\left(D\right)\subset D$. }By Theorem \ref{thm:2}, the solution $(s(t),v(t))$ of \eqref{GrindEQ__2_1_}, \eqref{GrindEQ__3_1_} satisfies $0\le v_{i} (t)\le v_{\max } $ and inequality \eqref{GrindEQ__3_8_} for all $t\ge 0$ and $i=1,\ldots,n$. Since $\psi \in C^{1} \left([0,v_{\max } ]\right)$ and $\psi (0)=0$, there exists a constant $L>0$ such that $0\le \psi (v)\le Lv$, $v\in [0,v_{\max } ]$ and consequently 
\begin{equation}\label{GrindEQ__7_47_}
\chi (v)\le v_{\max } Lv, \textrm{ for }v\in [0,v_{\max } ].
\end{equation}
 
\noindent { Since $|\sigma(H)|\leq| H|$ and $\sigma$ is sign-preserving (properties \eqref{s3} and \eqref{s2}, respectively), $\chi (v)\ge 0$ and $\beta_i>0$, inequality \eqref{GrindEQ__3_8_} gives 
$\beta _{i} \chi (v_{i} (t))\sigma(H_{i} (t))\ge -\beta _{i} \chi (v_{i} (t))(r_i-\ell_i)$. } The previous inequality, \eqref{GrindEQ__7_47_}, \eqref{GrindEQ__2_1_}, and \eqref{GrindEQ__3_1_} give along the solution $(s(t),v(t))$ that
\begin{equation} \label{GrindEQ__7_48_} 
\dot{v}_{i} (t)\ge \frac{1}{h_{i} } v_{i-1} (t)-\left(\frac{1}{h_{i} } +\beta _{i} Lv_{\max } (r_i-\ell_i)\right)v_{i} (t) ,
\end{equation} 
for $i=1,\ldots,n$  and $t\ge 0$.

Assume now that there exists $\epsilon_{i-1}>0$ such that
\begin{equation}\label{GrindEQ__7_49_}
v_{i-1} (t)\ge \epsilon_{i-1} \textrm{ for }t\ge 0.
\end{equation}
 
\noindent Using assumption \eqref{GrindEQ__7_49_} and the fact that $v_{i} (t)\ge 0$ for $t\ge 0$, inequality \eqref{GrindEQ__7_48_} implies that
\begin{equation} \label{GrindEQ__7_50_} 
\begin{aligned} &{v_{i} (t)\ge \exp \left(-A_{i} t\right)v_{i} (0)+\frac{ \epsilon_{i-1} }{h_{i} A_{i} } \left(1-\exp \left(-A_{i} t\right)\right)} 
{\ge \min \left\{v_{i} (0),\frac{ \epsilon_{i-1} }{h_{i} A_{i} } \right\}:=\epsilon_{i} >0} ,\end{aligned} 
\end{equation} 
where $v_i(0)\in(0,v_{\max})$ and 
\begin{equation} \label{GrindEQ__7_51_} 
A_{i} :=\frac{1}{h_{i} } +\beta _{i} Lv_{\max }(r_i-\ell_i)>0 .
\end{equation} 
Notice now that due to assumption $0<\underline{v}\le v_{0} (t)\le \bar{v}<v_{\max } $, $t\ge 0$, inequality \eqref{GrindEQ__7_49_} holds for $i=1$ with $\epsilon_{0} =\underline{v}$. Thus, inequality \eqref{GrindEQ__7_50_}  holds for $i=1$. Using the same arguments successively we get that there exist constants $\epsilon_{i} >0$ such that $v_{i} (t)\ge \epsilon_{i} $ for $t\ge 0$ and consequently 
\begin{equation} \label{GrindEQ__7_52_} 
v_{i} (t)\ge \underline{\delta }:=\min \left\{\epsilon_{i} ,i=1,\ldots,n\right\}, t\ge 0 .
\end{equation} 
Since $A_{i} \ge 1/h_{i} $, it follows that $\frac{ \epsilon_{i-1} }{h_{i} A_{i} } \le \epsilon_{i-1} $ which in conjunction with definition $\epsilon_{i} :=\min \left\{v_{i} (0),\frac{ \epsilon_{i-1} }{h_{i} A_{i} } \right\}$ in \eqref{GrindEQ__7_50_} gives that $\epsilon_{i} \le \epsilon_{i-1} $. Since $\epsilon_{0} =\underline{v}$ it follows that $\bar{v}\ge \underline{\delta }={\mathop{\min }\limits_{1\le i\le n}} \left\{\epsilon_{i} \right\}=\epsilon_{n} >0$.

Next, define 
\begin{equation} \label{GrindEQ__7_53_} 
z_{i} :=v_{\max } -v_{i} ,\, \, \, i=0,1,\ldots,n 
\end{equation} 
which due to \eqref{GrindEQ__2_1_} and \eqref{GrindEQ__3_1_} gives
\begin{equation} \label{GrindEQ__7_54_} 
\dot{z}_{i} =\frac{1}{h_{i} } \left(z_{i-1} -z_{i} \right)-\beta _{i} \chi (v_{i} ) {\sigma(H_{i})}, i=1,\ldots,n 
\end{equation} 
Since $\psi $ is continuous with $\psi (v)>0$ for $v>0$, we can define $\psi _{\max } :={\mathop{\max }\limits_{v\in [0,v_{\max } ]}} \left(\psi (v)\right)>0$. Using definition \eqref{GrindEQ__3_3_} and \eqref{GrindEQ__7_53_} we get for $v_{i} \in [0,v_{\max } ]$ 
\begin{equation} \label{GrindEQ__7_55_} 
\chi (v_{i} )=(v_{\max } -v_{i} )\psi (v_{i} )=z_{i} \psi (v_{i} )\le \psi _{\max } z_{i}  .
\end{equation} 

Recall that along solutions it holds that $\max \left\{H_{i} (t),0\right\}\le \max \left\{H_{i} (0),0\right\}$ (inequality \eqref{GrindEQ__3_8_}). {Moreover, since $\sigma$ is continuous, the previous inequality and \eqref{s3} give that $\sigma(H_i(t))\leq \max\{H_i(0),0\}$.} Thus, using \eqref{GrindEQ__7_55_} and inequalities $\beta _{i} >0$ and $\chi (v)\ge 0$, we get
\begin{equation} \label{GrindEQ__7_56_} 
-\beta _{i} \chi (v_{i} (t)){\sigma(H_{i} (t))}\ge -\beta _{i} \psi _{\max } \max \left\{H_{i} (0),0\right\}z_{i} (t),\, \, \, \, \, t\ge 0 .
\end{equation} 
Using \eqref{GrindEQ__7_54_} and \eqref{GrindEQ__7_56_} we finally get that for $t\ge 0$ and for all $i=1,\ldots,n$, it holds that
\begin{equation} \label{GrindEQ__7_57_} 
\dot{z}_{i} (t)\ge \frac{1}{h_{i} } z_{i-1} (t)-\left(\frac{1}{h_{i} } +\beta _{i} \psi _{\max } \max \left\{H_{i} (0),0\right\}\right)z_{i} (t) .
\end{equation} 
Assume now that there exists $\bar{\delta}_{i-1}>0$ such that
\begin{equation} \label{GrindEQ__7_58_} 
z_{i-1} (t)\ge \bar{\delta }_{i-1} >0,\, \, \, t\ge 0 .
\end{equation} 
Then \eqref{GrindEQ__7_57_} implies that the following estimate holds for $t\ge 0$
\begin{equation} \label{GrindEQ__7_59_} 
\begin{aligned} &{z_{i} (t)\ge \exp \left(-B_{i} t\right)z_{i} (0)+\frac{ \bar{\delta }_{i-1} }{h_{i} } \left(1-\exp \left(-B_{i} t\right)\right)} \\ &{\ge \min \left\{z_{i} (0),\frac{ \bar{\delta }_{i-1} }{h_{i} B_{i} } \right\}=:\bar{\delta }_{i} >0}, \end{aligned} 
\end{equation} 
where 
\begin{equation} \label{GrindEQ__7_60_} 
B_{i} =\frac{1}{h_{i} } +\beta _{i} \psi _{\max } \max \left\{H_{i} (0),0\right\} .
\end{equation} 

Notice now that since by our assumption $0<\underline{v}\le v_{0} (t)\le \bar{v}<v_{\max } $ for all $t\ge 0$, \eqref{GrindEQ__7_58_} holds for $i=1$ with $\bar{\delta }_{0} =v_{\max } -\bar{v}$. Thus, it follows that \eqref{GrindEQ__7_59_} holds for $i=1$. Using the same arguments successively we get that there exist constants $\bar{\delta }_{i} >0$ such that $z_{i} (t)\ge \bar{\delta }_{i} $ for $t\ge 0$ and consequently $z_{i} (t)\ge \bar{\delta }:=\min \left\{\bar{\delta }_{i} ,i=1,\ldots,n\right\}$ or due to definition \eqref{GrindEQ__7_53_} that $v_{i} (t)\le v_{\max } -\bar{\delta }$. Notice that from $\bar{\delta }_{i} :=\min \left\{z_{i} (0),\frac{ \bar{\delta }_{i-1} }{h_{i} B_{i} } \right\}$ in \eqref{GrindEQ__7_59_} and definition \eqref{GrindEQ__7_60_}, we have that $\bar{\delta }_{i} \le \bar{\delta }_{i-1} $. Since $\bar{\delta }_{0} =v_{\max } -\bar{v}$, it follows that $\bar{\delta }:=\min \left\{\bar{\delta }_{i} ,i=1,\ldots,n\right\}\le v_{\max } -\bar{v}$. Therefore, the statement in the Lemma holds with $\delta =\min \left\{\underline{\delta },\bar{\delta }\right\}\le \min \left\{\underline{v},v_{\max } -\bar{v}\right\}$. The proof is complete.
\end{proof}

\begin{proof}[Proof of Proposition \ref{prop:2}]
Let \textit{$(s(0),v(0))\in\interior\left(D\right)\subset D$. }By Theorem \ref{thm:2}, the solution $(s(t),v(t))$ of \eqref{GrindEQ__2_1_}, \eqref{GrindEQ__3_1_} exists for $t\ge 0$ and satisfies $(s(t),v(t))\in D$. Let $0<\underline{v}\le v_{0} (t)\le \bar{v}<v_{\max } $. Then, from Lemma \ref{lem:1}, there exists $\delta >0$ such that 
\begin{equation}\label{GrindEQ__7_61_}
\delta \le v_{i} (t)\le v_{\max } -\delta, \textrm{ for all }t\ge 0 \textrm{ and }i=1,\ldots,n.
\end{equation}

\noindent By definition \eqref{GrindEQ__3_3_} and since $\chi $ is continuous with $\chi (v)>0$ for $v\in (0,v_{\max } )$, there exists $\underline{\chi }:={\mathop{\min }\limits_{v\in [\delta ,v_{\max } -\delta ]}} \left(\chi (v)\right)>0$. Hence, by \eqref{GrindEQ__7_61_} it follows that
\begin{equation} \label{GrindEQ__7_62_} 
\chi (v_{i} (t))\ge \underline{\chi }, t\ge 0 , i=1,\ldots,n .
\end{equation} 

Define now the deviation from the speed of the leading vehicle for each follower
\begin{equation} \label{GrindEQ__7_63_} 
\zeta_{i} (t):=v_{i} (t)-v_{0} (t),\, \, \, \zeta_{0}(t) \equiv 0, t\ge 0, i=1,\ldots,n ,
\end{equation} 
Using \eqref{GrindEQ__2_1_} and \eqref{GrindEQ__3_1_} we get
\[\dot{\zeta}_{i} =\frac{1}{h_{i} } \left(\zeta_{i-1} -\zeta_{i} \right)+\beta _{i} \chi (v_{i} ){\sigma(H_{i})} -\dot{v}_{0} , i=1,\ldots,n\] 
from which we have for $i=1,\ldots,n$ and $t\ge 0$
\begin{equation} \label{GrindEQ__7_64_} 
\begin{aligned}  &{\zeta_{i} (t)=\exp \left(-  {\frac{1}{h_i}}t \right)\zeta_{i} (0) +  {\frac{1}{h_i}}\int _{0}^{t}\exp \left(- {\frac{1}{h_i}}(t-\tau )  \right)\zeta_{i-1} (\tau )d\tau  } \\& {+\beta _{i} \int _{0}^{t}\exp \left(- {\frac{1}{h_i}}(t-\tau )  \right)\chi (v_{i} (\tau )){\sigma(H_{i} (\tau ))}d\tau  -\int _{0}^{t}\exp \left(- {\frac{1}{h_i}}(t-\tau )  \right)\dot{v}_{0} (\tau )d\tau  } \end{aligned} 
\end{equation} 
Using \eqref{GrindEQ__7_62_} and \eqref{Hi} withd $R_i=|H_i(0)| $, we get
\begin{equation} \label{GrindEQ__7_65_} 
\left|H_{i} (t)\right|\le \exp \left(-\left( \beta _{i} \underline{\chi }h_i \kappa(R_i)\right)t\right)\left|H_{i} (0)\right|, t\ge 0, i=1,\ldots,n .
\end{equation} 
where $\kappa$ is defined by \eqref{kappa:inf}. Then from \eqref{s2}, \eqref{GrindEQ__7_64_}, \eqref{GrindEQ__7_65_} and definition $\chi _{\max } :={\mathop{\max }\limits_{v\in [0,v_{\max } ]}} \left(\chi (v)\right)>0$, we get 
\begin{equation} \label{GrindEQ__7_66_} 
\begin{aligned} &{\left|\zeta_{i} (t)\right|\le \exp \left(- {\frac{1}{h_i}}t \right)\left|\zeta_{i} (0)\right|+  {\frac{1}{h_i}}\int _{0}^{t}\exp \left(- {\frac{1}{h_i}}(t-\tau ) \right)\left|\zeta_{i-1} (\tau )\right|d\tau  } \\ 
&{+\beta _{i} \int _{0}^{t}\exp \left(- {\frac{1}{h_i}}(t-\tau )  \right)\chi (v_{i} (\tau ))\left|H_{i} (\tau )\right|d\tau  +\int _{0}^{t}\exp \left(- {\frac{1}{h_i}}(t-\tau ) \right)\left|\dot{v}_{0} (\tau )\right|d\tau  } \\ 
&{\le \exp \left(- {\frac{1}{h_i}}t  \right)\left|\zeta_{i} (0)\right|+  {\frac{1}{h_i}} \int _{0}^{t}\exp \left(- {\frac{1}{h_i}}(t-\tau )  \right)\left|\zeta_{i-1} (\tau )\right|d\tau  } \\ 
&{+\beta _{i} \chi _{\max } \int _{0}^{t}\exp \left(- {\frac{1}{h_i}}(t-\tau ) \right)\left|H_{i} (\tau )\right|d\tau  +{\mathop{\sup }\limits_{0\le \tau \le t}} \left(\left|\dot{v}_{0} (\tau )\right|\right)\int _{0}^{t}\exp \left(- {\frac{1}{h_i}}(t-\tau )  \right)d\tau  } \end{aligned} 
\end{equation} 
Now let 
\begin{equation} \label{GrindEQ__7_67_} 
{0<\mu <{\mathop{\min }\limits_{i=1,\ldots,n}} \left\{\frac{1}{h_{i} } , {h_{i}} \beta _{i} \kappa(R_i)\underline{\chi }\right\} .}
\end{equation} 
It follows from \eqref{GrindEQ__7_65_} and \eqref{GrindEQ__7_67_} that
\begin{equation} \label{GrindEQ__7_68_} 
\begin{aligned} &{\int _{0}^{t}\exp \left(-\frac{1}{h_i}(t-\tau )  \right)\left|H_{i} (\tau )\right|d\tau  \le \left|H_{i} (0)\right|\int _{0}^{t}\exp \left(- {\frac{1}{h_i}}(t-\tau )  \right)\exp \left(-\left( \beta _{i} \underline{\chi } m_{R_i} {h_i}\right)\tau \right)d\tau  } \\
&{\le \left|H_{i} (0)\right|\int _{0}^{t}\exp \left(- {\frac{1}{h_i}}(t-\tau ) \right) \exp \left(-\mu \tau \right)d\tau } \\ &{=\left|H_{i} (0)\right|\frac{\exp \left(-\mu t\right)-\exp \left(- {\frac{1}{h_i}}t  \right)}{ {\frac{1}{h_i}} -\mu } \le \frac{1}{ {\frac{1}{h_i}} -\mu } \exp \left(-\mu t\right)\left|H_{i} (0)\right|} \end{aligned} 
\end{equation} 
where $\mu <1/h_{i} $  due to \eqref{GrindEQ__7_67_}. Combining \eqref{GrindEQ__7_65_}, \eqref{GrindEQ__7_66_}, \eqref{GrindEQ__7_67_}, and \eqref{GrindEQ__7_68_} we obtain for $i=1,\ldots,n$, $t\ge 0$
\begin{equation} \label{GrindEQ__7_69_} 
\begin{aligned} &{\left|\zeta_{i} (t)\right|+\left|H_{i} (t)\right|\le \exp \left(-\mu t\right)\left|\zeta_{i} (0)\right|+ {\frac{1}{h_i}}\int _{0}^{t}\exp \left(- {\frac{1}{h_i}}(t-\tau ) \right)\left|\zeta_{i-1} (\tau )\right|d\tau  } \\ 
&{+\left(1+\frac{\beta _{i} \chi _{\max } }{ {1/h_{i}} -\mu } \right)\exp \left(-\mu t\right)\left|H_{i} (0)\right|+ {h_{i}} {\mathop{\sup }\limits_{0\le \tau \le t}} \left(\left|\dot{v}_{0} (\tau )\right|\right)} \end{aligned} 
\end{equation} 
We now show that \eqref{GrindEQ__4_3_} holds by induction. For $i=1$, and since $q_{0} \equiv 0$, we get from \eqref{GrindEQ__7_65_} and \eqref{GrindEQ__7_69_} that 
\begin{equation} \label{GrindEQ__7_70_} 
\left|\zeta_{1} (t)\right|+\left|H_{1} (t)\right|\le \left(1+\frac{\beta _{1} \chi _{\max } }{  {1/h_{1}} -\mu } \right)\exp \left(-\mu t\right)\left(\left|\zeta_{1} (0)\right|+\left|H_{1} (0)\right|\right)+h_{1}  {\mathop{\sup }\limits_{0\le \tau \le t}} \left(\left|\dot{v}_{0} (\tau )\right|\right) .
\end{equation} 
Thus, \eqref{GrindEQ__4_3_} holds with $C_{1} =1+\frac{\beta _{1} \chi _{\max } }{1/h_{1} -\mu } $ and $\Gamma _{1} =h_{1}  $. Suppose now that \eqref{GrindEQ__4_3_} holds for $i-1\ge 1$. Then,
\begin{equation} \label{GrindEQ__7_71_} 
\left|\zeta_{i-1} (t)\right|\le \left|\zeta_{i-1} (t)\right|+\left|H_{i-1} (t)\right|\le C_{i-1} \exp \left(-\mu t\right)\sum _{j=1}^{i-1}\left(\left|\zeta_{j} (0)\right|+\left|H_{j} (0)\right|\right) +\Gamma _{i-1} \left\| \dot{v}_{0} \right\| _{\left[0,t\right], \infty}  .
\end{equation} 
Therefore, using \eqref{GrindEQ__7_69_} and \eqref{GrindEQ__7_71_}, we get
\[\begin{aligned} &{\left|\zeta_{i} (t)\right|\le \left(1+\frac{\beta _{i} \chi _{\max } }{ {1/h_{i}} -\mu } \right)\exp \left(-\mu t\right)\left(\left|\zeta_{i} (0)\right|+\left|H_{i} (0)\right|\right)+h_i{\mathop{\sup }\limits_{0\le \tau \le t}} \left(\left|\dot{v}_{0} (\tau )\right|\right)} \\ 
&+ {\frac{1}{h_i}}C_{i-1} \left(\sum _{j=1}^{i-1}\left(\left|\zeta_{j} (0)\right|+\left|H_{j} (0)\right|\right) \right)\int _{0}^{t}\exp \left(- {\frac{1}{h_i}}(t-\tau )  \right)\exp \left(-\mu \tau \right)d\tau \\
& +\left(\Gamma _{i-1} +  {h_{i}} \right){\mathop{\sup }\limits_{0\le \tau \le t}} \left(\left|\dot{v}_{0} (\tau )\right|\right) \\ 
&{\le \left(1+\frac{\beta _{i} \chi _{\max } }{ {1/h_{i}} -\mu } \right)\exp \left(-\mu t\right)\left(\left|\zeta_{i} (0)\right|+\left|H_{i} (0)\right|\right)+\exp \left(-\mu t\right)\frac{C_{i-1} }{1-\mu  {h_{i} }} \left(\sum _{j=1}^{i-1}\left(\left|\zeta_{j} (0)\right|+\left|H_{j} (0)\right|\right) \right)} \\ 
&{+\left( {h_{i} } +\Gamma _{i-1} \right){\mathop{\sup }\limits_{0\le \tau \le t}} \left(\left|\dot{v}_{0} (\tau )\right|\right)} \\ 
&{\le \max \left\{1+\frac{\beta _{i} \chi _{\max } }{ {1/h_{i}} -\mu } ,\frac{C_{i-1} }{1- \mu {h_{i} } } \right\}\exp \left(-\mu t\right)\sum _{j-1}^{i}\left(\left|\zeta_{j} (0)\right|+\left|H_{j} (0)\right|\right) +\left( {h_{i}} +\Gamma _{i-1} \right){\mathop{\sup }\limits_{0\le \tau \le t}} \left(\left|\dot{v}_{0} (\tau )\right|\right)} \end{aligned}\] 
Thus \eqref{GrindEQ__4_3_} holds with
\[\begin{aligned} &{C_{i} =\max \left\{1+\frac{\beta _{i} \chi _{\max } }{ {1/h_{i}} -\mu } ,\frac{C_{i-1} }{1- \mu {h_{i} } } \right\}} ,\\ &{\Gamma _{i} = {h_{i}} +\Gamma _{i-1} }. \end{aligned}\] 
 The proof is complete.
\end{proof}

\begin{proof}[Proof of Theorem \ref{thm:4}]
Define for $(s,v)\in D$
\begin{equation} \label{GrindEQ__7_72_} 
V(s,v)=\sum _{i=1}^{n}\lambda ^{i-1} \left(\frac{1}{2} H_{i}^{2} +\frac{\phi}{2} p_{i}^{2} \right)  
\end{equation} 
where $\lambda \in (0,1)$, $\phi>0$, $H_{i} =s_{i} -r_{i} -h_i v_{i} $, and $p_{i} =v_{i} -v^{*} $. Using \eqref{GrindEQ__2_1_}, \eqref{GrindEQ__3_1_}, we obtain 
\begin{equation} \label{GrindEQ__7_73_} 
\dot{H}_{i} =-  {h_{i}} \beta _{i} \chi \left(v_{i} \right){\sigma(H_{i} )}, 
\end{equation} 
\begin{equation} \label{GrindEQ__7_74_} 
\dot{p}_{i} =\dot{v}_{i} = {\frac{1}{h_i}} \left(p_{i-1} -p_{i} \right)+\beta _{i} \chi \left(v_{i} \right){\sigma(H_{i} )} . 
\end{equation} 
Then, using \eqref{GrindEQ__7_73_}, \eqref{GrindEQ__7_74_}, and \eqref{GrindEQ__7_72_} it follows that
\begin{equation} \label{GrindEQ__7_75_} 
\begin{aligned} &{\dot{V}(s,v)=\sum _{i=1}^{n}\lambda ^{i-1} \left(H_{i} \dot{H}_{i} +cp_{i} \dot{p}_{i} \right) } \\ 
&{=\sum _{i=1}^{n}\lambda ^{i-1} \left(-\left(  {h_{i}} \beta _{i} \chi (v_{i} )\right)H_{i}{\sigma(H_i)} + \phi  {\frac{1}{h_i}} p_{i} p_{i-1} - \phi  {\frac{1}{h_i}} p_{i}^{2} +\phi\beta _{i} \chi (v_{i} ){\sigma(H_i)} p_{i} \right) }.\end{aligned} 
\end{equation} 
{From properties \eqref{s2} and \eqref{s3} we have that
\begin{equation}\label{sigma:ineq}
\sigma(H)^2=|\sigma(H)|^2\leq|H||\sigma(H)|=H\sigma(H).
\end{equation}
Thus, using \eqref{sigma:ineq} and Young's inequality, $\left|ab\right|\le \frac{\varepsilon }{2} a^{2} +\frac{1}{2\varepsilon } b^{2} $ , $a,b\in {\mathbb R}$, $\varepsilon >0$, the following hold
\begin{equation} \label{GrindEQ__7_76_} 
\begin{aligned}  \phi\beta _{i} \chi (v_{i} )\left|\sigma(H_{i}) p_{i} \right|& \le \frac{h_{i} \beta _{i} \chi (v_{i} )}{2 } \sigma^{2}(H_{i}) +\frac{\phi^{2} \beta _{i} \chi (v_{i} )}{2h_{i} } p_{i}^{2}   \leq\frac{h_{i} \beta _{i} \chi (v_{i} )}{2 } H_{i}\sigma(H_i) +\frac{\phi^{2} \beta _{i} \chi (v_{i} )}{2h_{i} } p_{i}^{2} , \\
 \phi  {\frac{1}{h_i}} \left|p_{i} p_{i-1} \right|&\le  \frac{\phi}{2h_{i} } p_{i}^{2} + \frac{\phi}{2h_{i} } p_{i-1}^{2}  . \end{aligned} 
\end{equation} 
Combining \eqref{GrindEQ__7_75_} with \eqref{GrindEQ__7_76_} we get
\begin{equation} \label{GrindEQ__7_77_} 
\dot{V}(s,v)\le \sum _{i=1}^{n}\lambda ^{i-1} \left(-\frac{h_{i} \beta _{i} \chi (v_{i} )}{2 } H_{i}\sigma(H_i) - \phi  {\frac{1}{2h_{i}}} p_{i}^{2} +\frac{\phi^{2} \beta _{i}  \chi (v_{i} )}{2h_i} p_{i}^{2} \right) +\sum _{i=2}^{n}\lambda ^{i-1} \frac{\phi}{2h_{i} } p_{i-1}^{2}  . 
\end{equation} }
Since $\lambda \in (0,1)$, $c>0$, and $h_{i} >0$ we also have that
\begin{equation} \label{GrindEQ__7_78_} 
\sum _{i=2}^{n}\lambda ^{i-1} \frac{\phi}{2h_{i} } p_{i-1}^{2}  =\lambda \sum _{i=1}^{n-1}\lambda ^{i-1} \frac{\phi }{2h_{i+1} } p_{i}^{2}  \le \lambda \sum _{i=1}^{n}\lambda ^{i-1} \frac{\phi }{2h_{\min } } p_{i}^{2} ,  
\end{equation} 
where $h_{\min }:={\mathop{\min }\limits_{i=1,\ldots,n}} \left(h_{i} \right)>0$.  Thus, \eqref{GrindEQ__7_77_}, \eqref{GrindEQ__7_78_}, and definition $\chi _{\max } ={\mathop{\max }\limits_{v\in [0,v_{\max } ]}} \left(\chi (v)\right)$ give
\begin{equation} \label{GrindEQ__7_79_} 
 \dot{V}(s,v)\le -\sum _{i=1}^{n}\lambda ^{i-1} \frac{h_{i} \beta _{i} \chi (v_{i} )}{2 } H_{i}{\sigma(H_i)}  -\frac{\phi}{2} \sum _{i=1}^{n}\lambda ^{i-1} \left( \frac{1}{h_{i} }(1-\phi\beta _{i} \chi _{\max } ) -\frac{\lambda }{h_{\min } } \right) p_{i}^{2}  .
\end{equation} 
Let 
\[\begin{aligned}  &{\beta _{\max } ={\mathop{\max }\limits_{i=1,\ldots,n}} \left(\beta _{i} \right)}, \\ &{\phi=\frac{1}{2\beta _{\max } \chi _{\max } } }, \end{aligned}\] 
and
\[0<\lambda <{\mathop{\min }\limits_{i=1,\ldots,n}} \left\{\frac{ h_{\min } }{h_{i}   } \left(1-\phi\beta _{i} \chi _{\max } \right)\right\}.\] 
Notice that   $1-\phi\beta _{i} \chi _{\max } <1$. Thus $\lambda \in (0,1)$ . Define also
\begin{equation} \label{GrindEQ__7_80_} 
m_{\lambda } :={\mathop{\min }\limits_{i=1,\ldots,n}} \left\{ \frac{1}{h_{i} }(1-\phi\beta _{i} \chi _{\max }) -\lambda  {\frac{1}{h_{\min } }} \right\} .
\end{equation} 
Notice that due to \eqref{GrindEQ__7_80_}, it follows that $m_{\lambda } >0$. Combining \eqref{GrindEQ__7_80_} and \eqref{GrindEQ__7_79_} we get
\begin{equation} \label{GrindEQ__7_81_} 
\dot{V}(s,v)\le -\sum _{i=1}^{n}\lambda ^{i-1} \left(\frac{h_{i} \beta _{i} \chi (v_{i} )}{2 } H_{i}{\sigma(H_i)} +\frac{\phi m_{\lambda } }{2} p_{i}^{2} \right) . 
\end{equation} 
Since $v\in [0,v_{\max } ]$, it follows by \eqref{GrindEQ__3_3_} that $\chi (v)\ge 0$ for $v\in [0,v_{\max } ]$, and consequently, inequality \eqref{GrindEQ__7_81_} and {property \eqref{s2}} give
\begin{equation} \label{GrindEQ__7_82_} 
\dot{V}(s,v)\le 0 ,
\end{equation} 
for $(s,v)\in D$. Recall now that Theorem  \ref{thm:2} establishes that $v_{i} (t)\in [0,v_{\max } ]$, $i=1,\ldots,n$ and that \eqref{GrindEQ__3_9_}  holds for all  $t\ge 0$ and $i=1,\ldots,n$ for any $(s(0),v(0))\in D$. Let $(s(0),v(0))\in D$ and define the set
\begin{equation} \label{GrindEQ__7_83_} 
\Omega =\overline{\bigcup _{t\ge 0}\left\{s(t),v(t)\right\} } ,
\end{equation} 
which is bounded, due to Theorem \ref{thm:2}. Thus $\Omega $ is compact, positively invariant, and satisfies $\Omega \subset D$. Define also the set
\begin{equation} \label{GrindEQ__7_84_} 
Q:=\left\{(s,v)\in \Omega :\dot{V}(s,v)=0\right\}.
\end{equation} 
By \eqref{GrindEQ__7_81_}, every $(s,v)\in Q$ satisfies
\begin{equation} \label{GrindEQ__7_85_} 
\begin{aligned}  &{\chi (v_{i} )H_{i}{\sigma(H_i)} =0} \\ &{p_{i} =0} \end{aligned},\, \, \, i=1,\ldots,n .
\end{equation} 
In particular, we have
\begin{equation} \label{GrindEQ__7_86_} 
Q=\left\{(s,v)\in \Omega :p_{i} =0,\chi (v_{i} )H_{i}{\sigma(H_i)} =0,i=1,\ldots,n\right\}.
\end{equation} 
We can now distinguish the following cases:\\

\noindent \textbf{Case 1:} $v^{*} \in (0,v_{\max } )$. First recall that in this case, we have that the equilibrium configuration is given by the singleton $E_{v^{*} } $ defined by \eqref{GrindEQ__5_1_}. Since $p_{i} =0$ on $Q$ we have that $v_{i} =v^{*} $ for all $i=1,\ldots,n$. Since $v^{*} \in (0,v_{\max }) $, definition \eqref{GrindEQ__3_3_} gives $\chi (v^{*} )>0$. Hence, by { property \eqref{s1}} and \eqref{GrindEQ__7_85_} we have that $H_{i} =0$ for $i=1,\ldots,n$. Therefore,   by LaSalle's Invariance principle, the state $(s(t),v(t))$ approaches the largest invariant set in $Q$ as $t\to +\infty $, that is $\Omega \cap \left\{w^{*} \right\}$. Hence, 
\begin{equation} \label{GrindEQ__7_87_} 
\left. \begin{array}{c} {{\mathop{\lim }\limits_{t\to +\infty }} \left(v_{i} (t)\right)=v^{*} } \\ {{\mathop{\lim }\limits_{t\to +\infty }} \left(H_{i} (t)\right)=0} \end{array}\right\}\Rightarrow \left. \begin{array}{c} {{\mathop{\lim }\limits_{t\to +\infty }} \left(v_{i} (t)\right)=v^{*} } \\ {{\mathop{\lim }\limits_{t\to +\infty }} \left(s_{i} (t)\right)=s_{i}^{*} =r_{i} +  {h_{i}} v^{*} } \end{array}\right\} 
\end{equation} 
Notice that by definition \eqref{GrindEQ__7_72_}, $V(s,v)>0$  for $(s,v)\in D$ with $(s,v)\ne w^{*} $ and $V(w^{*} )=0$. Continuity of $V$, \eqref{GrindEQ__7_82_} and \eqref{GrindEQ__7_87_} show global asymptotic stability of $w^{*} $.\\

\noindent \textbf{Case 2:} $v^{*} =0$. Recall that in this case, the equilibrium configuration is given by the set $E_{0} $ defined by \eqref{GrindEQ__5_2_}. Condition $p_{i} =0$ in \eqref{GrindEQ__7_85_} implies that $v_{i} =0$ for all $i=1,\ldots,n$. By definition \eqref{GrindEQ__3_3_} we have $\chi (0)=0$, and therefore the condition $\chi (v_{i} )H_{i}{\sigma(H_i)} =0$ is satisfied automatically. Hence,
\[Q=\left\{(s,v)\in \Omega :v_{i} =0,i=1,\ldots,n\right\}.\] 
Since $\Omega \subset D$, every $(s,v)\in Q$ also satisfies 
\[s_{i} >\ell _{i} ,i=1,\ldots,n.\] 
{By LaSalle's Invariance Principle, the state $(s(t),v(t))$ approaches the largest invariant set in  $Q$ as $t\to +\infty $, that is $\Omega \cap E_{0} $. The latter implies that ${\mathop{\lim }\limits_{t\to +\infty }} \left(v_{i} (t)\right)=0$ for $i=1,\ldots,n$. Recall now that $\dot{H}_i=-h_i\beta_i\chi(v_i)\sigma(H_i)$, which together with \eqref{s2} and the fact that $\chi(v)\ge0$, implies monotonicity of $H_i$. Indeed, the latter follows from the facts that $H_i$ preserves sign, and that if $H_i\ge0$, then $\dot{H}_i\leq0$, while if $H_i<0$, then $\dot{H}_i\ge 0$. Finally, monotonicity and boundedness of $H_i$ (recall \eqref{H}), imply that ${\mathop{\lim }\limits_{t\to +\infty }} \left(H_{i} (t)\right)$ exists, that is, for every $i=1,\ldots,n$ there exist constants $\bar{H}_{i} \in {\mathbb R}$  such  that ${\mathop{\lim }\limits_{t\to +\infty }} \left(H_{i} (t)\right)=\bar{H}_{i} $.  By \eqref{GrindEQ__3_8_} it follows that $\bar{H}_{i} \in \left(\ell _{i} -r_{i} ,\max \left\{H_{i} (0),0\right\}\right]$. Using the fact that $s_{i} (t)=r_{i} +  {h_{i}} v_{i} (t)+H_{i} (t)$ and that ${\mathop{\lim }\limits_{t\to +\infty }} \left(v_{i} (t)\right)=0$, it follows that ${\mathop{\lim }\limits_{t\to +\infty }} \left(s_{i} (t)\right)=r_{i} +\bar{H}_{i} $.}\\

\noindent \textbf{Case 3:} $v^{*} =v_{\max } $. Recall that in this case, the equilibrium configuration is given by the closed and unbounded set $E_{v_{\max } } $defined by \eqref{GrindEQ__5_3_}. Condition $p_{i} =0$ in \eqref{GrindEQ__7_85_} implies that $v_{i} =v_{\max } $ for all $i=1,\ldots,n$. By definition \eqref{GrindEQ__3_3_} we have $\chi (v_{\max } )=0$, and therefore the condition $\chi (v_{i} )H_{i}^{2} =0$ is satisfied automatically. Hence,
\[Q=\left\{(s,v)\in \Omega :v_{i} =v_{\max } ,i=1,\ldots,n\right\}.\] 
Since $\Omega \subset D$, every  $(s,v)\in Q$ also satisfies 
\[s_{i} >\ell _{i} +h_{i} v_{\max } ,i=1,\ldots,n.\] 
By LaSalle's Invariance Principle, the state $(s(t),v(t))$ approaches the largest invariant set in  $Q$ as $t\to +\infty $, that is $\Omega \cap E_{v_{\max } } $. The latter implies that ${\mathop{\lim }\limits_{t\to +\infty }} \left(v_{i} (t)\right)=v_{\max } $. Similar to the previous case, monotonicity and boundedness of $H_{i} (t)$ in \eqref{H}, it follows that there exist constants $\tilde{H}_{i} \in \left(\ell _{i} -r_{i} ,\max \left\{H_{i} (0),0\right\}\right]$, such that  ${\mathop{\lim }\limits_{t\to +\infty }} \left(H_{i} (t)\right)=\tilde{H}_{i} $. Using the fact that $s_{i} (t)=r_{i} +h_iv_{i} (t)+H_{i} (t)$ and that ${\mathop{\lim }\limits_{t\to +\infty }} \left(v_{i} (t)\right)=v_{\max } $, it follows that ${\mathop{\lim }\limits_{t\to +\infty }} \left(s_{i} (t)\right)=r_{i} +h_{i} v_{\max } +\tilde{H}_{i} $. The proof is complete.  
\end{proof}

\begin{proof}[Proof of Proposition \ref{prop:3}] Let $\left(\tilde{s}_{0} ,\tilde{v}_{0} \right)\in \interior\left(D\right)$. Since $\interior (D)\subset D$, it follows by Theorem \ref{thm:2} that the solution of \eqref{GrindEQ__2_1_}, \eqref{GrindEQ__3_1_} with $\left(s(0),v(0)\right)=\left(\tilde{s}_{0} ,\tilde{v}_{0} \right)$ is defined for all times and satisfies \eqref{GrindEQ__3_7_}. Since $\left(s(0),v(0)\right)\in \interior (D)$ and $v_{0} (t)\equiv v^{*} \in (0,v_{\max } )$ it follows by Proposition \ref{prop:2} (with $v^{*} =\bar{v}=\underline{v}$ and $\dot{v}_{0} \equiv 0$) that for every $i=1,\ldots,n$, the following estimates hold
\begin{equation} \label{GrindEQ__7_89_} 
\left|v_{i} (t)-v^{*} \right|+\left|H_{i} (t)\right|\le C_{0} \exp \left(-\mu t\right)\sum _{j=1}^{i}\left(\left|v_{j} (0)-v^{*} \right|+\left|H_{j} (0)\right|\right) , t\ge 0 
\end{equation} 
where $C_{0} ={\mathop{\max }\limits_{i=1,\ldots,n}} \left(C_{i} \right)$ and $\mu >0$. Using definitions \eqref{GrindEQ__3_2_} and \eqref{GrindEQ__5_1_} we have that
\begin{equation} \label{GrindEQ__7_90_} 
s_{i} -s_{i}^{*} =s_{i} -r_{i} -h_i v^{*} =H_{i} +h_i (v_{i} -v^{*} ) .
\end{equation} 
Thus, for each $i=1,\ldots,n$ and $t\ge 0$ it holds 
\begin{equation} \label{GrindEQ__7_91_} 
\left|s_{i} (t)-s_{i}^{*} \right|+\left|v_{i} (t)-v^{*} \right|\le \left(1+ {h}_{\max } \right)\left(\left|H_{i} (t)\right|+\left|v_{i} (t)-v^{*} \right|\right) ,
\end{equation} 
where $ {h}_{\max } ={\mathop{\max }\limits_{i=1,\ldots,n}} \left(  {h_{i}} \right)$. Equation \eqref{GrindEQ__7_90_} also gives for $i=1,\ldots,n$ that
\begin{equation} \label{GrindEQ__7_92_} 
\begin{aligned} & {\left|v_{i} (0)-v^{*} \right|+\left|H_{i} (0)\right|\le \left(1+ {h_{i}} \right)\left(\left|v_{i} (0)-v^{*} \right|+\left|s_{i} (0)-s_{i}^{*} \right|\right)} \\ &{\le \left(1+ {h}_{\max } \right)\left(\left|v_{i} (0)-v^{*} \right|+\left|s_{i} (0)-s_{i}^{*} \right|\right)} \end{aligned} 
\end{equation} 
 Combining, \eqref{GrindEQ__7_89_}, \eqref{GrindEQ__7_91_}, \eqref{GrindEQ__7_92_}, and \eqref{GrindEQ__5_1_} give for $t\ge 0$ that 
\begin{equation} \label{GrindEQ__7_93_} 
\begin{aligned} &{\sum _{i=1}^{n}\left(\left|s_{i} (t)-s_{i}^{*} \right|+\left|v_{i} (t)-v^{*} \right|\right)\le \left(1+ {h}_{\max } \right)\exp \left(-\mu t\right)nC_{0}  \sum _{j=1}^{n}\left(\left|v_{j} (0)-v^{*} \right|+\left|H_{j} (0)\right|\right) } \\ 
&{\le \left(1+ {h}_{\max } \right)\exp \left(-\mu t\right)nC_{0} \sum _{i=1}^{n}\left(\left|v_{i} (0)-v^{*} \right|+\left|H_{i} (0)\right|\right) } \\ 
&{\le \left(1+ {h}_{\max } \right)^{2} \exp \left(-\mu t\right)nC_{0} \sum _{i=1}^{n}\left(\left|v_{i} (0)-v^{*} \right|+\left|s_{i} (0)-s_{i}^{*} \right|\right) } \end{aligned} 
\end{equation} 
Thus, estimate \eqref{GrindEQ__5_8_} is a direct consequence of \eqref{GrindEQ__5_1_}, \eqref{GrindEQ__7_93_}, and equivalence of norms. The proof is complete.
 \end{proof}

To prove Theorem \ref{thm:5} we need the following lemma whose proof is similar to that of Lemma \ref{lem:1} and is omitted.
\begin{lemma}\label{lem:2}
{Assume that $v_{0} (t)\equiv v^{*} \in (0,v_{\max } )$. Let $\bar{\delta }>0$ and $\bar{G}\ge  \max_{i=1,\ldots,n}(r_{i} +h_{i} v^{*}) $ with $v_{\max } -\bar{\delta }>0$. Moreover, suppose that $v_{i} (0)\in [\bar{\delta },v_{\max } -\bar{\delta }]$ and $0<G_{i} (0)\le \bar{G}$ for all $i=1,\ldots,n$. Then, there exists constants $\Delta _{1} :=\Delta _{1} (\bar{\delta },\bar{G})>0$, $\Delta _{2} :=\Delta _{2} (\bar{\delta },\bar{G})>0$,  such that  $\Delta _{1} \le v_{i} (t)\le v_{\max } -\Delta _{2} $  for all $t\ge 0$ and $i=1,\ldots,n$.}
\end{lemma}

\begin{proof}[Proof of Theorem \ref{thm:5}] Let $v^*\in(0,v_{\max})$ and $\bar{\delta },\bar{G},>0$ (arbitrary) satisfying $v_{\max } -\bar{\delta }>0$, $\bar{G}\ge r_{i} + {h_{i}} v^{*} $  and consider the set $W_{\delta,\bar{G}}$ in \eqref{GrindEQ__5_9_}.  Let also $\left(s(0),v(0)\right)\in W_{\bar{\delta},\bar{G}}$. Since $W_{\delta,\bar{G}}\subset D$, Theorem \ref{thm:2} establishes that the solution of \eqref{GrindEQ__2_1_}, \eqref{GrindEQ__3_1_} exists for all times and satisfies $(s(t),v(t))\in D$ for all $t\ge 0$. By virtue of Lemma \ref{lem:2} above, we have that there exist constants $\Delta _{1} :=\Delta _{1} (\bar{\delta },\bar{G})>0$, $\Delta _{2} :=\Delta _{2} (\bar{\delta },\bar{G})>0$ such that
\begin{equation} \label{GrindEQ__7_105_} 
\Delta _{1} \le v_{i} (t)\le v_{\max } -\Delta _{2}, \;\; t\ge 0, i=1,\ldots,n .
\end{equation} 
From definition \eqref{GrindEQ__3_3_} and \eqref{GrindEQ__7_105_} we get for all $i=1,\ldots,n$ that 
\begin{equation} \label{GrindEQ__7_106_} 
\chi (v_{i} (t))=\left(v_{\max } -v_{i} (t)\right)\psi (v_{i} (t))\ge \Delta _{2} {\mathop{\min }\limits_{v\in [\Delta _{1} ,v_{\max } -\Delta _{2} ]}} \psi (v)=\underline{\chi }>0,t\ge 0 .
\end{equation} 
Define for $(s,v)\in D$
\begin{equation} \label{GrindEQ__7_107_} 
V(s,v)=\sum _{i=1}^{n}\lambda ^{i-1} \left(\frac{1}{2} H_{i}^{2} +\frac{\phi}{2} p_{i}^{2} \right) , 
\end{equation} 
where $\lambda \in (0,1)$, $\phi>0$, $H_{i} =s_{i} -r_i- {h_{i}}v_{i} $, and $p_{i} =v_{i} -v^{*} $. Using \eqref{GrindEQ__2_1_}, \eqref{GrindEQ__3_1_}, and the same analysis as in the proof Theorem \ref{thm:4}, we obtain 
\begin{equation} \label{GrindEQ__7_108_} 
\dot{V}(s,v)\le -\sum _{i=1}^{n}\lambda ^{i-1} \left(\frac{h_{i} \beta _{i} \chi (v_{i} )}{2 } H_{i}{\sigma(H_i)} +\frac{\phi m_{\lambda } }{2  } p_{i}^{2} \right)  ,
\end{equation} 
for certain constants $\phi,\lambda ,m_{\lambda } >0$ (see \eqref{GrindEQ__7_80_}). 
{Due to assumption $G_i<\bar{G}$, it follows by definition \eqref{GrindEQ__3_2_} and \eqref{GrindEQ__3_5_}, that $|H_i|\leq R_{\bar{G},i}:= \max\{ r_i -\ell_i, \bar{G}-(r_i-\ell_i)\}$. Hence, from condition \eqref{s4} with $R=R_{\bar{G},i}$, it follows that there exists a constant $m_{\bar{G},i}>0$ such that $H_i\sigma(H_i)\ge m_{\bar{G},i} H_i^2$. Combining the latter with \eqref{GrindEQ__7_108_}, \eqref{GrindEQ__7_106_}, and definition $V(t):=V(s(t),v(t))$, $t\ge 0$, we get 
\begin{equation} \label{GrindEQ__7_109_} 
\dot{V}(t)\le -\eta \sum _{i=1}^{n}\lambda ^{i-1} \left(H_{i}^{2} (t)+p_{i}^{2} (t)\right) , t\ge 0 ,
\end{equation} 
where $\eta ={\mathop{\min }\limits_{i=1,\ldots,n}} \left\{\frac{h_{i} \beta _{i} \underline{\chi }m_{\bar{G},i}}{2  } ,\frac{\phi m_{\lambda } }{2 } \right\}$}. Notice also that by definition of $V$ in \eqref{GrindEQ__7_107_} we have that
\begin{equation} \label{GrindEQ__7_110_} 
m\sum _{i=1}^{n}\lambda ^{i-1} \left(H_{i}^{2} +p_{i}^{2} \right) \le V(s,v)\le M\sum _{i=1}^{n}\lambda ^{i-1} \left(H_{i}^{2} +p_{i}^{2} \right)  ,
\end{equation} 
where $m=\frac{1}{2} \min \left(1,\phi\right)$ and $M=\frac{1}{2} \max \left(1,\phi\right)$. Thus, it follows from \eqref{GrindEQ__7_109_} and \eqref{GrindEQ__7_110_} that
\begin{equation} \label{GrindEQ__7_111_} 
\dot{V}(t)\le -\mu V(t), t\ge 0 ,
\end{equation} 
where $\mu :=\eta /M$. Inequality \eqref{GrindEQ__7_111_} implies that
\begin{equation}\label{GrindEQ__7_112_}
V(t)\le \exp \left(-\mu t\right)V(0)\textrm{ for }t\ge 0,
\end{equation}
 or equivalently, due to \eqref{GrindEQ__7_110_},
\begin{equation} \label{GrindEQ__7_113_} 
\sum _{i=1}^{n}\lambda ^{i-1} \left(H_{i}^{2} (t)+p_{i}^{2} (t)\right) \le \frac{M}{m} \exp \left(-\mu t\right)\sum _{i=1}^{n}\lambda ^{i-1} \left(H_{i}^{2} (0)+p_{i}^{2} (0)\right) , t\ge 0 .
\end{equation} 

Since $\lambda \in (0,1)$ it follows that
\begin{equation} \label{GrindEQ__7_114_} 
\begin{aligned} &{\lambda ^{n-1} \sum _{i=1}^{n}\left(H_{i}^{2} +p_{i}^{2} \right) \le \sum _{i=1}^{n}\lambda ^{i-1} \left(H_{i}^{2} +p_{i}^{2} \right) \le \sum _{i=1}^{n}\left(H_{i}^{2} +p_{i}^{2} \right) } \\ &{\Rightarrow \sum _{i=1}^{n}\left(H_{i}^{2} +p_{i}^{2} \right) \le \lambda ^{-(n-1)} \sum _{i=1}^{n}\lambda ^{i-1} \left(H_{i}^{2} +p_{i}^{2} \right) }. \end{aligned} 
\end{equation} 
Thus, combining \eqref{GrindEQ__7_114_} and \eqref{GrindEQ__7_113_} we obtain for $t\ge 0$
\begin{equation} \label{GrindEQ__7_115_} 
\begin{aligned} &{\sum _{i=1}^{n}\left(H_{i}^{2} (t)+p_{i}^{2} (t)\right) \le \lambda ^{-(n-1)} \frac{M}{m} \exp \left(-\mu t\right)\sum _{i=1}^{n}\lambda ^{i-1} \left(H_{i}^{2} (0)+p_{i}^{2} (0)\right) } \\ &{\le \lambda ^{-(n-1)} \frac{M}{m} \exp \left(-\mu t\right)\sum _{i=1}^{n}\left(H_{i}^{2} (0)+p_{i}^{2} (0)\right) } \end{aligned} 
\end{equation} 
Define 
\begin{equation} \label{GrindEQ__7_116_} 
e_{i} =s_{i} -s_{i}^{*} , \; \;i=1,\ldots,n ,
\end{equation} 
and
\begin{equation} \label{GrindEQ__7_117_} 
e:=\left(e_{1} ,\ldots,e_{n} \right)' ,H:=\left(H_{1} ,\ldots,H_{n} \right)' ,\, \, \, p:=\left(p_{1} ,\ldots,p_{n} \right)'  .
\end{equation} 
Since $H_{i} =s_{i} -r_{i} -  {h_{i}} v_{i} $ and $p_{i} =v_{i} -v^{*} $, \eqref{GrindEQ__7_116_} can be written also as
\begin{equation} \label{GrindEQ__7_118_} 
e_{i} =H_{i} +  {h_{i}} p_{i}  .
\end{equation} 
Then \eqref{GrindEQ__7_115_} is written as
\begin{equation} \label{GrindEQ__7_119_} 
\left|(H(t),p(t))\right|\le R\exp \left(-\mu t/2\right)\left|\left(H(0),p(0)\right)\right|, \;\; t\ge 0 ,
\end{equation} 
where $R=\sqrt{\lambda ^{-(n-1)} M/m} $ and $(H,p)'\in {\mathbb R}^{2n} $. Moreover, by \eqref{GrindEQ__7_118_} we have that
\begin{equation} \label{GrindEQ__7_120_} 
\left(\begin{array}{c} {e} \\ {p} \end{array}\right)=\left(\begin{array}{cc} {I_{n} } & {\diag\left\{ {h_1} ,\ldots,h_n \right\}} \\ {0} & {I_{n} } \end{array}\right)\left(\begin{array}{c} {H} \\ {p} \end{array}\right) .
\end{equation} 
Since the above matrix is invertible, there exist constants $c_{1} ,c_{2} >0$ such that 
\begin{equation} \label{GrindEQ__7_121_} 
c_{1} \left|(H,p)\right|\le \left|(e,p)\right|\le c_{2} \left|(H,p)\right|. 
\end{equation} 
Hence, combining \eqref{GrindEQ__7_119_} and \eqref{GrindEQ__7_121_} we get for $t\ge 0$
\[\begin{aligned}   \left|(e(t),p(t))\right|&\le c_{2} \left|(H(t),p(t))\right|\le c_{2} R\exp \left(-\mu t/2\right)\left|\left(H(0),p(0)\right)\right|  {\le \frac{c_{2} }{c_{1} } R\exp \left(-\mu t/2\right)\left|\left(e(0),p(0)\right)\right|} \end{aligned}\] 
The latter and definition of \eqref{GrindEQ__7_117_}, \eqref{GrindEQ__7_116_} and the fact that $p_{i} =v_{i} -v^{*} $, give estimate \eqref{GrindEQ__5_10_} with $\tilde{C}=\frac{c_{2} }{c_{1} } R$ and $\tilde{\mu }=\mu /2$. The proof is complete.
\end{proof}

\section{Conclusions}\label{sec:conc}

In this paper, we investigated safety, stability, and string-stability properties of vehicle platoons operating under the CTH policy. After revisiting the safety and invariance properties of the classical linear CTH controller and identifying its limitations, we proposed a nonlinear modification. The resulting controller guarantees collision avoidance, preservation of admissible speeds bounds, and convergence of the spacing-policy errors. A key feature of the proposed design is the explicit characterization of the spacing-policy error dynamics, which enables a direct nonlinear analysis of the closed-loop system. Furthermore, sufficient conditions for $L^p$ string stability, $1\le p\le\infty$, were established, and the stability properties of the equilibria were analyzed.  The properties of the nonlinear CTH-based controller were illustrated in numerical simulations, including performance comparisons with the linear CTH controller and the controller from \cite{Karafyllis2023}, as well as practically realistic scenarios involving the NGSIM and OpenACC datasets.

Future research will focus on the extension of the proposed framework to higher-order and nonlinear vehicle dynamics models, and the development of corresponding safety, stability, and string-stability results. Additional directions include the investigation of robustness of the nonlinear CTH-based controller with respect to sensing  and actuation imperfections, and, in particular, on the analysis and compensation of the effects of input delays on the closed-loop platoon dynamics (see e.g. \cite{Bekiaris}, \cite{Haan}, \cite{Orosz}).

\printcredits


\begin{thebibliography}{99}
\bibitem{Akcelik} 	R. Akcelik and D. C. Biggs, ``Acceleration Profile Models for Vehicles
 	in Road Traffic'', {\em Transportation Science},  21, 36--54, 1987.	

 
\bibitem{Alan}
A. Alan, A. J. Taylor, C. R. He, A. D. Ames and G. Orosz, ``Control Barrier Functions and Input-to-State Safety With Application to Automated Vehicles'', \emph{IEEE Transactions on Control Systems Technology}, 31, 2744--2759, 2023.
 
 \bibitem{Ames2017}
A. D. Ames, X. Xu, J. W. Grizzle, and P. Tabuada,
``Control Barrier Function Based Quadratic Programs for Safety Critical Systems'',
\emph{IEEE Transactions on Automatic Control},  62,  3861--3876, 2017.


\bibitem{Ampountolas}
K. Ampountolas, ``The Unscented Kalman Filter for Nonlinear Parameter Identification of Adaptive Cruise Control Systems'', \emph{IEEE Transactions on Intelligent Vehicles},  8,  4094--4104, 2023.

\bibitem{Aubin}
J.-P. Aubin, \emph{Viability Theory}, Boston: Birkhauser, 1991.

\bibitem{Bek} N. Bekiaris-Liberis, C. Roncoli and M. Papageorgiou, ``Predictor-Based Adaptive Cruise Control Design'', \emph{IEEE Transactions on Intelligent Transportation Systems},  19, 3181--3195,  2018.

\bibitem{Bekiaris1}
N. Bekiaris-Liberis, ``Robust String Stability and Safety of CTH Predictor-Feedback CACC'', \emph{IEEE Transactions on Intelligent Transportation Systems}, 24,  8209--8221, 2023.

\bibitem{Bekiaris}
N. Bekiaris-Liberis, ``Nonlinear Predictor-Feedback Cooperative Adaptive Cruise Control of Vehicles with Nonlinear Dynamics and Input Delay'', \emph{International Journal of Robust and Nonlinear Control}, 34, 6683-6698, 2024.

\bibitem{Berger}
T. Berger and B. Besselink, ``String Stability and Guaranteed Safety via Funnel Cruise Control for Vehicle Platoons'', \emph{IEEE Transactions on Automatic Control}, 71, 81-90,   2026.
 
\bibitem{Besselink2017}
B. Besselink and K. H. Johansson,
``String Stability and a Delay-Based Spacing Policy for Vehicle Platoons Subject to Disturbances'',
\emph{IEEE Transactions on Automatic Control},  62,  4376--4391, 2017.

\bibitem{Brezis}
H. Br\'{e}zis, \emph{Functional Analysis, Sobolev Spaces and Partial Differential Equations}, Springer, 2011.
 

\bibitem{Feng2019}
S. Feng, Y. Zhang, S. E. Li, Z. Cao, H. X. Liu, and L. Li,
``String Stability for Vehicular Platoon Control: Definitions and Analysis Methods'',
\emph{Annual Reviews in Control},  47,  81--97, 2019.

\bibitem{Haan}
R. de Haan, T. P. J. van der Sande, E. Lefeber and I. J. M. Besselink, ``Cooperative Adaptive Cruise Control for Heterogeneous Platoons With Delays: Controller Design and Experiments'', \emph{IEEE Transactions on Control Systems Technology}, 33, 1361--1371, 2025.

\bibitem{Hamdipoor}
V. Hamdipoor, N. Meskin, C. G. Cassandras, ``Safe Control Synthesis Using Environmentally Robust Control Barrier Functions'', \emph{European Journal of Control},
 74, 100840, 2023.
  
\bibitem{He2018}
C. R. He and G. Orosz,
``Safety Guaranteed Connected Cruise Control'',
 \emph{Proceedings of the 21st International Conference on Intelligent Transportation Systems (ITSC)}, Maui, HI,  549--554, 2018.
 
\bibitem{Ioannou1993}
P. A. Ioannou and C. C. Chien,
``Autonomous Intelligent Cruise Control'',
\emph{IEEE Transactions on Vehicular Technology},  42,  657--672, 1993.
 
\bibitem{Karafyllis2023}
I. Karafyllis, D. Theodosis, M. Papageorgiou, ``Nonlinear Adaptive Cruise Control of
Vehicular Platoons'' \emph{International Journal of Control}, 96, 147--169, 2023.

\bibitem{Karafyllis2025}
I. Karafyllis, D. Theodosis, and M. Papageorgiou, ``Nonlinear Cruise Controllers with
Bidirectional Sensing for a String of Vehicles'', \emph{European Journal of Control}, 85, 101341, 2025.

\bibitem{Karafyllis2026}
 I. Karafyllis, D. Theodosis, and M. Papageorgiou, ``Sufficient Conditions for String Stability'', \emph{International Journal of Systems Science}, 1–18, https://doi.org/10.1080/00207721.2026.2641208.

\bibitem{Khalil2002}
H. K. Khalil,
\emph{Nonlinear Systems},
Prentice Hall, 2002.
 
 
\bibitem{Lunze2019}
J. Lunze,
``Adaptive Cruise Control with Guaranteed Collision Avoidance'',
\emph{IEEE Transactions on Intelligent Transportation Systems},  20,  1897--1907, 2019.

 \bibitem{OpenACC}
 M. Makridis, K. Mattas, A. Anesiadou, and B. Ciuffo, ``OpenACC. An Open Database of Car-Following Experiments to Study the Properties of Commercial ACC Systems'', \emph{Transportation Research Part C: Emerging Technologies}, 125, 103047, 2021.
 
 	\bibitem{Martinez}
 {	J. J. Martinez and C. Canudas-de-Wit, ``A Safe Longitudinal Control for Adaptive Cruise Control and Stop-and-Go Scenarios'', {\em IEEE Transactions on Control Systems Technology},  15,   246--258, 2007.}
 


\bibitem{Molnar}
T. G. Molnar, G. Orosz and A. D. Ames, ``On the Safety of Connected Cruise Control: Analysis and Synthesis with Control Barrier Functions'', \emph{2023 62nd IEEE Conference on Decision and Control (CDC)}, Singapore, 1106--1111, 2023.

\bibitem{Montanino}
M. Montanino, J. Monteil, V. Punzo, ``From Homogeneous to Heterogeneous Traffic Flows: $L_p$ String Stability Under Uncertain Model Parameters'', \emph{Transportation Research Part B: Methodological}, 146, 136-154, 2021.

\bibitem{Montanino2}
M. Montanino, V. Punzo, ``Trajectory Data Reconstruction and Simulation-Based Validation Against Macroscopic Traffic Patterns'', \emph{Transportation Research Part B: Methodological}, 80, 82--106, 2015.

\bibitem{Monteil2017}
J. Monteil, M. Bouroche and D. J. Leith, ``$L_2$ and $L_{\infty}$ Stability Analysis of Heterogeneous Traffic With Application to Parameter Optimization for the Control of Automated Vehicles'',  \emph{IEEE Transactions on Control Systems Technology}, 27, 934-949, 2019.

 

\bibitem{Orosz}
G. Orosz  and  T. G. Moln\'{a}r, \emph{Dynamics and Control of Connected Vehicles}, Surveys and Tutorials in
the Applied Mathematical Sciences,  Springer, Cham, 2025.

 

\bibitem{Ploeg2014}
J. Ploeg, N. van de Wouw, and H. Nijmeijer,
``String Stability of Cascaded Systems: Application to Vehicle Platooning'',
\emph{IEEE Transactions on Control Systems Technology},  22,  786--793, 2014.

\bibitem{Rajamani2012}
R. Rajamani,
\emph{Vehicle Dynamics and Control},
Springer-Verlag, New York, NY, USA, 2012.
 
\bibitem{Santhanakrishnan2003}
K. Santhanakrishnan and R. Rajamani,
``On Spacing Policies for Highway Vehicle Automation'',
\emph{IEEE Transactions on Intelligent Transportation Systems},  4,  198--204, 2003.
 
\bibitem{Shen}
 J. Shen, L. Du, ``Sequential Feasibility and Constraint Properties of CAV Platoons Under Various Vehicle Dynamics and Safety Distance Constraints'', \emph{Transportation Research Part B: Methodological}, 185, 102966, 2024.
  
\bibitem{Swaroop1996}
D. Swaroop and J. K. Hedrick,
``String Stability of Interconnected Systems'',
\emph{IEEE Transactions on Automatic Control},  41,  349--357, 1996.
 
 

\bibitem{Verginis2018}
C. K. Verginis, C. P. Bechlioulis, D. V. Dimarogonas, and K. J. Kyriakopoulos,
``Robust Distributed Control Protocols for Large Vehicular Platoons with Prescribed Transient and Steady-State Performance'',
\emph{IEEE Transactions on Control Systems Technology},  26,  299--304, 2018.
 
\bibitem{Wang}
J. Wang, S. Gong, S. Peeta, L. Lu, ``A Real-Time Deployable Model Predictive Control-Based Cooperative Platooning Approach for Connected and Autonomous Vehicles'', \emph{Transportation Research Part B: Methodological}, 128, 271-301, 2019.
 
\bibitem{MWang}
  {	M. Wang, S. P. Hoogendoorn, W. Daamen, B. van Arem, and R. Happee, ``Game Theoretic Approach for Predictive Lane-Changing And Car-Following Control", {\em Transportation Research Part C: Emerging Technologies},  58,  73--92, 2015.}
  	
\bibitem{Wijnbergen2020}
P. Wijnbergen and B. Besselink,
``Existence of Decentralized Controllers for Vehicle Platoons: On the Role of Spacing Policies and Available Measurements'',
\emph{Systems \& Control Letters},  145, p. 104796, 2020.

\bibitem{Xiao}
L. Xiao and F. Gao, ``Practical String Stability of Platoon of Adaptive Cruise Control Vehicles'', \emph{IEEE Transactions on Intelligent Transportation Systems}, 12, 1184-1194,  2011.
  
\bibitem{Yang}
L. Yang, Z. Sun, Y. Liu, L. Chen, ``Robust Control for Connected Automated Vehicle Platoon With Multiple-Predecessor Following Topology Considering Communication Loss'', \emph{Transportation Research Part B: Methodological}, 196, 103212,  2025.
  
\bibitem{Zheng2016}
Y. Zheng, S. E. Li, J. Wang, D. Cao, and K. Li,
``Stability and Scalability of Homogeneous Vehicular Platoon: Study on the Influence of Information Flow Topologies'',
\emph{IEEE Transactions on Intelligent Transportation Systems},  17,  14--26, 2016.
 
 
\bibitem{Zhao}
C. Zhao, T. G. Molnar, H. Yu, ``Leveraging Cooperative Connected Automated Vehicles for Mixed Traffic Safety'', \emph{Transportation Research Part B: Methodological}, 203, 103352, 2026.

\bibitem{Zhou}
Y. Zhou, S. Ahn, ``Robust Local and String Stability for a Decentralized Car Following Control Strategy for Connected Automated Vehicles'', \emph{Transportation Research Part B: Methodological}, 125, 175-196, 2019.

\end{thebibliography}
\end{document}